\documentclass[12pt,  oneside
]{amsart}

\usepackage{amscd,amsmath, 
graphicx, wasysym}

\usepackage{amsrefs}

\usepackage[margin=1in]{geometry}

\usepackage{amsfonts}
\DeclareMathSymbol{\blacktriangle}{\mathord}{AMSa}{"4E}
\DeclareMathSymbol{\nsubseteq}{\mathrel}{AMSb}{"2A}
\DeclareMathSymbol{\lesssim}{\mathord}{AMSa}{"2E}
\DeclareMathSymbol{\gtrsim}{\mathord}{AMSa}{"26}

\title[]{A Frobenius-Nirenberg theorem with parameter}
\author[]{Xianghong Gong}

 \address{Department of Mathematics,
 University of Wisconsin - Madison, Madison, WI 53706, U.S.A.}
 \email{gong@math.wisc.edu}


\oddsidemargin=\evensidemargin

\newcommand{\dist}{\operatorname{dist}}
\newcommand{\DD}[2]{\frac{\partial #1}{\partial #2}}

\newtheorem{thm}{Theorem}[section]
\newtheorem{cor}[thm]{Corollary}
\newtheorem{prop}[thm]{Proposition}
\newtheorem{lemma}[thm]{Lemma}

\theoremstyle{definition}

\newtheorem{defn}[thm]{Definition}
\newtheorem{exmp}[thm]{Example}

\newtheorem{rem}[thm]{Remark}

\renewcommand{\th}[1]{\begin{thm}\label{#1}}
\newcommand{\eth}{\end{thm}}
\newcommand{\co}[1]{\begin{cor}\label{#1}}
\newcommand{\eco}{\end{cor}}
\renewcommand{\le}[1]{\begin{lemma}\label{#1}}
\newcommand{\ele}{\end{lemma}}
\newcommand{\pr}[1]{\begin{prop}\label{#1}}
\newcommand{\epr}{\end{prop}}

\newcommand{\ps}{the cone property}

\newcommand{\ga}{\begin{gather}}
\newcommand{\ega}{\end{gather}}
\newcommand{\gan}{\begin{gather*}}
\newcommand{\egan}{\end{gather*}}
\newcommand{\al}{\begin{align}}
\newcommand{\eal}{\end{align}}
\newcommand{\aln}{\begin{align*}}
\newcommand{\ealn}{\end{align*}}
\newcommand{\eq}[1]{\begin{equation}\label{#1}}
\newcommand{\eeq}{\end{equation}}

\newcommand{\A}{\mathcal A}
\newcommand{\B}{\mathcal B}

\newcommand{\ci}{~\cite}

\newcommand{\f}[2]{\frac{#1}{#2}}

\newcommand{\mfl}[1]{\lfloor #1\rfloor}
\newcommand{\ple}{\,\lesssim\,}

\newcommand{\cc}{{\bf C}}
\newcommand{\nn}{{\bf N}}

\newcommand{\rr}{{\bf R}}

\newcommand{\ov}{\overline}

\newcommand{\RE}{\operatorname{Re}}
\newcommand{\IM}{\operatorname{Im}}

\newcommand{\rank}{\operatorname{rank}}

\newcommand{\cL}{\mathcal}

\newcommand{\all}{\alpha}

\newcommand{\gaa}{\gamma}

\newcommand{\del}{\delta}

\newcommand{\var}{\varphi}
\newcommand{\e}{\epsilon}
\newcommand{\om}{\omega}
\newcommand{\Om}{\Omega}

\newcommand{\la}{\lambda}

\newcommand{\pd}{\partial}

\newcommand{\re}[1]{(\ref{#1})}
\newcommand{\rea}[1]{$(\ref{#1})$}
\newcommand{\rl}[1]{Lemma~\ref{#1}}

\newcommand{\rp}[1]{Proposition~\ref{#1}}
\newcommand{\rt}[1]{Theorem~\ref{#1}}
\newcommand{\rd}[1]{Definition~\ref{#1}}
\newcommand{\rrem}[1]{Remark~\ref{#1}}

\newcommand{\rla}[1]{Lemma~$\ref{#1}$}

\newcommand{\rpa}[1]{Proposition~$\ref{#1}$}
\newcommand{\rta}[1]{Theorem~$\ref{#1}$}

\newcommand{\db}{\overline\partial}

\newcounter{pp}
\newcommand{\bpp}{\begin{list}{$\hspace{-1em}\alph{pp})$}{\usecounter{pp}}}
\newcommand{\epp}{\end{list}}

\newcounter{ppp}
\newcommand{\bppp}{\begin{list}{$\hspace{-1em}(\roman{ppp})$}{\usecounter{ppp}}}
\newcommand{\eppp}{\end{list}}

\begin{document}
\begin{abstract}
The Newlander-Nirenberg theorem says that a formally integrable complex structure
is locally equivalent to the   complex structure in the complex Euclidean space.
We will show two results about the Newlander-Nirenberg theorem
with parameter. The first   extends  the Newlander-Nirenberg theorem to a parametric version,
and its proof  yields a sharp regularity result as Webster's proof for  the Newlander-Nirenberg theorem.
 The second  concerns a   version of
 Nirenberg's complex Frobenius theorem and its proof yields a result with a mild loss of regularity.
\end{abstract}


 \maketitle

\tableofcontents

\setcounter{section}{0}
\setcounter{thm}{0}\setcounter{equation}{0}
\section{Introduction}\label{sec1}
Let $D$ be a domain  in $\rr^N$ and let $S$ be a complex subbundle of   $T(D)\otimes \cc$.
  The rank of
 $S$ over $\cc$ is denoted by $\rank_\cc S$. Assume that $S$ is of class $\cL C^r$.
We say
that $S$ is {\it formally integrable} if $X,Y$ are $\mathcal C^1$ sections of $S$ then  the Lie bracket $[X,Y]$
remains a
section of $S$;
it   is {\it CR} 
  if 
  $S+\ov {S}$ is  a complex bundle; and it is {\it Levi-flat} if both $S$ and  $S+\ov S$ are formally integrable.
When $S$ is a CR vector bundle, we  call $\rank_\cc(S+  \ov S)-\rank_\cc S$ the {\it CR dimension} of $S$. When $S$ has CR dimension $0$ additionally, i.e. $\ov S=S$, $S$ is the standard (real) Frobenius structure.
 We first formulate a finite smoothness result analogous to  a theorem of Nirenberg~\ci{Ni57} for the $\cL C^\infty$ case.
\pr{th0}
 Let $S$ be a Levi-flat CR vector bundle of class $\cL C^r$ with $r\in(1,\infty)$ in a domain $D\subset\rr^N$.
For each $p\in D$
 there exist a neighborhood $U$ of $p$ and a diffeomorphism  $F$ in $\cL C^a(U)$ for all $a<r$
    such that $F_*S$  is spanned by
\eq{stcfs}
\pd_{\ov z_1},\dots,
\pd_{\ov z_n}, \ \pd_{t_{1}},\dots\pd_{t_{M}},
\eeq
where  $n$ is the CR dimension of $S$, and $M+n$ equals $\rank_\cc S$, the rank of $S$.
\epr

Throughout the paper, $y:=(x,t,\xi)$ are the coordinates of $\rr^N=\rr^{2n}\times\rr^M\times\rr^L$, while $z$    with $z_j=x_j+ix_{n+j}$ are the complex coordinates of $\rr^{2n}$.  We have chosen coordinates such that for the standard Levi-flat structure \re{stcfs}, the $z$ space is     the holomorphic leave space, $(\RE z,\IM z,t)$ space is    the total leave space, and the $\xi$ space is transversal to the total leave space.

Let us discuss special cases of \rp{th0}.
When the $S$ defines a  real    structure  additionally, i.e., $S=\ov S$,  the result is the classical Frobenius theorem~\ci{Fr77} (see also Hawkins~\ci{Ha13}) and it actually holds  for $F\in \cL C^r$ (see Guggenheimer~\ci{Gu62}, Narasimhan~\ci{Na73}, F. and R. Nevanlinna~\ci{NN73}, and \rp{rfp} below).
When $S$ defines a complex structure, i.e. its CR dimension is $\f{N}{2}$, the result with a better regularity is the well-known  Newlander-Nirenberg theorem ~\ci{NN57};   it is sharp  with
 $F\in \cL C^{r+1}$  when $r\in(1,\infty]\setminus\nn$,
 which is a theorem of Webster~\ci{We89}.

We now formulate results in a parametric version.  We  first define   H\"{o}lder spaces with parameter.
Let  $\nn$ be the set of non-negative integers and let  $I=[0,1]$. For $r\in[0,\infty)$, let $[r]$ denote  the largest integer $\leq r$ and set $\{r\}=r-[r]$.
 When  $D$ is bounded and $0\leq s\leq r\leq\infty$,
 by an element  in
 $\cL C^{r,s}(\ov D )$   we   mean
 a family $\{u^\la;\la\in I\}$ of functions $u^\la$ in $\ov D$  having the property that if $i,j\in\nn$, $i+j\leq r$ with $j\leq s$, then
$\pd_y^i\pd_\la^ju^\la(y)$ are continuous
in $\ov D \times I$ and have bounded $\cL C^{\{r\}}(\ov D)$-norms ($r<\infty$) as $\la$ varies in $I$ and
  bounded $\cL C^{\{ s\}
}(I)$-norms ($s<\infty$) as $y$ vary in $\ov D$.
We say that a family $\{S^\la\}$ of CR vector bundles in $D$  is of class $\cL C^{r,s}$,  if    each vector $v\in S^{\la_0}_{p_0}$ extends to a vector field
\ga\nonumber
 a_{1}^\la(y)\pd_{y_1}+\dots+ a_{N}^\la(y)\pd_{y_N}, \quad \{a_{j}^\la\}\in \cL C^{r,s}(\ov D).
\end{gather}
Set $\cL C^{r_-,s}(  U )=\bigcap_{a<r}\cL C^{a,s}(  U )$ and  $\cL C^{r_-,s_-}(  U )=\bigcap_{a<r,b<s}\cL C^{a,b}(  U )$.

Let us first state a result for a family of complex structures.
\pr{pnl} Let $s\in\nn$.
 Let $D$ be a domain in $\rr^{2n}$.
Let $\{S^\la\}$ be a family of
complex structures in $D$ of class $\cL C^{r,s}$ with $r>1$.
For each $p\in D$
 there exist a neighborhood $U$ of $p$ and  a family $\{F^\la\}\in\cL C^{1,0}( U )$  of diffeomorphisms $F^\la$  such that   each $F^\la_*S^\la$ is spanned by $\pd_{\ov{z_1}}, \dots, \pd_{\ov z_n}$.  If $r\not\in\nn$, the
  $\{F^\la\}$ is in $\cL C^{r+1,s}(  U )$ when   $r> s+1$.
If $r,s$ are in $\nn$ and     $ r\geq s+3$, then $\{F^\la\}\in\cL C^{(r+1)_-,s}(  U )$. \epr
The above result for $n=1$, $r\in(0,\infty)\setminus\nn$ and $s\in\nn$   is due to   Bertrand-Gong-Rosay~\ci{BGR14}.  
The main result of this paper is the following.
\th{thm1}
 Let $D$ be a domain in $\rr^N$ with $N\geq2$.
Let $\{S^\la\}$ be a family of
Levi-flat CR vector bundles $S^\la$ in $D$ of class $\cL C^{r,s}$ with $r>1$. Assume that  $S^\la$ have constant CR dimension $n$ and   rank $M+n$.
For each $p\in D$
 there exist a neighborhood $U$ of $p$ and  a family $\{F^\la \}\in\cL C^{1,0}( U )$  of diffeomorphisms $F^\la $   so that  each $F^ \la_*S^\la$ is spanned by \rea{stcfs}. Moreover, the $\{F^\la \}$ satisfies the   following   properties\,$:$
\bppp
\item $\{F^\la \}\in\cL C^{\infty,s}(  U )$,  provided $r=\infty$, and $s\in\{0\}\cup [1,\infty]$.
\item
 $\{F^\la \}\in\cL C^{r_-,s}(  U )$, provided
$(${\it a}$)$ $  r>s+2\in\nn$, or $(${\it b}$)$ $s=0$ and $  r>1$.
  \item
 $\{F^\la \}\in\cL C^{r_-,s_-}(  U )$, provided $(a)$ $ r\geq s+3\geq4$ and $ \{r\}\geq\{s\}>0$,  or $(b)$  $r=s>1$.   \eppp
\eth
Note that \rp{th0} is a special case of ($ii$) with $s=0$.

%
%

 We will also study an {\it  analytic} family of Frobenius structures.
 We will show that if $a\in(1,\infty)\setminus\nn$,  the real analytic family $\{P^\la\}$ of $\cL C^a$ vector fields in $\rr^2$, defined by
\eq{Plambda}
P^\la =\pd_{x_1}+(|x_2|^a+\la)\pd_{x_2},
\eeq
cannot be transformed into the span of $\pd_{x_1}$ by a family $\{F^\la \}$ of diffeomorphisms with $\pd_\la F^\la \in\cL C^a$.       In  contrast, Nirenberg~\ci{Ni57} showed that an analytic family of $\cL C^\infty$ {\it complex} structures can be transformed into the standard complex structure by an  {\it analytic} family of $\cL C^\infty$  diffeomorphisms. In~\ci{Ni57} Nirenberg raised the question if
 an {analytic}   family of $\cL C^\infty$ complex Frobenius structures can be transformed into the standard complex Frobenius structures by a real analytic family of $\cL C^\infty$ diffeomorphisms. The above example gives a negative answer to a finite smooth version of the analogous question; see \rp{countex} for the proof.
However, we do not know a counter-example to Nirenberg's original question.

  The loss of derivatives in an arbitrarily small H\"{o}lder exponent has also occurred in the regularity result by  Gong-Webster~\ci{GW12}
  for the local CR embedding problem, which improves regularity results of Webster~\ci{We89} and Ma-Michel~\ci{MM94}.  However,
  it is
 unknown if such a loss  is necessary in ~\ci{GW12} and our results.
\rt{thm1} for the $\cL C^{\infty,\infty}$ case
is due to Nirenberg~\ci{Ni57}.
 See also Tr{e}ves's book~\ci{Tr92}*{pp. 296-297} for another proof when $s\in\nn$ and $r=\infty$. Treves uses a method in ~\ci{Ni57} combining with the method in Webster~\ci{We89}. We will also use
these  methods  to prove \rp{nithm},  a weak form of \rt{thm1}.

We now  mention
  several proofs of the Newlander-Nirenberg theorem.
  Nijenhuis and Woolf~\ci{NW63} reduced the required smoothness in~\ci{NN57}.
   Kohn~\ci{Ko63} gave a proof for the $\cL C^\infty$ case based on a solution  to
 the $\ov\pd$-Neumann problem. H\"{o}rmander~\ci{Ho65}  studied the $L^2$ regularity of solutions of first-order differential operators defined by the complex Frobenius
structures that satisfy  a convexity condition.
 Malgrange~\ci{Ma69} proved the Newlander-Nirenberg theorem
  by a reduction to partial differential equations with real analytic coefficients,  for which the
  quasi-linear elliptic theory applies.
 Newlander-Nirenberg~\ci{NN57} and  Nijenhuis-Woolf~\ci{NW63}  proved a parametric  version of Newlander-Nirenberg theorem.   Hill-Taylor~\cites{HT03, HT07}  proved the Newlander-Nirenberg type theorem for complex Frobenius structures of less than $\mathcal C^1$ smoothness.
We note that  Nirenberg's complex Frobenius theorem   has an application  in Kodaira-Spencer~\ci{KS61}. Furthermore, associated to the vector fields \re{stcfs} there is a natural   complex differential $\cL D:=\db+d_t$, which acts on exterior differential forms in $\cc^n\times\rr^M$ and satisfies $\cL D^2=0$.
Hanges and Jocobowitz~\ci{HJ97} proved the interior $\cL C^\infty$ regularity of $\cL D$ equations for smoothly bounded domains in $\cc^n\times\rr^M$.

Our proofs are motivated by Webster's  methods~\cites{We89, We89a} of  homotopy
formulae.
The proof of the first part of \rp{pnl},  i.e. for the case of  $s\in\nn$ and $r\in(s+1,\infty)\setminus\nn$, relies   on two ingredients. The first is the gain of one derivative in the interior estimate of Koppelman-Leray homotopy formula and the second is a KAM rapid convergence. Both are adapted from the work of Webster~\ci{We89}. We also need improved iteration methods    in  Gong-Webster~\cites{GW10, GW12}.
  For the general complex Frobenius structures,  we use a homotopy formula due to Tr{e}ves~\ci{Tr92}, which is a combination of   the Poincar\'e lemma and
    the Bochner-Martinelli and Koppelman-Leray  formulae.
However, the estimates for this homotopy formula do not gain any derivative; see~\cite{Tr92} and section~\ref{sec5-}. In our proofs   the chain rule is indispensable. The chain rule requires us to take less derivatives in parameter and thus to use the spaces $\cL C^{r,s}$; see~\rrem{rsrem}.  The space $C_*^{r,s}$, defined in section~\ref{sec2}, is not suitable for our proof of \rt{thm1} and this is illustrated by the negative results, Propositions~\ref{countex} and \ref{cex},  on the real analytic family of vector fields \re{Plambda}.

The paper is organized as follows.
In section 2, we describe   H\"{o}lder spaces and give an initial
normalization for the complex Frobenius structures.  In section 3,
we describe   some sharp interpolation inequalities in  H\"{o}lder norms. For the exposition purpose,  the proofs of these inequalities are given in the appendix.
In section 4 we   adapt a proof of the Frobenius theorem in F.~and R.~ Nevanlinna~\ci{NN73}
for a parametric version and we  provide  examples to discuss the exact regularity of \rt{thm1}. In section 5, we recall the homotopy formula for $\cL D$ by  Tr{e}ves~\ci{Tr92} and  adapt Webster's interior estimates~\ci{We89} for the homotopy operators for forms depending on a parameter.  In section 6 we prove \rp{pnl}.  In section 7, we derive a differential complex arising from a complex Frobenius structure and apply Tr{e}ves's homotopy formula to describe a general rapid iteration procedure to be used in section~\ref{sec7}.   In section~\ref{sec7},  using a Nash-Moser smoothing technique, we prove the main result.

\setcounter{thm}{0}\setcounter{equation}{0}
\section{H\"{o}lder spaces and an initial normalization
}\label{sec2}

In this section, we  discuss properties of the H\"{o}lder spaces mentioned in the introduction.   The reader can find applications of  these spaces in Nijenhuis and Woolf~\ci{NW63}, Bertrand-Gong~\ci{BG14}, and Gong-Kim~\ci{GK14}.
We also obtain an initial normalization of complex Frobenius structures at a given point $p$.

  Recall that $z_\all=x_\all+ix_{n+\all}$, and that $y=(\RE z,\IM z,t,\xi)$ are coordinates of $\rr^N=\rr^{2n}\times\rr^M\times\rr^{L}$.
Let $k,j$ be non negative integers and let $\all,\beta\in[0,1)$.  Let $D$ be a bounded domain in $\rr^N$ and $I$ be a finite and closed
 interval in $\rr$.
The   $\cL C^{k+\all}(\ov D)$ is the standard H\"{o}lder space  with norm
$|\cdot|_{\ov D;k+\all}$.
For  a family  $u:=\{u^\la\}$ of functions $u^\la $ in $\ov D$, by $u 
\in{\cL C}_*^{k+\all, j+\beta}(\ov D)$  we mean that
$\pd_y^ m \pd_\la^iu^\la(y)$  are
  continuous in $\ov D\times I $  and have bounded $\cL C^{\all}$ norms on $\ov D$ as $\la$ varies in $I$  and  bounded $\cL C^{\beta}$ norms
 on $I$ as $y$ vary in $\ov D$, for all $ m \leq k$ and $i\leq j$.
  We will take $I=[0,1]$ unless specified otherwise.
  Throughout the paper $\pd^k$ denotes the set of partial derivatives of order $k$ in $y\in\rr^N$, and  ${\mathbf B}^M_\rho$ denotes the ball  in $\rr^M$ of radius $\rho$ centered   at the origin.

The norm on  $\cL C^{k+\all,j+\beta}_*(\ov D )$  is defined  by
 \ga
 \nonumber
|u|_{D;k+\all,j+\beta }=\max_{  0\leq i\leq j,\ell\leq k}
\{\sup_{\la\in I}|\{\pd^i\pd_\la^\ell u^\la\} |_{\ov D;\all},\ \sup_{y\in \ov D} |\{\pd_y^i\pd_ \bullet^Iu^ \bullet\} |_{ I;\beta}\}.
\end{gather}
Assuming further that   $k\geq j$,  we define
\ga
\cL C^{k+\all,j+\beta}(\ov D)= \bigcap_{i=0}^j{\cL C}_*^{k-i+\all, i+\beta}(\ov D),\nonumber
\\
\nonumber 
 \|u\|_{D;k+\all,j+\beta}=\max\{|u|_{D;k-i+\all, i+\beta}\colon 0\leq i\leq j\}.
  \end{gather}
When the domain $D$ is clear in the context, we   abbreviate
$|\cdot|_{D;r}$,  $|\cdot|_{D; r,s}$, $|\cdot|_{D;r,s}$,  $\|\cdot\|_{D;r,s}$ by $|\cdot|_{r}$, $|\cdot|_{r,s}$, and $\|\cdot\|_{r,s}$, respectively. Note that the spaces $\cL C_*^{a,b}, \cL C^{r,s}$ are the same as $\cL B_*^{a,b},\cL B^{r,s}$  in~\ci{BG14}.

To simplify notation,  let $\pd_{\ov\all}=\pd_{\ov z_\all}, \pd_{\all}=\pd_{z_\all}$,
$\pd'_ m =\pd_{t_{ m }}$, and $\pd_\ell''=\pd_{\xi_{\ell}}$.
Throughout the paper,
repeated lower indices are summed over. The Greek letters $\all,\beta,\gamma$ have the range $1,\dots, n$, and indices $\ell, \ell'$ have the range $1,\dots, L:=N-M-2n$ and the indices $m,m'$ have the range $1,\dots, M$.
\le{intialcod} Let $ r\geq \max\{1,s\}$.
Let $\{S^\la \}\in \cL C^{r,s}(\ov D)$ be a family of CR vector bundles in a domain $D$ of $\rr^N$ with constant rank $M+n$ and CR dimension $n$.
 For each $p\in D$,  there exists a family   of affine transformations $\var^\la $ of $\rr^N$, sending $p$ to the origin,
  so that
  the coefficients of $\var^\la$ are in $\cL C^s(I)$ in $\la$,  while the fiber at the origin of each $\var^\la _*S$ is spanned by $
\pd'_{ m }, \pd_{\ov \all}$ with $1\leq  m \leq M$, $1\leq\all\leq n$. Consequently,
 near the origin of $\rr^N$, $\var^\la _*S^\la $ has a unique  basis
\al \label{xiza1}
X^\la _{ m }&=\pd'_{ m }+ \RE (B^\la _{ m \beta} \pd_\beta)+b^\la _{ m  \ell}\pd_{\ell}'',  \qquad 1\leq  m \leq M,\\ Z^\la _{\ov\all}&=\pd_{\ov\all}+ A^\la _{\ov\all \beta}\pd_\beta+a^\la _{\ov\all \ell}\pd_{\ell}'', \ \ \  \qquad\qquad
  1\leq \all\leq n,\label{xiza2}
 \end{align}
where $a^\la _{\ov\all \ell}, A^\la _{\ov\all\beta}, B^\la _{ m \beta}$ are complex-valued,   $b^\la _{ m  \ell}$ are real-valued, and they
vanish at the origin.
\ele
\begin{proof}
 It is convenient to work with $K^\la $, the set of complex-valued $1$-forms in $\rr^N$ that
annihilate $S^\la $.  We are normalizing the linear part of $S^\la $ at a point $p$. We may assume that $p=0$ and for simplicity   $ K^\la ,S^\la $ denote the fibers of $K^\la ,S^\la $ at the origin.

In the following, all basis  are over $\cc$.  Note that $L=\dim_\cc K^\la \cap\ov{K^\la }$.
For $\la$ near a given value,   $K^\la \cap\ov{K^\la }$ has
a basis    of real $1$-forms  $\eta_1^\la ,\dots, \eta_L^\la $  of   which the   coefficients  are $\cL C^s$ functions in $\la$. We can also   find $ \om_{\all}^\la $,  $1\leq\all\leq n$,   so that   $\{\om_\all^\la ,\eta_ \ell ^\la \}$ is a basis of
$K^\la $.
Then   $\{\om_\all^\la ,\eta_ \ell ^\la , \ov\om_\all^\la \}$ is a basis of
$K^\la +\ov {K^\la }$.  Finally we complete them with real forms $\theta_1^\la ,\dots, \theta_M^\la $ so that
 $\{\om_\all^\la ,\eta_\ell ^\la , \ov\om_\all^\la ,\theta_m^\la \}$ is a basis of
$T_0^*\rr^N\otimes\cc$. All coefficients of the forms are  of class $\cL C^s$ for $\la$ near a given value.

We cover $I$ by intervals $I_1,\dots, I_k$ which are either open or of the forms $[0,a), (b,1]$ so that $I_j$ contains $[\la_{j-1},\la_j]$, with $\la_0=0$ and $\la_k=1$. Furthermore, in $I_j$ there are bases
$$\{\eta_{j\ell }^\la \},\quad \{\eta_{j \ell }^\la ,\om_{j\all}^\la \},\quad \{\eta_{j \ell }^\la ,\om_{j\all}^\la ,\ov{\om_{j\all}^\la }, \theta_{jm}^\la \}$$
for  $K^\la \cap \ov{K^\la }$, $K^\la $,  and $T_0^*\rr^N\otimes\cc$, respectively. 
 We also assume the coefficients of these forms are in $\cL C^s(I_j)$. Next, we match the bases at the end points of $[\la_{j-1},\la_j]$.  At $\la=\la_1$, we have
\gan
\eta^{\la_1}_{1 \ell }=A_{ \ell   \ell'}\eta^{\la_1}_{2\ell '}, \quad \om^{\la_1}_{1\all}=B_{ \all \beta}\om^{\la_1}_{2 \beta}+b_{ \all\ell }\eta^{\la_1}_{2\ell },\\
 \theta^{\la_1}_{1m}=c_{mm'}\theta^{\la_1}_{2m'}+\RE\{C_{m\all}\om_{2\all}^{\la_1}\}+d_{m\ell}\eta_{2\ell }^{\la_1}.
\end{gather*}
 Let us replace $\theta_{2m }^\la $, $\om_{2\all}^\la ,\eta_{2\ell}^\la $ by  the above linear combinations in which $\la_1$ is replaced by  the variable $\la\in I_2$. We repeat  the procedure  for $I_3,\dots, I_p$ successively. Thus we may assume that at the origin of $\rr^N$
$$
\eta_{j\ell }^{\la_{j}}=\eta_{(j+1) \ell }^{\la_{j}}, \quad \om_{j\all}^{\la_{j}}=\om_{(j+1)\all}^{\la_{j}},\quad \theta_{jm}^{\la_{j}}=\theta_{(j+1)m}^{\la_{j}}.
$$
We have obtained continuous bases that are piecewise smooth in class $\cL C^s$.  Using a partition of unity for $I$,  we find  bases $\{\eta_{ \ell }^\la \}$, $\{\eta_{ \ell }^\la ,\om_{ \all}^\la \}$,    $\{\eta_{\ell }^\la ,\om_{ \all}^\la ,\ov{\om_{ \all}^\la }, \theta_{m}^\la \}$ of $(S^\la)^* \cap \ov{(S^\la )^*}$, $(S^\la)^*$,    and $T_0^*\rr^N\otimes\cc$,
respectively. Moreover,
  the  coefficients of the bases  are  in $\cL C^{s}(I)$, and $\theta_{  m }^\la ,\eta^\la _{\ell}$ remain   real-valued.

We   express uniquely
 $$
 \om^\la _\all=df^\la _\all, \quad \theta_m^\la =dg_m^\la , \quad\eta_\ell ^\la =dh_\ell ^\la ,
 $$
where $f^\la _\all,g_ m ^\la ,h_\ell^\la $ are linear functions in $(\RE z,\IM z, t,\xi)$.
  Let
$$
\var^\la \colon \hat z_\all=f_\all^\la (z,t,\xi), \quad \hat t_{m}=g_m^\la (z,t,\xi), \quad \hat \xi_ \ell =h^\la _ \ell (z,t,\xi).
$$
Then $\{\var^\la \}$  is $\cL C^{\infty,s}$ near the origin, and at the origin   the pull-backs of
  $d\hat z_\all,d\hat t_ m $ via $\var^\la $ are $\om^\la _\all$  and $\theta_ m ^\la $, respectively.
Using the new linear coordinates, we have found, with possible new linear combination, bases
 \aln
 \om_\all^\la &=\tilde A_{\all\beta}^\la dz_\beta+\tilde B^\la _{\all\beta}d\ov z_\beta+\tilde C_{\all m}^\la \, dt_m+\tilde D_{\all \ell }^\la d\xi_ \ell ,\\
 \eta_\ell ^\la &=\RE\{\tilde a_{\ell \all}^\la dz_\all\}+\tilde c_{\ell  m}^\la \, dt_{m}+\tilde d_{\ell   \ell'}^\la d\xi_{\ell'}.
 \end{align*}
 At the origin, $(\tilde A_{\all\beta}^\la )$ and $(\tilde d_{ \ell  \ell'}^\la )$ are identity matrices and all other coefficients
 vanish.    Near the origin, the vector fields annihilated by the forms have a unique basis of the form \re{xiza1}-\re{xiza2}.
 The initial normalization is achieved.
\end{proof}

 Following Webster~\ci{We89}, we call $\{X_m^\la ,Z_\all^\la \}$  an {\it adapted frame} of $S^\la $. The adapted frame will be useful to reformulate the integrability via differential forms for the $ d_t+\db$ complex.

Another consequence of the proof of \rl{intialcod} is to  identify suitable homogeneous solutions for $\{S^\la \}$ with the normalization of $\{S^\la \}$.
\pr{solmap} Let $D$ be  a domain in $\rr^N$ and let $p\in D$.
 Let $\{S^\la \}$ be a family of CR vector bundles  in $D$ of class $\cL C^{r,s}$ with $r\geq\max\{1,s\}$. Suppose that $S^\la $ have rank $n+M$  and  CR dimension $n$. Let $L=N-2n-M$. Suppose that for each $\la_0\in I$ there are an open interval $J$ with $\la_0\in J$ or $\la_0\in\ov J$ for $\la_0=0,1$, a neighborhood $U$ of $p$, complex-valued  functions $f^\la =(f_1^\la ,\ldots, f_{n}^\la )$, and  real functions $h^\la =(h_1^\la ,\dots, h_L^\la )$ for $\la\in \ov J$  so that $\{(f^\la ,h^\la );\la\in J\}$ are in $\cL C^{r,s}( U)$,  and $f^\la ,h^\la $ are annihilated by sections of $S^\la $. Assume further that the Jacobian matrix of $( f^\la ,\ov {f^\la }, h^\la )$ has
  rank $2n+L$ at $p$.  Then there exist a neighborhood $V$ of $p$ and
$\{\var^\la ; \la\in I\}\in \cL C^{r,s}(\ov V)$
such that each  $\var^\la $ is a diffeomorphism  defined in $V$ and  $\var^\la _*(S^\la )$ is spanned by
$$
\pd_{\ov 1}, \dots,\pd_{\ov n}, \quad \pd'_1,\dots, \pd'_M.
$$
\epr
\begin{proof}
Let   $\om_\all^\la $ and $\eta^\la _\ell $ be the linear forms such that at the origin of $\rr^N$,  $\om_\all^\la =df_\all^\la $ and $\eta^\la _ \ell  =dh_ \ell  ^\la $.   By an argument similar to the proof of \rl{intialcod}, we can   find  linear real $1$-forms
 $\theta^\la _1,\dots,\theta^\la _M$ with  coefficients in $\cL C^s([0,1])$  so that $\om_\all^\la , \ov{\om_\all^\la }, \theta_m^\la ,\eta_ \ell  ^\la $ form a basis of $T_0^*\rr^N\otimes\cc$. Note that when $s\leq r<s+1$, the coefficients of $\om_\all^\la$ and $\eta_\ell^\la$ are $\cL C^{r-1}$ in $\la$. We first find $\theta_m^\la$ of which the coefficients are of $\cL C^{r-1}$ in $\la$. Then by an approximation, we can replace them by forms of which the coefficients are $\cL C^s$ in $\la$.
 Since $
 \theta^\la _m$ is linear, then
 $\theta^\la _m=dg_m^\la $ for a linear function $g_m^\la $.
Define   
$$
\var^\la\colon z=f^\la (\hat z,\hat t,\hat \xi), \quad t=g^\la (\hat z,\hat t,\hat \xi), \quad \xi=h^\la (\hat z,\hat t,\hat \xi).
$$
Then $\var^\la _*S^\la $ is spanned by $\partial_{\ov\all},\pd'_ m $.
\end{proof}


In the definition of  $\cL C^{r,s}$ spaces,   \rl{intialcod},
 \rp{solmap} and some propositions in the paper,  it suffices to require that $r\ge \max\{1,s\}$ without the restriction of $\{r\}\geq\{s\}$.
In this paper,  the use of the chain rule is indispensable and thus it restricts the exponents $r,s$ as shown by the following.
\begin{rem} \label{rsrem}  Let $U,V$ be bounded domains in Euclidean spaces. Suppose that there is a constant $C>0$ so that two points $p,q$ in $\ov U$ can be connected by a piecewise smooth curve of length at most $C|p-q|$. Suppose that
$G^\la$ map $\ov U$ into $\ov V$.  If  $\{f^\la\}\in\cL C^{r,s}(\ov {V} )$ and
$\{G^\la\}\in\cL C^{r,s}(\ov U)$
 then $\{f^\la\circ G^\la\}\in\cL C^{r,s}(\ov U )$, provided \eq{rsrs}
 \text{   $s=0$;  or $1\leq s \leq r$ and $\{s\}\leq\{r\}$}.
 \eeq
  \end{rem}

Assume   that $r$ and $s$ satisfy \re{rsrs}. Let $\{S^\la \}\in \cL C^{r,s}$  be as in \rl{intialcod}.
   Then the adapted frame
$\{X_m^\la ,Z_{\ov\all}^\la \}$ are in $\cL C^{r,s}$.
Since the coefficients $A^\la ,B^\la ,a^\la ,b^\la $ in the  frame
  vanish at the origin,   by dilation $(z,t,\xi)\to \del (z,t,\xi)$    we achieve
\eq{initiale}
\|\{(A^\la ,B^\la ,a^\la ,b^\la )\}\|_{B_1;r,s}\leq
\e
\eeq
for a  given $\e>0$.

Throughout the rest of paper, we assume that $r,s$ satisfy \re{rsrs}.

\setcounter{thm}{0}\setcounter{equation}{0}
\section{Preliminary and interpolation inequalities
}\label{sec3}

The main purpose of this section is to acquaint  the reader  with some interpolation inequalities for H\"{o}lder norms on domains with the cone property; see the appendix for the definition of the cone property.
 The reader is referred to  H\"{o}rmander~\ci{Ho76} and Gong-Webster~\ci{GW11} for the proof of these inequalities.
The parametric version of those inequalities are proved in  appendix A.

If $D, D_1, D_2$ are domains of \ps,  we have
\gan
|u|_{D;(1-\theta)a+\theta b}\leq
C_{a,b}   |u |_{D;a}^{1-\theta} |u |_{D;b}^{\theta },\label{A1-theta}
\\
 |f_1 |_{a_1+b_1} |f_2 |_{a_2+b_2}\leq C_{a,b} (|f_1|_{a_1+b_1+b_2} |f_2 |_{a_2}+ |f_1 |_{a_1} |f_2 |_{a_2+b_1+b_2}).\label{Af1f2}
\end{gather*}
  See~\cite{Ho76}*{Theorem~A.5} and~\cite{GW11}*{Proposition~A.4}.
   Here $ |f_i |_{a_i}= |f_i |_{D_i;a_i}$.  Let  $D_i$ be finitely many domains of the cone property. If $ |u_i |_{a_i}= |u_i |_{D_i;a_i}$,  then
\ga
\nonumber 
\prod_{j=1}^m |u_j |_{b_j+a_j}\leq C^m_{a,b}\sum_{j=1}^m |u_j |_{b_j+a_1+\dots+a_m}\prod_{i\neq j} |u_i |_{b_i}.
\end{gather}

To deal with two H\"{o}lder exponents in $y,\la$ variables, we introduce
 the following notation
\ga
\nonumber |f|_{D;a,0}\odot |g|_{D';0,b}:=|f|_{D;a,0}|g|_{D';0,[b]}+
 |f|_{D;[a],0}|g|_{D';0,b},\\  \nonumber 
Q^*_{D,D';a,b}(f,g) := |f|_{D;a,0}\odot |g|_{D';0,b}+ |f|_{D;0,b}\odot|g|_{D';a,0}, \\
\nonumber |g|_{D';0,b}\odot|f|_{D;a,0}=|f|_{D;a,0}\odot |g|_{D';0,b}\\
\label{Aqrsfg}
Q_{D,D';r,s}(f,g):=\sum_{0\leq k\leq \mfl{s}} Q^*_{D,D';r-j,j+\{s\}}(f,g),\quad [r]\geq [s],\\
\hat Q_{D,D';r,s}(f,g):=\|f\|_{D;r,s}|g|_{D';0,0}+
|f|_{D;0,0}\|g\|_{D';r,s}+Q_{D,D';r,s}(f,g).\nonumber
\end{gather}
We will write $Q^*_{D,D';a,b}$ as $Q^*_{a,b}$, when the domains are clear from the context, and the same abbreviation applies to $Q$ and $\hat Q$. From \rl{ff-}, we have
\le{Aff-} For each $i$,  let $D_i$  be a  domain in $\rr^{n_i}$ with the cone property and let $a_i,b_i,r_i,s_i $ be non-negative real numbers.
Assume that   $(r_1,\dots,r_m, a_1,\dots, a_m)$ or $(s_1,\dots,s_m, b_1,\dots, b_m)$ is in $\nn^{2m}$, or  at most one of  $(r_1,s_1), \dots, (r_m,s_m)$
is not in $\nn^2$. Then
\aln
&\prod_{i=1}^m|f_i|_{D_i;r_i,s_i}  
\leq C^m_{r,s}\left\{\sum_i|f_i|_{r,s}\prod_{\ell\neq i}|f_\ell |_{0,0}
+ \sum_{ i\neq j}Q^*_{r,s}(f_i,f_j)\prod_{\ell\neq i,j}|f_ \ell|_{0,0}
\right\},\\
&\prod_{i=1}^m|f_i|_{D_i;r_i+a_i,s_i+b_i}
\leq C^m_{r,s,a,b}\left\{\sum_i|f_i|_{r_i+a,s_i+b}\prod_{l\neq i}|f_ \ell |_{r_ \ell ,s_ \ell }\right.
\\
&\hspace{10em}
+ \left.\sum_{ i\neq j}|f_i|_{r_i+a,s_i}|f_j|_{r_j,s_j+b} \prod_{\ell\neq i,j}|f_ \ell|_{r_ \ell ,s_ \ell }
\right\}.\nonumber
\end{align*}
Here   $a=\sum a_i$,  $b=\sum b_i$, etc..
Assume further that $r_i\geq s_i$ for all $i$. Then
\ga\nonumber 
\|f_1\|_{D_1;r_1,s_1}\cdots \|f_m\|_{D_m;r_m,s_m}\leq C^m_{r,s}\sum_{i\neq j}\hat Q_{r,s}(f_i,f_j)\prod_{ \ell \neq i,j}|f_ \ell |_{0,0}.
\end{gather}
\ele

Let $D_\rho=B_\rho^{2n}\times B_\rho^{M}\times B_\rho^L$, and set $\|\cdot\|_{\rho;r,s}=\|\cdot\|_{D_\rho;r,s}$.
Throughout the paper  $s_*=0$ for $s=0$,  and $s_*=1$ for $s\geq1$.
By \rp{pgrsf}, we have
\pr{Agrsfp}
 Let    $0<\rho<\infty$, $ 0<\theta<1/4$ and $\rho_j=(1-\theta)^j\rho$.
Let $F^\la =I+f^\la $ map $D_\rho$ into $\rr^N$.
Assume that $f\in\cL C^{1,0}(
 {D_{\rho}})$ and
\eq{ft0m} \nonumber %
f^\la (0)=0, \quad  |\{\pd f^\la \} |_{\rho;0,0}\leq \theta/{C_N}.
\eeq
 Then $F^\la \colon D_{\rho}\to D_{(1-\theta)^{-1}\rho}$ are injective. There are $G^\la=I+g^\la$ satisfying $g^\la (0)=0$ and
\ga\nonumber
G^\la \colon D_{\rho_1}
\to D_{\rho}, \quad F^\la\circ G^\la=I  \quad \text{on $D_{\rho_1}$},\\
G^\la\circ F^\la=I \quad \text{on $D_{\rho_2}$}.\nonumber
\end{gather} Assume further that $r,s$ satisfy \rea{rsrs},
$1/4<\rho<2$,   and   $|f|_{\rho;1,s_*}\leq1$.  Then
\ga\label{Agrsf} \nonumber %
\|g\|_{\rho_1;r,s}\leq C_r\left\{\|f\|_{\rho;r,s}+\|f\|_{\rho;s+1,s}\odot\|f\|_{\rho;r,s_*}\right\},\\
\label{AuF-1}\|\{u^\la \circ (F^\la )^{-1}\}\|_{\rho_1;r,s}\leq C_r\left\{|u |_{\rho;1,s_*} \|f\|_{\rho;r,s}+\|u\|_{\rho;s+1,s}\odot\|f\|_{\rho;r,s_*}\right.\\
  \left.\hspace{6em} +\|u\|_{\rho;r,s}+|u|_{\rho;1,0}\|f\|_{\rho;s+1,s}\odot\|f\|_{\rho;r,s_*} +\|u\|_{\rho;r,s_*}\odot\|f\|_{\rho;s+1,s}\right\}.\nonumber
\end{gather}
\epr
By \rp{uFF}, we have the following.
\pr{AuFF}Let $F_i^\la =I+f_i^\la \colon D_i\to D_{i+1}$.  Assume that $D_1,\dots, D_{m+1}$ have the cone property of which $C_*(D_i)$ are independent of $i$ $($see Appendix A for notation$)$ and
$$
f_i^\la (0)=0,\quad   |f_i|_{D_i;1,s_*} \leq 1.
$$
Let $\|f_i\|_{r,s}=\|f_i\|_{D_i;r,s}$. Suppose that $r,s$ satisfy \rea{rsrs}. Then
 \aln
&\|\{u^\la \circ F_m^\la \circ\cdots\circ F_1^\la \}\|_{D_1;r,s}\leq C_r^{m}
 \{
\|u\|_{D_{m+1};r,s}+ |u|_{D_{m+1};1,0} \sum_{i\leq j} \|f_i\|_{s+1,s}\odot\|f_j\|_{r,s_*} \\
&\quad +
\sum_i(|u|_{D_{m+1};1,s_*} \|f_i\|_{r,s}
+\|u\|_{s+1,s}\odot\|f_i\|_{r,s_*}+\|u\|_{r,s_*}\odot\|f_i\|_{s+1,s})\}.
\end{align*}
\epr

The proof of the following lemma is straightforward.
\le{Ask2k}Let   $0\leq\theta_k^*\leq\f{1}{(2+k)^2}$.
Then
$$
\prod_{k=0}^\infty(1-\theta_k^*)>1/2. 
$$
Let $\rho_{k+1}=(1-\theta_k^*)\rho_k$, $\rho_\infty=\lim_{k\to\infty}\rho_k >0$.  Suppose that
$
F_k( D_{\rho})\subset D_{(1-\theta_k^*)^{-1}\rho}$
  for each $\rho\in(0, \rho_k)$.
 Then
$
F_i\circ\cdots\circ F_0(D_{\rho_\infty/2})\subset D_{\rho_\infty}.
$
\ele

\setcounter{thm}{0}\setcounter{equation}{0}
\section{Frobenius theorem with parameter and examples}\label{sec4}

 There are several ways to prove the classical  Frobenius theorems. In the literature there seems to  lack   examples showing exact regularity for the various types of Frobenius theorems.
 In the end of this section, we will provide some examples showing that some regularity results in this paper are almost sharp.

First we reformulate the proof of Frobenius's  theorem in  F.~and R. Nevanlinna~\cite{NN73} for the parametric case.
 The reader can compare it with    the loss of derivatives in \rt{thm1}.

\begin{defn} Let $[r_1]\geq [ r_2]\geq [s]$.
 The $\cL C^{r_1,r_2;s}(\ov U\times \ov V)=\cL C^{r_1,r_2;s}_{x,y;\la}(\ov U\times \ov V)$ is the set of functions $u^\la (x,y)$ such  that $\pd_x^i\pd_y^j\pd_\la^ku^\la(x,y)$ are continuous in $\ov U\times\ov V\times I$ and have bounded $\cL C^{r_1-\{r_1\}}(\ov U)$ norms for $(y,\la)\in\ov V\times I$, bounded $\cL C^{r_2-\{r_2\}}(\ov V)$ norms for $(x,\la)\in\ov U\times I$, and bounded $\cL C^{s-\{s\}}(I)$ norms for $(x,y)\in\ov U\times \ov V$,  for $k\leq s, j+k\leq r_2, i+j+k\leq r_1$.
\end{defn}
In next two propositions we consider (real) Frobenius structure $S$, i.e.  a complex Frobenius structure $S$ in a domain $D$ of $\rr^N$ with  CR dimension $0$, i.e. $\ov S=S$. Then  $M:=\dim_\cc S_p$ is independent of point $p\in D$ and it is the rank of $S$.
\pr{rfp}{\rm\bf(Frobenius Theorem with parameter.)} Let $r,s$ satisfy
\rea{rsrs}.
 Let $S=\{S^\la \}$ be a family
of $($real\,$)$ Frobenius structures in a domain $D$ of $ \rr^N$ of class $\cL C^{r,s}(\ov D)$. Assume that $S^\la$ have the constant rank $M$.
Near each point $p\in D$, there exists a family $\{\Phi^\la \}\in\cL C^{r,s}(\ov U)$ of diffeomorphisms $\Phi^\la $ from $U$ onto ${B_\delta^M}\times {B_{\epsilon}^{N-M}}$
such that $\Phi^ \la_* S^\la $ are the span of $\pd_{x_1},\ldots, \pd_{x_M}$. Furthermore, if $S_0^\la $ are tangent to $\rr^M\times\{0\}$, then we can take $(\Phi^\la )^{-1}\colon \tilde x=x,\tilde y=F^\la (x,y)$ with $\{F^\la \}\in\cL C^{r+1,r;,s}(\ov{B_\delta^M}\times \ov{B_{\epsilon}^{N-M}})$,  where  $x,y$ are coordinates of $\rr^M,
\rr^{N-M}$,  respectively. 
\epr
We remark that for the standard Frobenius structure defined by $\pd_{x_1},\dots, \pd_{x_M}$ in $\rr^N$, the
$(x_1,\dots, x_M)$ space is   the leaf space of the foliation, while $y=(y_1,\dots, y_{N-M})$ is the subspace of $\rr^N$ transversal to the leaves of the foliation.
In this section
repeated lower indices $m,m'$ are summed over the range $1,\dots, M$ and repeated indices $\ell, \ell'$ are summed over  the range $1,\dots, N-M$.
\begin{proof} We adapt  the proof in \ci{NN73}, which does not involve a parameter and is for integers $r,s$. To recall the proof let us first assume that $S^\la=S$ are independent of $\la$.
By a linear change of coordinates, we may assume that $p=0$ and
the tangent space of $S_0$ is the  $x$ subspace.   Near the origin, $S$ consists of vector fields annihilated by
\eq{dyA}
dy_ \ell -A_{ \ell  m}\, dx_m, \quad 1\leq\ell\leq N-M.
\eeq
Write $A_{m}
=(A_{1m},\dots, A_{(N-M)m})$.
 Let $d_x$ be the differential in $x$ variable.
We seek  integral submanifolds $M\colon Y=F(x,y)$ so that for a fixed $y$
\eq{dxf}
d_xF_ \ell (x,y)=A_{\ell  m}(x,F(x,y))\, dx_m, \quad F(0,y)=y.
\eeq
Then we can verify that $\tilde x=x$ and $\tilde y=F(x,y)$ transform $\pd_{ x_m}$ into $\f{\partial}{\partial\tilde x_m}+  A_{\ell m}(\tilde x,\tilde y)\f{\partial}{\partial\tilde y_ \ell }$.
It is clear that for the integral manifolds,
  $F$ must satisfy the equation in the radial integral
\ga\label{FTF}
F  (x,y)=TF  (x,y).
\end{gather}
Here $ (TF) _ \ell  (x,y)=y _ \ell  +x_m\int_0^1A_{ \ell   m}( t x, F( t x,y))\, dt.$

To prove next proposition, let us consider a more general operator
\ga\label{Twithb}
 (TF) _ \ell  (x,y):=b_\ell(y)  +x_m\int_0^1A_{ \ell   m}( t x, F( t x,y))\, dt
\end{gather}
with $b_\ell(0)=0$ and $ b_\ell\in \cL C^{r}(\ov {B_{\del_0}^M})$. Let $F_0(x,y)=b(y)$
and $F_{n+1}=TF_n(x,y)$.
By the Picard iteration,   equations \re{FTF}-\re{Twithb} admit
a $\cL C^1$ solution $F=\lim_{n\to\infty}F_n$ in  $K_{\delta,\epsilon}:=\ov{B_\delta^M}\times \ov{B_{\epsilon}^{N-M}}$ for some positive $\e$ and $\delta$, provided  that on  $K_{\delta_0,\epsilon_0}$ \eq{}
|A(x,y)|<C, \quad |\pd_y A(x,y)|<{C}.\nonumber
\eeq

We can verify that $F$ satisfies \re{dxf}. Since \re{Twithb} is a slightly general situation, we provide some details for the reader,    using a method in \cite{NN73}*{pp.~158-163}.
The
integrability condition means that $$d(A_{\ell m}(x,y)\, dx_m)=0\mod \text{span}\{dy_{\ell'}-A_{{\ell'}m}(x,y)\, dx_m\}. 
$$
Replacing $dy_{\ell'}$ by $A_{{\ell'}m'}\, dx_{m'}$   in the first identity after the differentiation, we obtain
\al\label{AiAj}
\pd_{x_{m'}}A_{\ell m}(x,y)&-\pd_{x_m}A_{\ell {m'}}(x,y)\\
&= \pd_{y_{\ell'}}A_{\ell {m'}}(x,y)A_{{\ell'}m}(x,y)-
\pd_{y_{\ell'}}A_{\ell m}(x,y)A_{{\ell'}{m'}}(x,y).\nonumber
\end{align}
To show that $d_xF(x,y)=A(x,y)dx$ for $y=F(x,y)$, it suffices to verify that
$$
(R_{n})_{\ell m}(x,y)=\pd_{x_m}(F_n)_\ell(x,y)-A_{\ell m}(x,F_{n}(x,y))$$
tends to zero in sup-norm as $n\to\infty$.
Differentiating
$$(F_{n+1})_\ell(x,y)=b_\ell(y)+x_m\int_0^1A_{\ell m}( t x,F_n( t x,y))\, dt,$$ we get
\al\nonumber%
\pd_{x_m}(F_{n+1})_\ell(x,y)&=\int_0^1\Bigl\{A_{\ell m}( t x,F_n( t x,y))  + x_{m'}  t (\pd_{x_m}A_{\ell {m'}})
( t x,F_n( t x,y))\Bigr\}dt\\
\nonumber
&\quad+ x_{m'}\int_0^1(\pd_{Y_{\ell'}}A_{\ell {m'}})( t x,F_n( t x,y)) t (\pd_{x_m}(F_n)_{\ell'})( t x,y)\, dt.
\end{align}
Adding $0=A_{\ell m}(x,F_n(x,y))-\int_0^1\pd_{t}( t A_{\ell m}( t x,F_n( t x,y))\, d t $ to the right-hand side, we
get
\aln
\pd_{x_m}(F_{n+1})_\ell(x,y)&=A_{\ell m}(x,F_n(x,y)) - x_{m'}\int_0^1t\pd_{x_{m'}}A_{\ell m}( t x,F_n( t x,y))\, dt\\
&- x_{m'}\int_0^1t\pd_{Y_{\ell'}}A_{\ell m}( t x,F_n( t x,y))\pd_{x_{m'}} (F_n)_{\ell'}( t x,y)\, dt\\
&+ x_{m'}\int_0^1 t (\pd_{x_m}A_{\ell {m'}})
( t x,F_n( t x,y))\, d t \\
&+ x_{m'}\int_0^1(\pd_{Y_{\ell'}}A_{\ell {m'}})( t x,F_n( t x,y)) t (\pd_{x_m}(F_n)_{\ell'})( t x,y)\, dt.
\end{align*}
We want to express the above in terms of  $R_n$. By
\re{AiAj} in which $F_n(x,y)$ substitutes for $y$, we simplify the right-hand side and  obtain
\aln
&(R_{n+1})_{\ell m}(x,y)=x_{m'}\int_0^1\pd_{Y_{\ell'}}A_{\ell {m'}}( t  x,F_n( t  x,y))(R_n)_{{\ell'}m}( t  x,y)\, dt-
\\
&x_{m'}\!\!\int_0^1\!\!
\pd_{Y_{\ell'}}
A_{\ell m}( t  x,F^n( t  x,y))(R_n)_{\ell'{m'}}( t  x,y)\, d t
 +A_{\ell m}(x,F_n(x))-A_{\ell m}(x,F_{n+1}(x,y)).
\end{align*}
We have $|\pd_yA(x,y)|<C$ and
$$|A(x,F_{n+1}(x,y))-A(x,F_n(x,y))|\leq C|F_{n+1}(x,y)-F_n(x,y)|.
$$
For some $0<\theta<1$, we have
$|F_{n+1}(x,y)-F_n(x,y)|\leq C\theta^n$ and
\al
|R_{n+1}|_0\leq C_0\theta^n+\theta|R_n|_0, \quad n\geq0.\nonumber%
\end{align}
Therefore,  $R_n$ tends to $0$ in the sup norm as $n\to\infty$ and  $d_xF(x,y)=A(x,F(x,y))\, dx$.

We now consider the parametric case. By the initial normalization in \rl{intialcod}, we may assume that $p=0$ and all $S_0^\la$ are tangent to the $y$-subspace.  Thus $S^\la$ are defined by \re{dyA} in which
 $A_m=(A_{1m},\dots, A_{(N-M)m})=A^\la_m$ depend on $\la$. Then we  take
 $F_0^\la(x,y)=b^\la(y)$ and $F_{n+1}^\la (x,y)=TF_n^\la(x,y)$. Thus
\eq{fLay}
(F_{n+1}^\la)_\ell (x,y)= b_\ell^\la(y)+x_m\int_0^1A_{\ell m}^\la (\theta x,F_n^\la (t x,y))\, dt.
\eeq
By the   Picard iteration, if $|A|_{K_{\delta_0,\epsilon_0};1,0}<C_0$, we find $F^\la=\lim_{n\to\infty}F_n^\la$ with $\{F^\la\}\in \cL C^{1,0}(K_{\delta,\epsilon})$ for some positive $\del,\e$. By a bootstrap argument by shrinking $\del,\e$  a few more times and keeping
 them independent of $r,s$,  one can verify the full regularity.   We briefly indicate the argument.
Let $\all=\{r\}>0$ and consider the case $[r]=[s]=1$ and $\{s\}=\beta\leq\all$. For $1\leq m\leq M$,
\al\nonumber
&\pd_{x_m}F^\la (x,y)=x_{m'}\int_0^1 t (\pd_{y_ \ell  }A_{m'}^\la )( t  x,F^\la ( t  x,y))\pd_{x_m}F_ \ell  ^\la ( t  x,y)\, d t \\
&\qquad +x_{m'}\int_0^1 t  \pd_{x_m}A_{m'}^\la ( t  x,F^\la ( t  x,y))\, d t  + \int_0^1A_m^\la ( t  x,F^\la ( t  x,y)) \, d t.\nonumber
\end{align}
By another Picard iteration, we can show that $\pd_{x_m}F\in \cL C^{\all,0}$. One can also show that $\{\pd_{y_ \ell  }F^\la \}\in \cL C^{\all,0}$.  When $s\geq1$, we can verify that $\pd_\la F^\la \in C^0$ and
\al\nonumber
\pd_{\la}F^\la (x,y)&=\pd_\la b^\la(y)+x_m\int_0^1 (\pd_{y_ \ell }A_m^\la )( t  x,F^\la ( t  x,y))\pd_{ \la}F_ \ell  ^\la ( t  x,y)\, d t \\
&\quad +x_m\int_0^1  (\pd_{\la}A_m^\la )( t  x,F^\la ( t  x,y))\, d t .\nonumber
\end{align}
By the Picard iteration, we can verify that $\{\pd_\la F^\la \}\in C_*^{0,\beta}$, using $\all\geq\beta$.

While taking derivatives one-by-one without further shrinking $\del,\e$, we can verify that $\{F^\la \}\in \cL C^{r,s}$.
\end{proof}

The above  proof for \re{dyA}, \re{FTF} and \re{Twithb}   shows the following.
\pr{imf}Let $\{A^\la\}\in \cL C^{r,s}(\ov {B_{\e_0}^M}\times\ov{B_{\del_0}^{N-M}})$ with $A^\la$ being $L\times M$ matrices satisfying the integrability condition \rea{AiAj}.  Let $b\in \cL C^{r,s}(\ov {B_{\e_0}^M})$. Assume that $A^\la(0)=0$ and $b^\la(0)=0$. There exists $\{F^\la\}\in \cL C^{r+1,r;s}(\ov{B_\del^M}\times\ov{B^{N-M}_\e})$ for   some positive  $\del$, $\e$ satisfying $F^\la(0,y)=b^\la(y)$ and
\eq{dxF} \nonumber %
d_xF^\la_\ell(x,y) = A^\la_{\ell m}(x,F^\la(x,y))\, dx_m,\quad 1\leq\ell\leq N-M.
\eeq
\epr

It is interesting that the proof by F.~and R.~Nevanlinna produces a sharp result. Let us demonstrate it
 by examples.
Our first example shows that in general  one cannot expect to gain any derivatives  in $\la$ for finite smooth case.
\begin{exmp} (Complex version.)
Let $a^\la $ be a  function that is $\cL C^\infty$
for $\la \neq0$ and is $\cL C^{1+\all}$. Suppose that $a^\la $ is independent of $z$ and $|a^\la |<1$.
Let $F^\la(z)=z+{a^\la }\ov z$. Then $F^\la_*\pd_{\ov z}$  is
$$
Z^\la=\pd_{\ov z}+\ov{a^\la }\pd_z.
$$
Thus $Z$ is of class $\cL C^{1+\all}$ in $(z,\la)\in\cc\times\rr$. When regarding $F^\la(z)$ as a mapping $\tilde F(z,\la)=(F^\la(z),\la)$ and $Z^\la(z)$ as a vector field $\tilde Z$ in $(z,\la)$, we still have
$$
\tilde F_*\pd_{\ov z}=\tilde Z.
$$
We want  to show that there does not exist a diffeomorphism $G\colon (z,\la)\to G(\la,z)$ so
that $G$ is defined in
a neighborhood of origin in $\cc\times\rr$ and of class $ \cL C^{1+\beta}$ in $(z,\la)$ for  $\beta>\all$,  while $G$ transform $\tilde Z$ into $\pd_{\ov z}$.

Suppose that such a  $G$ exists. Then $G(F^\la(z),\la)=H(z,\la)$ is
 holomorphic in $z$. Write $H(z,\la)=(h(z,\la),r(z,\la))$ and $G(z,\la)=(g(z,\la),s(z,\la))$, where $r$ and $s$ are real-valued. Since $r(z,\la)$ is holomorphic in $z$ and real-valued, then $r(z,\la)$ is independent of $z$. Thus, $\pd_zh\neq0$. Set  $\tilde  w=(1-|a^\la |^2)^{-1}(z-a^\la \ov z)$. We  have
$$
g(z+a^\la \ov z,\la)=h(z,\la), \quad g(z,\la)=h(\tilde w,\la).
$$
 Since $h(z,\la)$ is holomorphic in $z$ and is $ \cL C^{1+\all}$, Cauchy's formula implies that $\pd_zh(z,\la)\in \cL C^{1+\all}$. Then
\eq{pdzg}
\pd_zg(z,\la)=(1-|a^\la |^2)^{-1}\pd_{\tilde  w }h(\tilde w,\la), \  \pd_{\ov z}g(z,\la)=-(1-|a^\la |^2)^{-1}a^\la \pd_{\tilde  w}h(\tilde w,\la)
\eeq
  are in $\cL C^{1+\all}$.  Since $h$ is holomorphic in $z$, then
$$
\int_{|z|=\e}zg(z+a^\la \ov z,\la)\, dz=\int_{|z|=\e }zh(z,\la)\, dz=0.
$$
Let $w=z+{a^\la }\ov z$. Taking $\la$ derivative, we get
\eq{AtBt}
A^\la   \pd_\la{a^\la } +B^\la \pd_\la\ov {a^\la } =-\int_{|z|=\e}z\pd_\la g(w,\la)\, d\la
\eeq
where
$
A^\la =\e^2\int_{|z|=\e}\pd_w g(w,\la)\, d\la$ and $  B^\la = \int_{|z|=\e}z^2\pd_{\ov w}g(w,\la)\, dz.
$
The right-hand side and $A^\la ,B^\la $ are   in $\cL C^\beta$.  Since $ \pd_wh(0,\la)\neq0$ and $|a^\la |<1$,  then \re{pdzg} imply that
 $|A^\la |>|B^\la |$, when $\epsilon$ is sufficiently small.
Solving for $\pd_\la a^\la ,\pd_\la\ov{a^\la }$ from \re{AtBt}, we see that $\pd_\la a^\la $ is in $\cL C^\beta$, a contradiction.
\end{exmp}

The next example shows that for   Frobenius structures, we cannot expect to gain derivatives. This  contrasts with the complex structures for which normalized transformation gain some derivatives.
\begin{exmp} (Real version.)
For $(x,y,z)\in\rr^3$, let $F(x,y,z)=(x,y+a(z)x,z)$.
Here $a$ is in $\cL C^{1+\all}$ but  not in $\cL C^{1+\beta}$ for any $\beta>\all$.
Then  $F_*\pd_x$ equals
$$
X=\pd_x+a(z)\pd_y.
$$
We claim that there is no $G\in\cL C^{1+\beta}$ with $\beta>\all$   transforming $X$ into $\pd_x$.

Suppose that such a $G$ exists. Then the last two components
of $G\circ F(x,y,z)=S(x,y,z)$ are independent of $x$. Let $s,g$ be last two components of $S$ and $G$, respectively.
   We
get
$$
s(y,z)=g(x,y+a(z)x,z).
$$
Then $s(y,z)=g(0,y,z)$ is also in $\cL C^{1+\beta}$.
Since the Jacobian of $S$ does not vanish at the origin, then $\pd_yg(0)\neq(0,0)$. Without loss of generality, we may assume that $\pd_yg_1(0,0)\neq0$. Let $u=u(x,y,z)$ be the $\cL C^{1+\beta}$ solution to $s_1(x,y)=g_1(x,u,z)$.  Then
 $y+a(z)x=u(x,y,z)$ is in $\cL C^{1+\beta}$
near $(x,y,z)=0$, a contradiction.
\end{exmp}

 As mentioned in the introduction,  Nirenberg~\ci{Ni57} asked if an  analytic family of $\cL C^\infty$ Levi-flat CR vector bundles could be normalized by an analytic family of smooth diffeomorphisms. Nirenberg's question  was justified by  his theorem on a holomorphic family of complex structures~\ci{Ni57}; see \rp{nithm} below.
We now provide a negative answer to the analogous question for the finite smoothness case. This shows    a significant difference between the analytic families of
general complex Frobenius and complex structures.  We recall
\eq{opPt}
P^\la =\pd_x+(|y|^{a+1}+\la)\pd_y, \quad (x,y)\in \rr^2.
\eeq
\pr{countex} Let $a\in(1,\infty)\setminus\nn$. Let $P^\la $ be defined by \rea{opPt}. Then
$
P^\la  N^\la =0
$
does not admit   solutions $N^\la $ defined in a neighborhood of the origin in $\rr^2$ that satisfy
\eq{exactreg}
dN^\la (0)\neq0, \quad\{N^\la \} \in  \cL C^{b,1}_*, \quad b>a.
\eeq
In particular, there does not exist a real analytic family of $\cL C^b$ diffeomorphisms that transform
$P^\la
$ into $\pd_{\hat x}$.
\epr
\begin{proof} We may assume that $a<b<[a]+1$.
Suppose for the sake of contradiction that there is a family of solutions $N^\la $ satisfying $P^\la N^\la=0$ and \re{exactreg}.
 We define $(x,y)=\varphi (\hat x,\hat y)$ as
$$
x=\hat x,\quad y=\hat y(1-a\hat x\hat y^{a})^{-{1}/{a}}.
$$
Note that $\var$ preserves the sign of the $y$ component and $\var$ is of class $\cL C^{a+1}$.  Thus,
 in what follows, we will restrict $y$ and $\hat y$ to be non-negative.
 Note that $\var^{-1}(x,y)=(x,y(1+axy^a)^{-1/a})$.
Applying the chain rule, we verify that
\ga\label{p0Q}
\var_*\pd_{\hat x}=P^0, \quad (\var)^{-1}_*\pd_y=Q,\\
Q:= (1-a\hat x\hat y^a)^{(a+1)/a}\pd_{\hat y}. \nonumber %
\end{gather}

Define    $N_j=(\pd_\la)^j|_{\la=0}N^\la $.  Using $P^\la N^\la =0$,   the constant and linear terms  in $\la$ on the right side  yield
$$
P^0N_0=0, \quad P^0N_1=-\pd_yN_{0}.
$$
In the $(\hat x,\hat y)$-coordinates,  the above equations become
\eq{P0n0}\nonumber %
\pd_{\hat x}\hat N_0=0, \quad\pd_{\hat x}\hat N_1= -Q\hat N_0
\eeq
with $\hat N_i=N_i\circ\var\in \cL C^b$.
Let us drop hats in $\hat N_i$ and $(\hat x,\hat y)$.
We first note that 
$N_0(x,y)$ is independent of $x$, which is now denoted by  $B(y)$. We have $B\in \cL C^b$.   By the fundamental theorem of calculus,  we obtain
\ga
N_1( x,  y)-  N_1(0,  y)=  -B'(y)u(x,y), \label{N1n1}\\ u(x,y):=\int_0^{  x}(1-axy^a)^{(a+1)/a} \, dx.\label{N1n1b}
\end{gather}
We have $u\in\cL C^a$.  Fix a small $x_0\neq0$ so that $u(x_0,y)\neq0$. Since $b>a$,  by $B'(y)=(N_1(x_0,y)-N_1(0,y))u^{-1}(x_0,y)$
 we obtain $B'\in \cL C^{a}$.

 Next, we show that $yB'(y)$ has a better regularity. Differentiating \re{N1n1}, we obtain
 \eq{}\nonumber
 B^{([a]+1)}(y)u(x,y)=-B'(y)\pd_y^{[a]}u(x,y)+\tilde B(x_0,y),
 \eeq
with $\tilde B\in \cL C^{b-[a]}$. It is straightforward that $y\pd_y^{[a]}u(x,y)$ is $\cL C^1$ in $y$. Therefore, $yB^{([a]+1)}$(y) is in $\cL C^{b-[a]}$.  Now,  the function $yB'(y)$ is in $\cL C^b$  because $(yB'(y))^{[a]}=yB^{[a]+1}+[a]B^{[a]}(y)$.

Using the Taylor series of the power function,  we express \re{N1n1b} as
$$
u (x,y)=x-\f{a+1}{2}x^2y^a+yE (x,y),
$$
with $E\in \cL C^{2a-1}$. In \re{N1n1},  we substitute $u$ by the above expansion and obtain
$$
xB'(y)-\f{a+1}{2}x^2y^aB'(y)=N_1(0,y)-N(x,y)-(yB'(y))E(x,y).
$$
The right-hand side is of class $\cL C^{b'}$ for $b'=\min(2a-1,b)$.  Plugging in $x=\epsilon, 2\epsilon$ in the above identity, we get two equations for $B'(y)$
and $y^aB'(y)$. Solving them, we conclude that $B'(y), y^aB'(y)$ are in $\cL C^{b'}$.
Since $dN^0(x,y)\neq0$ and $B(y)=N_0(x,y)$, then  $B'(0)\neq 0$. Therefore,  $y^a$ is in $\cL C^{b'}$, a contradiction.\end{proof}

We can also prove the following  result for  non-homogeneous equations.
\pr{cex} Let $P^\la $ be defined by \rea{opPt}. Let   $a'\in(1,\infty)\setminus \nn$ and $a'<a$. Then
$$
P^\la v^\la =|y|^{a'+1}
$$
does not admit a solution $\{v^\la \}\in \cL C^{b, 1}_*(\ov D)$ for $b>a'$ and any  neighborhood $D$ of $0$.
\epr
\begin{proof}Assume for the contrary that such a solution $\{v^\la \}$ exists.
Let $(x,y)=\var(\hat x,\hat y)$ be as in the above proof, and  let $v_i=\pd_\la^i|_{\la=0}v^\la $ and $\hat v_i(\hat x,\hat y)=v_i\circ\var(\hat x,\hat y)$.
  Then $P^\la v^\la =|y|^{a'+1}$ implies that
$$
P^0v_0=|y|^{a'+1}, \quad P^0v_1=-\pd_yv_0.
$$
 Let us drop hats in $\hat v_i$,  $\hat x,\hat y,\hat v_i$, and restrict $y\geq0$. By \re{p0Q}, we get
$$
\pd_x v_0(x,y)=y^{a'+1}(1-axy^a)^{-(a'+1)/a}, \quad \pd_xv_1(x,y)=-(1-axy^a)^{(a+1)/a}\pd_y v_0(x,y).
$$
 Using Taylor series of the power function in both identities, we get
\aln
v_0(x,y)&=A(y)+xy^{a'+1}  +E_{2a'+1}(x,y), \\
v_1(x,y)&=B(y)-C(x,y)A'(y)-\f{a'+1}{2}x^2y^{a'}  +E_{2a'}(x,y),
\end{align*}
where
$
C(x,y)=x+x^2E_{a}(x,y),
$ and $E_r\in \cL C^{r}$. We may assume that $a'<b<[a']+1$ and $b<a$.  Since $\{v^\la \}\in \cL C^{b,1}_*$ and $\var\in \cL C^{a+1}$, then $v_i, A,B$ are in $ \cL C^{b}$.
 Plugging in $x=\e,2\e$, we can express $A'(y),y^{a'}$ as linear combinations of
 $\cL C^b$ functions, of which the coefficients are also $\cL C^b$ functions. Here we have used
 $$
 \begin{vmatrix}
 (2\e)^2&\e^2\\
 C(2\e,y)&C(\e,y)
 \end{vmatrix}=2\e^3+4\e^4E_a(\e,y)-4\e^4E_a(4\e,y)\neq0
 $$
 when $\e>0$ is sufficiently small.
 This shows that $y^{a'}$ is of class $\cL C^b$ with $b>a'$, which is a contradiction.
\end{proof}

By \rp{rfp},  the equation in \rp{cex}  has a solution $\{v^\la \}\in \cL C^{r,r}(D)$ for some neighborhood $D$ of the origin  and $r=a'+1$, when $a\geq a'$ and $a'\in(0,\infty]\setminus\nn$.

\setcounter{thm}{0}\setcounter{equation}{0}
\section{A homotopy formula for mixed real and complex differentials
}\label{sec5-}

In this section, we will adapt estimates for the Koppelman-Leray homotopy formula, which are  due to Webster~\ci{We89}. We will also  recall a homotopy formula
 for the $d^0+\db$ complex, which is defined in this section.
The homotopy formula for $d^0+\db$ is obtained by combining the Bochner-Martinelli, Koppelman-Leray formulae for all degrees and Poincar'e homotopy formula, which is due to Tr{e}ves~\ci{Tr92}.

Let us first  recall the Bochner-Martinelli and Koppelman-Leray  formulae for the ball $B_\rho^{2n}\subset \cc^n$. The reader is refer to Webster~\ci{We89} and Chen-Shaw~\ci{CS01} for details.
For $(\zeta,z)\in\cc^n\times\cc^n$ with $\zeta\neq z$, let
\gan
\om^0=\f{1}{2\pi i}\f{(\ov\zeta-\ov z)\cdot d\zeta}{|\zeta-z|^2}, \quad\om^1=\f{1}{2\pi i}\f{\ov\zeta\cdot d\zeta}{\ov\zeta\cdot(\zeta-z)},\\
\Om^{\ell}=\om^i\wedge(\db\om^i)^{n-1},\quad i=0,1; \\ \Om^{01}=\om^0\wedge\om^1\wedge\sum_{\all+\beta=n-2}(\ov\pd\om^0)^\all\wedge
(\ov\pd\om^1)^\beta.
\end{gather*}
Let $\Om^0_{q}$, $\Om^{01}_{q}$ be the components of $\Om^0$
and $\Om^{01}$ of type $(0,q)$ in $z$, respectively:
$$
\Om^i=\sum_{q=0}^{n-1}\Om^i_{q}, \quad i=0,1;\qquad \Om^{01}=\sum_{q=0}^{n-2}\Om^{01}_{q}.
$$
The Koppelman lemma says that $\db_{\zeta,z}\Om^0=0$ and $\db_{\zeta,z}\Om^{01}=\Om^0-\Om^1$. Thus
\gan
\db_\zeta\Om_q^{0}=-\db_z\Om^0_{q-1}, \quad
\Om_q^{0}=\Om^1_q-\db_z\Om^0_{q-1}-\db_\zeta\Om^{01}_{q}.
\end{gather*}
Let $\var$ be a
   $(0,q)$ form in  $\cL C^1(\ov {B^{2n}_\rho})$, and let $f$ be a $\cL C^1$ function. We have
\gan
\var=  \db P_{q}\var+P_{q+1}\db \var,\quad q>0; \quad
\var= B_0^1 \var+ P_{ 1}\db \var, \quad q=0;\\
P_q\var= P_q^0\var-P_q^{01}\var;\\
P_q^0\var(z):=\int_{ B^{2n}_\rho}\Om_{q-1}^{0}(\zeta,z)\wedge\var(\zeta), \quad P_q^{01}\var(z)=: \int_{\pd B^{2n}_\rho} \Om_{q-1}^{01}(\zeta,z)\wedge \var(\zeta), \quad q>0;\\
B^1_0\var(z):=\int_{\pd B^{2n}_\rho}\Om_0^{1}(\zeta,z) \var(\zeta),\quad q=0.
 \end{gather*}
Note that $ \db_z \Om^1_0(\zeta,z)=0$.   When $n=1$, $P_1=P_1^0$ is the Cauchy-Green operator.

The Poincar\'e lemma for a $q$-form $\phi$  on the ball $\ov B^M_\rho$ has the form
\ga\label{pht}
\phi=d^0R_q\phi+R_{q+1}d^0\phi, \quad q>0;\quad \phi=R_1d^0\phi+\phi(0), \quad q=0; \\
\quad    R_q\phi(t):=\int_{\theta\in[0,1]}H^*\phi(t,\theta). \nonumber
\end{gather}
Here $H(t,\theta)=\theta t$ for $(t,\theta)\in B_\rho^M\times[0,1]$. Note that
$$
d(\theta t_{1})\wedge\cdots\wedge d(\theta t_{q})=\sum(-1)^{j-1} t_j\theta^{q-1}d\theta\wedge dt_1\wedge\cdots\widehat{dt_j}\cdots\wedge dt_{q}.
$$
It is immediate
that for $q>0$
\eq{3p13}
 |R_q\phi |_{\rho;a}\leq C_a\rho |\phi |_{\rho;a}, \quad \phi\in\cL C^a(\ov {B^M_\rho}), \quad a\in[0,\infty).
\eeq
Note that there is no gain in derivatives in estimates of $R_q$. Set $R_{j}=0$ for $j>M$.

Next, we consider the complex for the exterior differential
$$\cL D:=d_{t}+\db_z$$
for $(z,t)\in\cc^n\times\rr^M$.  We will also write the above as
$\cL D=d^0+\db.
$
A differential form $\var$ is called of {\it mixed type} $(0,q)$
if
$$
\var=\sum_{i=0}^ {q} [\var]_{i},
$$
where $[\var]_{ i}=\sum_{|I|=i, |I|+|J|=q} a_{IJ}d\ov z^I\wedge dt^J$.  Thus
$$
  [\var]_i=0, \quad \text{for}\  i>n.
$$
The $\cL D $ acts on a function $f$ and a $(0,q)$ form as follows
$$
\cL Df=\sum_{m=1}^M\f{\pd f}{\pd t_m}dt_m+\sum_{\all=1}^n\f{\pd f}{\pd\ov z_\all}\, d\ov z_\all, \
\cL D\!\! \sum_{|I|+|J|=q}a_{IJ}d\ov z^I\wedge dt^J=\sum_{|I|+|J|=q} \cL Da_{IJ}\wedge d\ov z^I\wedge dt^J.$$
Let $\cL C^a_{0,q}$ be space of $(0,q)$-forms of which the coefficients are of class $\cL C^a$. Then
$$
\cL D \colon \cL C_{(0,q)}^{\, a}\to \cL C_{(0,q+1)}^{\, a-1}.
$$
From $(d^0+\db+\pd)^2=0$, we get
\eq{}\nonumber
\cL D ^2=0, \quad    d^0\db+\db d^0=0.
\eeq
We also have
\ga \nonumber 
[\cL D \var]_0=d^0[\var]_{0},\quad
[\cL D \var]_i=d^0[\var]_{i+1}+\db[\var]_{i}, \quad 0<i\leq n.
\end{gather}
For   $\var=\sum \var_{IJ}d\ov z^I\wedge dt^J=\sum\tilde\var_{IJ}dt^J\wedge d\ov z^I$ on $\ov{B_\rho^{2n}}\times \ov{B_\rho^M}$, define
\aln
P_i\var&=\sum_{|I|=i}P_i(\var_{IJ}d\ov z^I)\wedge dt^J, \\
R_{q-i}\var &=\sum_{|J|=q-i}R_{q-i}( \tilde \var_{IJ}d t^J)\wedge d\ov z^I.
\end{align*}
Thus $P_i\var=P_i[\var]_i$,  while $R_{q-i}\var=R_{q-i}[\var]_i$ if $\var$ has the (mixed) type $(0,q)$. We have
\aln
d^0P_i\var&=\sum_{m=1}^M dt_m\wedge\pd_{t_m}P_i(\var_{IJ}\wedge d\ov z^I)\wedge dt^J\\
&=\sum_{m=1}^M(-1)^{i-1}\pd_{t_m}P_i(\var_{IJ}\wedge d\ov z^I)\wedge dt_m\wedge dt^J.
\end{align*}
Also $\pd_{t_m}P_i(\var_{IJ} d\ov z^I)\wedge dt_m=P_i(\pd_{t_m}\var_{IJ} d\ov z^I\wedge dt_m)$. Thus
\ga\label{d0p}
d^0P_i\var=-P_id^0\var \quad i>0; \quad \db R_{q-i}\var=-R_{q+1-i}\db\var, \quad i<q.
\end{gather}
We now derive the following homotopy formulae for $\cL D$.
\le{dpht} Let $1\leq q\leq N$. Let $D_\rho=B_\rho^{2n}\times B_\rho^M$. Let $\var$ be a mixed $(0,q)$ form in $\ov D_\rho$ of class $\cL C^{1,0}(\ov{D_\rho})$.
Then we have two homotopy formulae
\ga\label{vdhv}
\var=\cL DT_q\var+T_{q+1}\cL D \var,\\
\label{vdhv+}
\var=\cL D \tilde T_q\var+\tilde T_{q+1}\cL D \var,\\
\label{hvar}
T_q\var(z,t)=P_q[\var]_q(\cdot,0)(z)+\sum_{i<q}R_{q-i}[\var]_i(z,\cdot)(t) ,\\
\tilde T\var=R_qB_0^1[\var]_0+\sum_{i>0}P_{i}[\var]_i.
\label{hvar+}
\end{gather}
\ele
\begin{proof} The formula \re{hvar+} is derived in Treves~\ci{Tr92}*{VI.7.12, p.~294} for  $q\leq n$, while \re{hvar}  is stated in~\ci{Tr92}*{VI.7.13, p.~294} for $q\leq m$. For the convenience of the reader, we derive them for all $q$.


Let us start with the integral representation of $[\var]_0$ in $|z|<r$.  Since $\var$ has total degree $q$,
then $\deg_x[\var]_0=q>0$.
We apply the Bochner-Martinelli formula for functions and the Poincar\'e lemma for $d^0$ to obtain
\eq{bmkp}
[\var]_0=B_0^1[\var]_0+P_1\db[\var]_0=(d^0R_qB_0^1[\var]_0+R_{q+1}d^0B^1_0[\var]_0)+P_1\db[\var]_0.
\eeq
Since $\db_zB_0^1(\zeta,z)=0$, we have
$
d^0R_qB_0^1[\var]_0=\cL DR_qB_0^1[\var]_0$.
Combining with $d^0[\var]_0=[\cL D\var]_0$, we express \re{bmkp} as
 \eq{bmkp+}
[\var]_0= \cL DR_qB_0^1[\var]_0+R_{q+1}B^1_0[\cL D\var]_0+P_1\db [\var]_0.
\eeq

Assume that $j>0$. By the  Koppelman-Leray formula, we obtain
\eq{varj1}
[\var]_j=\db P_j[\var]_j+P_{j+1}\db[\var]_j=\cL DP_j[\var]_j-d^0P_j[\var]_j+P_{j+1}\db[\var]_j.
\eeq
 By \re{d0p} and $d^0[\var]_j=[\cL D \var]_j-\db[\var]_{j-1}$, we obtain
\aln
\sum_{j>0}\left(-d^0P_j[\var]_j+P_{j+1}\db[\var]_j\right)&=\sum_{j>0}\left(P_j[\cL D \var]_j -P_{j}\db[\var]_{j-1}+P_{j+1}\db[\var]_j \right)\\
&= -P_{1}\db[\var]_{0}+\sum_{j>0}P_j[\cL D \var]_j.
\end{align*}
 Here we have used $P_{n+1}=0$.
Using \re{bmkp+} and \re{varj1}, we obtain
\aln
\var=\cL DR_qB_0^1[\var]_0+R_{q+1}B^1_0[\cL D \var]_0+\sum_{j>0}\cL D  P_j[\var]_j+\sum_{j>0} P_j[\cL D \var]_j,
\end{align*}
which gives us \re{vdhv+} and \re{hvar+}.


Suppose that $\var$ has the mixed type $(0,q)$.
 By the Poincar\'e lemma we obtain
\aln
\var&=[\var]_q+\sum_{i<q}(d^0R_{q-i}[\var]_i+R_{q+1-i}d^0[\var]_i)=[\var]_q+\sum_{i<q}\cL DR_{q-i}[\var]_i+R_{q+1-i}\cL D[\var]_i.\end{align*}
Here we have  used $\db R_{q-i}[\var]_i=-R_{q+1-i}\db[\var]_i$ for $i<q$ by \re{d0p}. We can express
$$
\sum_{i<q}R_{q+1-i}\cL D[\var]_i=\sum_{i<q}R_{q+1-i}([d^0\var]_i+[\db\var]_{i+1})= -R_1[d^0\var]_{q}+\sum_{i<q}R_{q+1-i}[\cL D\var]_{i+1},
$$
because $[\db\var]_0=0$. We have
 $$
 [\var]_q(z,t)-R_{1}d^0[\var]_q(z,\cdot)(t)=[\var]_q(z,0).
 $$
 We now apply the  Koppelman-Leray formula to express
 $$
 [\var]_q(\cdot,0)=\db P_q[\var]_q(\cdot,0)+P_{q+1}\db[\var]_q(\cdot,0)=\cL DP_{q}[\var]_q(\cdot,0)+P_{q+1}([\cL D\var]_{q+1}(\cdot, 0)).
  $$
Combining the identities, we get \re{vdhv} and \re{hvar}. 
\end{proof}

We first consider the case that $N=2n$. We  recall some estimates for the homotopy formulae in Webster~\ci{We89}, adapting them for the parametric version.  Note that $P^0_q\var$ are integrals of the
potential-theoretic type
$$
L_0f(x)=\int_{B^{2n}_\rho}\pd_{\xi_i} p(\xi,x)f(\xi)\, dV(\xi),
$$
where $\xi=(\RE\zeta,\IM\zeta)$ and $x=(\RE z,\IM z)$,
 $p(\xi,x)=|\xi-x|^{2-2n}$ and $f$ is the real or imaginary part of the coefficients of $\var$.

We first consider the case  $\rho=1$. The general case will be archived by a dilation. It is a classical result that
$$
|L_0f|_{1;\all}\leq C_\all|f|_{1;0}, \quad 0<\all<1.
$$
Thus, we have
\eq{L0all}
|\{L_0f^\la\}|_{1;\all,\all}\leq C_\all|\{f^\la\}|_{1;0,\all}, \quad 0<\all<1.
\eeq
Using a smooth cut-off function $\chi$ which equals $1$ in a domain $B^{2n}_{1-\theta/2}$. We also assume that $\chi$ has compact support in $B^{2n}_{1-\theta/4}$ and $|\chi|_{r}\leq C_r\theta^{-r}$. Decompose $f=f_0+f_1$
for $f_0=\chi f$.  Then
$$
\|f_0\|_{1;r,s}\leq \f{C_{r}}{\theta^{r}}\|f\|_{1;r,s}.
$$
 We can also write
$$
L_0f_0(x)=\int_{B^{2n}_2}\pd_{\xi_i}p(\xi,x)f_0(\xi)\, dV(\xi).
$$
By a classical estimate for the H\"{o}lder norm (see~Gilbarg-Trudinger~\ci{GT01}*{Lemmas 4.1 and 4.4}),  we obtain
 \gan
 |L_0f_0|_{1; r+1}\leq |L_0f_0|_{2; r+1}\leq C_r|f_0|_{1;r}, \quad r\in(0,\infty)\setminus\nn,\\
 |L_0f_0|_{1; r}\leq |L_0f_0|_{2; r}\leq C_r|f_0|_{1;[r]}, \quad r\geq0.
 \end{gather*}
 Here $C_r$ depends on $1/{\{r\}}$ also.
When $f^\la $ is a family of functions on $D$, we decompose $f^\la =f_0^\la +f_1^\la $ with $f_0^\la =\chi f^\la $.  By the linearity of $L_0$, we can also estimate derivatives in $\la$ easily. We have
 \eq{L0fs}\nonumber %
 |\{L_0f^\la _0\}|_{1;r+1,s}\leq \f{C_r}{\theta^r}|\{f^\la \}|_{1;r,s}, \quad  s\in\nn, \quad r\in(0,\infty)\setminus\nn,\quad M=0.
 \eeq
 In particular, combing with \re{L0all}, we can obtain
 $$
 |\{L_0f^\la _0\}|_{1;r,s}\leq \f{C_r}{\theta^{[r]}}|\{f^\la \}|_{1;[r],s}, \quad  s\in\nn, \quad r\in[0,\infty),\quad M=0.
 $$
   Without gaining derivatives, it is straightforward that
$$
 |\{L_0f^\la _0\}|_{1;r,s}\leq \f{C_r}{\theta^{[r]}}|\{f^\la \}|_{1;r,s}, \quad    \quad r,s\in[0,\infty),\quad M\geq0.
 $$
 For $L_0f^\la _1$, we differentiate the integrand directly and obtain
$$
|\{L_0f^\la _1\}|_{1-\theta;r+1,s}\leq \f{C_r}{\theta^{r}}|\{f_1^\la \}|_{1;0,s}\leq \f{C_r}{\theta^{r}}|\{f^\la \}|_{1;0,s},\ s\in\nn, \
r\in(0,\infty)\setminus\nn, \  M=0.
$$
Here we have used, for $|z|\leq1-\theta$ and $r\not\in\nn$,
$$
\int_{1-\theta/2<|\zeta|<1}\f{1}{|\zeta-z|^{2n+r}}\, dV(\zeta)\leq   \f{C_r}{\theta^r},
\quad r\in(0,\infty)\setminus\nn.
$$

The boundary term $P_{q-1}^{01}\var$ in the homotopy formula is a sum of
$$
L_1f^\la (z)=\int_{\pd B^{2n}_1} \f{(\zeta_i-z_i)\ov\zeta_jf^\la (\zeta)}{|\zeta-z|^{2(n-1-k)}(\ov\zeta\cdot(\zeta-z))^{k+1}}\, dS(\zeta),\quad 0\leq k\leq n-2,
$$
where $f^\la $ is a coefficient of $\var^\la $, and $dS$ is a $(n,n-1)$-form in $\zeta$ with constant coefficients. When $z\in B^{2n}_1$ and $\zeta\in\pd B^{2n}_1$ are close to a boundary point in $\pd B^{2n}_1$, we may choose local coordinates $s=(s_1,s')\in\rr^{2n-1}$ so that  $|\ov\zeta\cdot(\zeta-z)|\geq (\delta+ |s_1|+|s|^2)/C$ and $|\zeta-z|\geq(|s|+\delta)/C$ for $\del=1-|z|$.   Then an $[r]$-th order derivatives of $L_1f^\la$ is a linear combination of
$$
L_1^K(z)=\int_{\pd B^{2n}_1}\f{(\zeta-z,\ov\zeta-\ov z)^Kp_K(\zeta)f^\la (\zeta)}{|\zeta-z|^{2(n-1-k+ m _1)}(r_\zeta\cdot(\zeta-z))^{k+1+ m _2}}\, dS(\zeta)
$$
where $|K|= m _1+1$ and $[r]= m _1+ m _2$.  Let $g(\zeta,z)$ be the integrand in $L_1^K$. Let $\zeta\in\pd B^{2n}_1$. For $|z|<1-\theta$, we have
\aln
|g(\zeta,z)|&\leq\f{C|f|_{1;0,0}}{|\zeta-z|^{2(n-k)-3+m_1}|r_\zeta\cdot(\zeta-z)|^{k+1+ m _2}}\\
&\leq\f{C\theta^{-[r]}|f|_{1;0,0}}{|\zeta-z|^{2(n-k)-3}|r_\zeta\cdot(\zeta-z)|^{k+1}}\\
&\leq \f{C'\theta^{-[r]}|f|_{1;0,0}}{|s|^{2n-3}(|s_1|+|s'|^2)}.
\end{align*}
Also $\int_0^1\int_0^1\f{r^{2n-3}}{(s_1+r)^{2n-3}(s_1+r^2)}\, ds_1dr<\infty$. This shows that $|L_1|_{1-\theta;[r],s}\leq C\theta^{-[r]}|f|_{1;0,s}$.
For the H\"older ratio in $z$ variables, let $\all=\{r\}>0$ and $z_1,z_0\in B^{2n}_{1-\theta}$. If $|z_1-z_0|>\theta$,
we get $|L_1^K(z_1)-L_1^K(z_0)|\leq (|L_1^K(z_1)|+|L_1^K(z_0)|)|z_1-z_0|^\all/\theta^\all\leq C|z_1-z_0|^\all\theta^{-[r]-\all}|f|_{1;0,0}$.  Assume that $|z_1-z_0|\leq\theta$, and let $z(\la)=(1-\la)z_0+\la z_1$. We still have $1-|z(\la)|\geq\theta$. As we just proved, the gradient $\nabla g(\zeta,z)$ in the $z$-variables satisfies
$$
|\nabla_zg(\zeta,z(\la))|\leq C'\theta^{-[r]-1}|f|_{1;0,0}{|s|^{-2n+3}(|s_1|+|s'|^2)^{-1}}.
$$
We have $|z_2-z_1|\leq|z_1-z_1|^\all\theta^{1-\all}$. Then by the gradient estimate and integration, we obtain
\eq{}\nonumber
|P^{01}_q\var|_{1-\theta;r+1,s}\leq C_{r}\theta^{-r}|\var|_{ 1;0,s}, \quad r\not\in\nn, \quad s\in\nn, \quad M=0.
\eeq
Without gaining derivatives, we can also verify
\eq{}\nonumber
|P^{01}_q\var|_{1-\theta;r,s}\leq C_{r}\theta^{-[r]}|\var|_{ 1;0,s}, \quad r,s\in[0,\infty),\quad M\geq0.
\eeq

We now consider the estimates on $B^{2n}_\rho$ with $\rho<C$. Let $ z=\tau(\tilde z)=\rho\tilde z $. Then $\tau_\rho\colon B^{2n}_1\to B^{2n}_\rho$. For a $(0,q)$ form $\var_q$, we   have
$$
|\tau_\rho^*\var_q|_{\rho';r,s}\leq C_r\rho^q|\var_q|_{\rho'\rho;r,s}, \quad |(\tau_\rho^{-1})^*\tilde\var_q|_{\rho';r,s}\leq C_r\rho^{-q-r}|\tilde\var_q|_{\rho'\rho;r,s},\quad\rho<C.
$$
Note that $\tau_\rho\times\tau_\rho$ preserves   $\Om_q^{0}$ and $\Om_q^{01}$. Now we get
$$
P_{B^{2n}_{\rho}}=(\tau_\rho^{-1})^*P_{B^{2n}_{1}}\tau_\rho^*.
$$
Then we have
\aln
|P_{B^{2n}_\rho}\var_q|_{1-\theta;r+1,s}&\leq C_r\rho^{-q+1-r-1}|P_{B^{2n}_{1}}\tau_\rho^*\var_q|_{1-\theta;r+1,s}
\leq C_r\rho^{-r}\theta^{-r}|\var_q|_{\rho;r,s},\\
|P_{B^{2n}_\rho}\var_q|_{1-\theta;r,s}&\leq C_r\rho^{-q+1-r}|P_{B^{2n}_{1}}\tau_\rho^*\var_q|_{1-\theta;r,s}
\leq C_r\rho^{1-r}\theta^{-[r]}|\var_q|_{\rho;[r],s},
\end{align*}
where the first estimate requires $ s\in\nn $ and $M=0$. In summary we have obtain, for $0<\rho<C$,
\eq{}\label{3p7}
|P_q\var|_{ {(1-\theta)\rho};r+1,s}\leq C_{r}(\rho\theta)^{-r}|\var|_{ \rho;r,s},\quad
r\not\in\nn, \  s\in\nn, \ M=0.
\eeq
The above estimate gains one derivative in $z$-variables. Without the gain, we obtain for $0<\rho<2$
\eq{}\label{3p70}
|P_q\var|_{(1-\theta)\rho;r,s}\leq C_{r}\rho^{1-r} \theta^{-[r]}|\var|_{\rho;[r],s},\quad
r,s\in[0,\infty).
\eeq

 Note that the decomposition $\var=\sum[\var]_j$ does not commute with $\cL D$.
 Obviously, the decomposition  respects the estimates; namely
 $$
 |\var|_{\rho;r,s}=\max_j|[\var]_j|_{\rho;r,s}.
 $$
 Therefore, we also have
 $$
 |T\var|_{\rho;r,s}\leq\sum_j|P_j[\var]_j|_{\rho;r,s}+|R_qB_0^1[\var]_0|_{\rho;r,s}.
 $$
 Combining two estimates \re{3p70} and \re{3p13}, we have the following.
 \pr{}
Let $D_\rho=B_\rho^{2n}\times B_\rho^M\subset\cc^n\times\rr^M$. Let $\var$ be a differential form on
$\ov D_\rho$ of mixed type $(0,q)$ with $1\leq q\leq n+m$. Let  $\var\in\cL C_*^{r,s}({\ov D_\rho})$. We have
\begin{gather}\nonumber
|T_q\var|_{D_{(1-\theta)\rho};r+1,s}\leq
C_{r}(\rho \theta)^{-r} |\var|_{D_\rho;r,s},\ s\in\nn,\ r\in(0,\infty)\setminus\nn,\ M=0;\\
\label{Pqva}
|T_q\var|_{D_{(1-\theta)\rho};r,s}\leq
C_{r}\rho^{1-r}\theta^{-[r]} |\var|_{D_\rho;r,s},\quad r,s\in[0,\infty),\quad M\geq0.
 \end{gather}
Assume further that  $\cL D\var=0$. Then $\cL DT_q\var=\var$.
\end{prop}

By estimating in the $C_*^{r,s}$ norm on $B_{\rho}^{2n}\times B_\rho^M$, we have taken the advantage that the chain rule is not used in
the construction of the homotopy formula.

%
%
%

\setcounter{thm}{0}\setcounter{equation}{0}
\section{The proof    for a family of complex structures
}\label{sec5}

In this section, we will present our first rapid iteration proof for the special case. The arguments
do not involve the Nash-Moser smoothing methods, due to the gain of a full
derivative in the estimates for the Koppelman-Leray homotopy formula. However, introducing parameter requires
an additional argument to obtain rapid iteration in higher order derivatives via the rapid convergence
in low order derivatives. This include the $\mathcal C^\infty$ case. We will use the rapid methods from Gong-Webster~\cites{GW10, GW12}.
%

We will apply rapid iteration methods several times to prove our results. Therefore, it will be convenient to
prove a general statement for the rapid iteration. 
\begin{defn}\label{def6.1}
Following Webster~\ci{We91} we say that the sequence $ L_j$ grows  $($at most$)$ {\it linearly} if
\eq{linear-gr}
0\leq L_{j}\leq e^{P(j)}, \quad j\geq0
\eeq
for some polynomial $P$ of non-negative coefficients.
We say that $e_j$ {\it converges rapidly} $($to zero$)$, if
\eq{rapid-decay}
0\leq e_{j}\leq \hat e_0 b^{\kappa^j-1}, \quad \kappa>1, \quad 0< b<1, \quad j\geq0.
\eeq
Here $\hat e_0$ is a finite number to be adjusted.
\end{defn}
 A precise notion is that the $L_j$ grows at most {\it exponentially} and the $e_j$ converges to $0$ {\it double} exponentially.
Fix a positive constant $a$. Obviously, if $L_j$ grows linearly,  so does $L_j^a$.
And $e_j^a$ converges rapidly with $e_j$.
\le{defrapid} Let $P$ be a polynomial in $t\in\rr$ of non-negative coefficients and suppose that the sequence   $L_j$ of numbers satisfies \rea{linear-gr}. Assume that the numbers $e_j,\kappa,b$ satisfy \rea{rapid-decay}. \bppp
\item Let $Q$ be a real polynomial in $t\in\rr$. There exists $\hat e_0'>0$, which depends only on $\kappa,b, Q$,  so that if the $\hat e_0$ in \rea{rapid-decay} is adjusted to satisfy  $\hat e_0\leq \hat e'_0$, then the sequence $e_j$ satisfies
\eq{}\nonumber
e_j\leq e^{-Q(j)}, \quad j\geq0.
\eeq
$($Note that $e_0\leq\hat e_0$ by \rea{rapid-decay} and consequently $e_0\leq \hat e_0'$ too.$)$
\item Suppose that $M_j$ is a sequence of non-negative numbers satisfying
\eq{}\nonumber
M_{j+1}\leq L_jM_j, \quad j\geq0.
\eeq
Then $M_{j+1}\leq M_0 e^{P_1(j)}$ for $P_1(j)=P(0)+\dots+P(j)$ when $j\geq0$, where   $P$  is in \rea{linear-gr}. Set $P(-1):=0$. In particular, $M_j$ grows linearly.
\item Suppose that $a_j$ is a sequence of non-negative   numbers such that
\eq{}\nonumber
0\leq a_{j+1}\leq L_je_ja_j.
\eeq
Let $b_1>b^{1/{(\kappa-1)}}$ for $b$   in \rea{rapid-decay}.  There is a constant $C$, depending only on $\hat e_0,b,b_1$, and $P$, such that
\eq{}\nonumber
a_{j}\leq Ca_0b_1^{\kappa^j}.
\eeq
In particular $a_j$ converges to zero rapidly.
\eppp
\ele
\begin{proof} (i). Obviously, there exists $j_0$ so that $b^{\kappa^j-1}<e^{-Q(j)}$ for $j>j_0$. Choose $0<\hat e'_0<1$ sufficiently small so that
$
\hat e'_0b^{\kappa^j-1}\leq e^{-Q(j)}
$
for $0\leq j\leq j_0$.  We get (i).
For (ii) it is immediate that $M_{j+1}\leq M_0e^{P_1(j)}$. Since
$$
1+2^d+\dots+j^d\leq \f{1}{d+1}(j+1)^{d+1}
$$
then $P_1(j)\leq Q(j)$, $j\in\nn$,  for a polynomial $Q$ of degree $d+1$.

(iii) We have
$$ a_{j+1}\leq a_0 \hat e_0^{j}\prod_{i=0}^{j-1}(b^{\kappa^i-1}e^{P(i)})\leq a_0\hat e_0^{j} e^{P_1(j-1)}b^{\frac{\kappa^{j}-1}{\kappa-1}-j}.
$$
Since $b_1>b^{1/{(\kappa-1)}}$, we can find $j_0$ so that for $j> j_0$,
$$
 \hat e_0^{j} e^{P_1(j)}b^{\frac{\kappa^{j}-1}{\kappa-1}-j}<b_1^{\kappa^j}.
$$
Take $C>1$ so that the above left side is less than $Cb_1^{\kappa^j}$ for $0\leq j\leq j_0$.
\end{proof}

The following is one of main ingredients in the KAM rapid iteration procedures.
\pr{KAM}  Let $K_j$ be a  sequence  of numbers satisfying
$$
0\leq K_j\leq e^{P(j)}, \quad j\geq0
$$
where $P$ is a polynomial of non-negative coefficients.
  Suppose that the sequence $a_i$  satisfies  \eq{aj+1-} 
a_{j+1}\leq K_j(a_j^{\kappa }+a_j^{\kappa_1}), \quad a_j\geq0, \quad j\geq0.
\eeq
Suppose that $\kappa_1>\kappa>1$. Set $C_\infty^*=\lim_{j\to+\infty} C_j^*$ for
$$
 C_j^*:=\exp\Bigl\{ \kappa^{-1}(\ln 2+K(0))+\sum_{ m =1}^j\kappa^{- m }(\ln 2+ P( m ))\Bigr\}, \quad 0\leq j<\infty.
 $$
 If $j\geq0$ and $0\leq \cL C^*_j a_0\leq 1$, then
\ga\label{aj1cs}
  a_{j+1}\leq (\cL C^*_ja_0)^{\kappa^j}.
\end{gather}
In particular, if $0\leq \cL C^*_\infty a_0\leq1$, then $a_{j+1}\leq(C_\infty^*a_0)^{\kappa^j}$.
\epr
\begin{proof} Note that $C_j^*$ increases with $j$ and $\cL C^*_j\geq1$, while $C_\infty^*<\infty$ since
$$
\sum_{m=1}^\infty\kappa^{-m}m^d\leq\int_1^\infty \f{t^d}{\kappa^t}\, dt<\infty.
$$
 Suppose that $0\leq \cL C^*_ja_0\leq1$. Then $a_0\in[0,1]$ and \re{aj+1-} imply that
$$
a_1\leq 2e^{P(0)}a_0^\kappa=(\cL C^*_0a_0)^{\kappa}.
$$
 Suppose that  \re{aj1cs} is verified for $j< m $.
This implies that $a_ m \leq1$. By \re{aj+1-} we get
\begin{equation*}
a_{ m +1}\leq 2K_ m  a_ m ^{\kappa}\leq 2e^{P( m )}(\cL C^*_{ m -1}a_0)^{\kappa^{ m }}=(\cL C^*_ m  a_0)^{\kappa^{ m }},
\end{equation*}
where the last identity follows from the definition of $C_m^*$.
\end{proof}

\medskip
\begin{proof}[Proof of \rpa{pnl}]  Let us recall the proposition.
We are in the special case of a family of complex structures on $\cc^n$, defined near the origin
 by the adapted vector fields
\ga\label{2p1}
\quad Z^\la _{\ov\all}=\pd_{\ov\all}+A_{\ov\all \beta}^\la (z)\pd_\beta, \quad
1\leq\all\leq n, \quad A^\la(0)=0, \quad |A^\la|<1/{C_0}.
\end{gather}
  We want to find a family of transformations
 \eq{}\nonumber
 F^\la\colon \hat
z=z+f^\la (z),\quad   f^\la (0)=0, \quad |\pd_z^1f|<1/{C_0}
\eeq
which transforms the complex structures into the standard one in $\cc^n$.

Let us assume that $n\geq2$. When $n=1$ the following proof with simplifications remains valid; see \rrem{1drem} following the proof.

The vector fields \re{2p1} are adapted so that   the integrability condition has a simple form
\al\label{2p2}
[Z^\la _{\ov\all},Z^\la _{\ov\beta}]&=
(Z^\la _{\ov\all}A^\la _{\ov\beta \gaa}-Z^\la _{\ov\beta}A^\la _{\ov\all \gaa})\pd_\gaa=0.
\end{align}

To simplify notation, let us drop $\la$ in $A^\la, f^\la$, etc.
We will use some abbreviations. If $A,B$ are differential forms,   $\langle A,B\rangle$ denotes a differential form of which the coefficients are finite sums of $ab$ with $a$ (resp. $b$) being a coefficient of $A$ (resp. $B$). We will also further abbreviate $\langle A,B\rangle$ by $AB$ sometimes.
For  $\{A^\la\}, \{B^\la\}$, the $\langle A,B\rangle$ denotes the family $\{\langle A^\la,B^\la\rangle\}$, while $AB$ denotes the family $\{A^\la B^\la\}$.

Following \ci{We89} we  form the following differential forms by using the coefficients of
the adapted vector fields. With an abuse of notation, define
$$
A_\beta=A_{\ov\all \beta} d\ov z_{\all},\quad 1\leq\beta\leq n.
$$
 Then   the integrability condition \re{2p2} takes the form
\ga\label{2p3}
\db A =\langle A,\pd A\rangle.
\end{gather}  Let us  compute the new vector fields after a change of coordinates by $F$.  We have
 \al\nonumber
F_*Z_{\ov\all}&=
(\del_{\ov\all \ov\beta}+A_{\ov\all \gaa}\pd_\gaa\ov  f_ {\beta})
\pd_{\ov{\hat z}_\beta}+
(\pd_{\ov\all}f_\beta+
A_{\ov\all \beta}+A_{\ov\all \gaa}\pd_\gaa f_\beta)\pd_{\hat z_\beta}.
\end{align}
The new
adapted frame for $F_*S$   still has the form
$$
\widehat Z_{\ov\all}=\pd_{\ov\all}+ \widehat A_{\ov\all \beta}\pd_\beta
$$ where the new coefficients $\widehat{A}$ can be computed via
\ga\label{2p6}
C_{\ov\all \ov\beta}\widehat Z_{\ov\beta}=F_*Z_{\ov\all},\quad C_{\ov\all \ov\beta}\circ F
=\delta_{\ov\all\ov \beta}+ A_{\ov\all \gaa}\pd_\gaa \ov f_\beta,\\
\nonumber
(C_{\ov\all \ov\gaa}\widehat A_{\ov\gaa \beta})\circ F=\pd_{\ov\all}f_\beta+
A_{\ov\all \beta}+ A_{\ov\all \gaa}\pd_\gaa f_\beta.
\end{gather}
  Here $A$ denotes $(A_1,\ldots, A_n)$. Then $F$  transforms $\{Z_{\ov\all}\}$ into the span of $\{\pd_{\ov\beta}\}$ if and only if
$\{\widehat A_\gamma\}$  are zero, i.e.
\ga \label{2p7}
\db f+A+\langle A,\pd f\rangle=0.
\end{gather}
If $\widehat A$ is not zero,  let us  express it in terms of $A$. For $\hat z=F(z)$ and $\widehat A_\beta=\widehat A_{\ov\all\beta}d\ov{\hat z_\all}$, it is convenient to use
$$
\widehat A_\beta\circ F:=\widehat A_{\ov\all \beta}\circ Fd\ov  z_\all.
$$
Then the new coefficients $\widehat A$ are given by
\ga\label{2p8}
(C\widehat A)\circ F=\db f+A+\langle A,\pd f\rangle.
\end{gather}

While thinking of solving  $f$ for the non-linear equation \re{2p7},  we use the homotopy formula to find an approximate solution. We then iterate to solve the equation eventually.
For the approximate solution, we  apply  the homotopy formula as in~\ci{We89}.
On $B_\rho^{2n}$,  we have
the homotopy formula
 \ga \nonumber
\var =\db P_{1}\var
 +P_{2}\db \var.
\end{gather}
To find approximate solutions to \re{2p7}, we take
\ga\label{2p11}
f_\beta=-P_1A_\beta+(P_1A_\beta)(0).
\end{gather}
By the homotopy formula, where $\var$ is a component of   $A$,  and the integrability condition \re{2p3} we write
$
\db f+A+\langle A,\pd f\rangle$ as $P_2 \langle A,\pd A\rangle+\langle A,\pd f\rangle.
$
Thus \re{2p8} becomes
\al \label{2p12}
(C\widehat A)\circ F=\langle A,\pd f\rangle+P_2\langle A,\pd A\rangle.
\end{align}
Without achieving $\widehat A=0$, we see that $\widehat A$ is written as products of two small functions. Hence, \re{2p11} is a good approximate solution.  We now estimate   $\widehat A$.


Recall that \ga \nonumber %
|f|_{D;a,0}\odot |g|_{D';0,b}:=|f|_{D;a,0}|g|_{D';0,[b]}+
 |f|_{D;[a],0}|g|_{D';0,b},\\  \nonumber 
Q^*_{D,D';a,b}(f,g) := |f|_{D;a,0}\odot |g|_{D';0,b}+|g|_{D';a,0}\odot |f|_{D;0,b}, \\
\nonumber 
Q_{D,D';r,s}(f,g):=\sum_{k=0}^{\mfl{s}}Q^*_{D,D';r-j,j+\{s\}}(f,g),\quad [r]\geq [s],\\
\hat Q_{D,D';r,s}(f,g):=\|f\|_{D;r,s}|g|_{D';0,0}+
|f|_{D;0,0}\|g\|_{D';r,s}+Q_{D,D';r,s}(f,g).
\label{Aqrsfg+m}\end{gather}
By \rp{phifi} we have  the product rule
\ga\label{recallq}
\left\|\left\{\prod_{i=1}^m f_i^\la \right\} \right\|_{a,b}\ple\sum_{i\leq m}\|f_i\|_{a,b}\prod_{ \ell \neq i}|f_ \ell |_{0,0}+
\sum_{i<j}Q_{a,b}(f_i,f_j)\prod_{ \ell\neq i,j}|f_ \ell |_{0,0},\\
\label{recallq+}\|\prod_{i=1}^{m-1} \phi_i(f_i^\la  )\|_{a,b}\ple\sum_{i<m}\|f_i\|_{a,b} \prod_{ \ell \neq i}|f_ \ell |_{0,0} +\sum_{i\leq j}Q_{a,b}(f_i,f_j)\prod_{ \ell\neq i,j}|f_ \ell |_{0,0},
\end{gather}
for $m\geq 2, a\geq b$, and $\phi_i$ are $\cL C^{[a]+1}$ functions satisfying $\phi_i(0)=0$.  Here $Q_{a,b}$ is defined by \re{Aqrsfg} and
 for the rest of the paper, by $A\ple B$ we mean that $A\leq CB$ for some constant $C$.   We will simply  use
$$
Q_{a,b}(f,g)\ple |f|_{a,0}\odot |g|_{0,b}+|f|_{0,b}\odot |g|_{a,0}. 
$$
In particular, we have
\ga
\label{qest2}
Q_{a,b}(f,g)\ple |f|_{a,0}|g|_{0,b}+|f|_{0,b} |g|_{a,0}. 
\end{gather}


Let us derive estimates before we apply the rapid iteration.  We first assume that  $s\in\nn$ and
\eq{}\nonumber
\infty>r>s+1, \quad \{r\}\neq0.
\eeq
  Define $s_*=0$ for $s=0$, and $s_*=1$ for $s\geq1$. Define
\gan
r_0=s_*+1+\{r\}, \quad r_1=s+1+\{r\},\\
\tilde\rho_i=(1-\theta)^i\rho,  \quad i=1,2,3.
\end{gather*}
We first use  \re{3p7} to estimate the approximate solution $f$:
\begin{gather}\label{4p4}
|f|_{{\tilde\rho_1}; m  +1+\{r\},s}\leq K( m +2)|A|_{{\rho}; m +\{r\},s}.
\end{gather}
Here and in the following $K(a)$ denotes a constant of the form
\eq{Kadef}
K(a):=K(a,\theta)=C_a\theta^{-a}.
\eeq
For the  rest of the proof,
 the norms of $ A$ and its derivatives are computed on $B^{2n}_\rho$,
the norms of  $f$ for $F=I+f$ and its derivatives are computed on the domain $B^{2n}_{\rho_1}$, and
the norms of  $g$ for $G:=F^{-1}=I+g$ and its derivatives are computed on the domain $B^{2n}_{\rho_2}$. We will abbreviate
$\| A\|_{a,b}=\| A\|_{\rho;a,b}$. Thus, all norms of $A$ are on the domain $B^{2n}_\rho$.

Recall  that $\widehat Q_{r,s}(A,A)$ is defined by \re{Aqrsfg+m} via the
$\odot$ product. We define
\eq{} \nonumber %
Q_{r,s}^{**}( A, A):=\| A\|_{r_1,s}| A|_{r,0}+\| A\|_{ r_0,s_*}\| A\|_{  r,s},
\eeq
which   is more suitable to  use  \re{4p4} which gains one derivative.

We want to estimate $\hat A$ by its formula \re{2p12}, which is rewritten as
\eq{4p3noi}
 (C\widehat  A)\circ F= A\pd f+P_2( A\pd A).
\eeq
 Let us first estimate various quadratic terms $Q_{r,s}$. By \re{recallq} or  \re{qr-1s}, we have
\eq{qada} \nonumber %
Q_{\rho,\rho;r-1,s}(A,\pd A)\ple\widehat Q_{r,s}(A,A)\ple Q^{**}_{r,s}(A,A).
\eeq
From \re{4p4}, we deduce two estimates
\gan |f|_{\tilde\rho_1;1,s}\leq|f|_{\tilde\rho_1;{3}/{2},s} \leq K(2)\|A\|_{s+1,s},\\
|f|_{\tilde\rho_1;r+1,0} \leq K(r+2)|A|_{r,0}.
\end{gather*}
We use the last three inequalities to estimate the following quadratic terms:
\ga
\label{qadf}
Q_{\rho,\tilde\rho_1;r,s}(A,\pd f)\ple|A|_{r,0}|f|_{\tilde\rho_1;1,s}+|A|_{0,s}|f|_{\tilde\rho_1;r+1,0}\leq K(r+2)Q^{**}_{r,s}(A,A),\\
\label{qdfdf}
Q_{\tilde\rho_1,\tilde\rho_1;r,s}(\pd f,\pd f)\ple|\pd f|_{\tilde\rho_1;r,0}|f|_{\tilde\rho_1;1,s} \leq
K(2)K(r+2)Q^{**}_{r,s}(A,A).
\end{gather}
  Using  the product rule \re{recallq}  and \re{qadf}, we verify
\al
\label{4p5}\|A\pd \ov f\|_{{\tilde\rho_1};r,s}&\ple\|A\|_{r,s}|\pd \ov f|_{\tilde\rho_1;0,0}+ |A|_{0,0}\|\pd\ov f\|_{\tilde\rho_1;r,s}+ Q_{\rho,\tilde\rho_1;r,s}(A,\pd \ov f)\\
&\leq  K(r+2)Q_{r,s}^{**}(A,A), \nonumber\\
\label{4p5+}\| A\pd  A\|_{{\tilde\rho_1};r-1,s}&
 \ple  Q_{r,s}^{**}(A,A).
\end{align}

Recall that $K(a)=K(a,\theta)=C_a\theta^{-a}$.  Let
\eq{dfK0}
 K_0(r_0,\theta):={K(r_0+2,\theta)}\cdot\f{C_{2n}}{\theta}.
\eeq
To estimate the inverse of $F^\la =I+f^\la $, let us assume that
\eq{4p6} 
\|A\|_{{\rho};r_0,s_*}\leq  K_0^{-1}(r_0,\theta).
\eeq
Note that this implies that $Q_{r_0,s_*}( A, A)\leq 1/{C_n}$. Then by \re{4p4} and \re{4p5},  we have
\eq{4p6+} \nonumber %
\|f\|_{\tilde\rho_1; r_0+1,s_*}\leq \theta/{C_{2n}}.
\eeq
Thus  we can use \rp{Agrsfp} or \rl{fg} to estimate $F$ and its   inverse map. For $1/4<\rho<2$, we have
\ga\nonumber
F^\la \colon B^{2n}_{\tilde\rho_i}\to B^{2n}_{\tilde\rho_{i-1}}, \quad i=1,2;\\ G^\la \colon B^{2n}_{ \tilde\rho_{2}}
\to B^{2n}_{ \tilde\rho_1},\quad
\nonumber F^\la \circ G^\la =I \quad\text{on
$B^{2n}_{ \tilde\rho_2}$}.
\end{gather}


By \re{4p3noi} we get  $(C\widehat A)\circ F= A\pd f+T( A\pd A)$. Then \re{4p5}-\re{4p5+} yield
\al
\label{4p10+}
\|(C\widehat A)\circ F\|_{{\tilde\rho_1};r,s}&\leq K(r+2) (\|A\pd f\|_{{\tilde\rho_1};r,s}
+  \|A\pd A\|_{{\tilde\rho_1};r-1,s})\\
\nonumber&\ple K(r+2)Q_{r,s}^{**}(A,A).
\end{align}
 By \re{2p6}, $C\circ F=I+ A\pd\ov f$. By  the product rule \re{recallq+}  and  \re{qadf}-\re{4p5},  we have
\al \label{4p11}
\|C^{-1}\circ F-I\|_{{\tilde\rho_1};r,s}&\ple    \|A\pd \ov f\|_{\tilde\rho_1;r,s}+Q_{r,s}(A\pd\ov f, A\pd\ov f)  \\
&\leq K( 2)K(r+2)Q_{r,s}^{**}(A,A).\nonumber
 \end{align}
To estimate a new quadratic term, we express \re{4p10+}-\re{4p11} for a special case:
\al
\label{c-1r}
|(C^{-1}\circ F-I|_{\tilde\rho_1;r,0}&=|(I-A\pd \ov f)^{-1}(A\pd \ov f)|_{\tilde\rho_1;r,0}\\
& \leq K(2)K(r+2)|A|_{1,0}|A|_{r,0}, \nonumber\\
\label{c-1r1}|(C\widehat A)\circ F|_{\tilde\rho_1;r,0}&=| A\pd f+P_2( A\pd A)|_{\tilde\rho_1;r,0}\\
\nonumber &\leq K(r+2)|A|_{1,0}|A|_{r,0}.
\end{align}
Also, a direct estimation by \re{4p4} yields
\al
\label{c-1r2}|(C^{-1}\circ F-I|_{\tilde\rho_1;0,s}&=|(I-A\pd \ov f)^{-1}(A\pd \ov f)|_{\tilde\rho_1;0,s}\\
& \nonumber
\leq K^2(2) |A|_{1,s}|A|_{1 ,0},\\
\label{c-1r3}|(C\widehat A)\circ F|_{\tilde\rho_1;0,s}&= | A\pd f+P_2( A\pd A)|_{\tilde\rho_1;0,s}\\
&\leq   K(2) |A|_{1,s}|A|_{1 ,0}.\nonumber
\end{align}
Since $|A|_{r_0,s_*}\leq1$ then from \re{c-1r}-\re{c-1r3} we get a   quadratic estimate
 \ga\label{640} \nonumber %
 Q_{\tilde\rho_1,\tilde\rho_1;r,s}((C^{-1}\circ F-I), (C  \widehat A )\circ F)\leq  K^3(2)K(r+2)Q_{r,s}^{**}(A,A).
 \end{gather}
By $ \widehat A\circ F =(C  \widehat A )\circ F +(C^{-1}\circ F -I) (C  \widehat A )\circ F$   and
 the product rule \re{recallq}, we get
 \begin{align}\label{hacf}
 \| \widehat A \circ F\|_{\tilde\rho_1,\tilde\rho_1;r,s} \leq  K^3(2)K(r+2)Q_{r,s}^{**}(A,A).
 \end{align}
 Set $K_1(r)=K^3(2)K(r+2)$.
By \re{hacf}, we have
\aln
\|\widehat A\circ F\|_{1, s_*} &\leq
K_1(r_0)Q^{**}_{r_0,s_*}(A,A)\leq 2 K_1(r_0)\|A\|^2_{r_0,s_*},\\
\|\widehat A\circ F\|_{s+1, s} &\leq
K_1(r_1)Q^{**}_{r_1,s}(A,A)\leq 2 K_1(r_1)\|A\|_{r_1,s_*}\|A\|_{r_1,s},\\
\|\widehat A\circ F\|_{r, s_*} &\leq
K_1(r)Q^{**}_{r,s_*}(A,A)\leq 2K_1(r_0)\|A\|_{r_0,s_*}\|A\|_{r,s_*}.
\end{align*}
By \re{4p4}, we have $\|f\|_{\tilde \rho_1,r,s}\leq K(r)\|A\|_{r,s}$.
Since $|f |_{\tilde\rho_1;1,0}\leq 1/C_N$ and $\|f\|_{\tilde\rho_1;0,s_*}\leq1$,  by \re{AuF-1}
we have a general estimate
\aln 
&\|u\circ F^{-1}\|_{\tilde\rho_2;r,s}\ple \left\{\|u\|_{\rho;r,s}+|u|_{\rho;1,s_*} \|f\|_{\rho;r,s}
\right.\\
  &\quad \left.+\|u\|_{s+1,s}\odot\|f\|_{r,s_*}+\|u\|_{r,s_*}\odot\|f\|_{s+1,s}+|u|_{\rho;1,0}\|f\|_{s+1,s}\odot\|f\|_{r,s_*} \right\}.
\nonumber
\end{align*}
Applying it to $u= A\circ F$,
we obtain   
 \al  \label{hAAf}
 |\widehat A|_{{\tilde\rho}_2;r,s}
 &\ple  \| \widehat A \circ F\|_{\tilde\rho_1,\tilde\rho_1;r,s} +\|A\|_{  \rho;r_0,s_*}\|A\|_{ \tilde\rho_1;r,s}
 +2\|A\|_{r_1,s_*}\|A\|_{r_1,s}\|A\|_{r,s_*}\\  &\quad
 + \|A\|_{r_1,s}\|A\|_{r,s_*}\nonumber.
 \end{align}
 Using \re{hacf}, from \re{hAAf} it follows
 \ga
 \label{hAtQ}
 \|\widehat A\|_{ \tilde\rho_2;r,s} \leq  C_r  K_1(r)   \left(Q^{**}_{\rho;r,s}(A,A)+\|A\|_{r_1,s_*}\|A\|_{r_1,s}\|A\|_{r,s_*}\right).
 \end{gather}

\medskip
\noindent
{\bf Case 1. $r>s+1$ and $r\not\in\nn$.}
We need to iterate the above construction and estimates to obtain a sequence of transformation $F_j$
so that $\tilde F_i^\la :=F^\la _i\circ\dots\circ F^\la _0$ converges to $\tilde F_\infty^\la $, while $S_{i+1}^\la :=(\tilde F_i)_*S_i^\la $, with $S_0^\la $ being the original structure, converges to the standard complex structure on $\cc^n$. More precisely, the coefficients $ A_{i}^\la $ of the adapted frame of $Z_{i}^\la $
tend to $0$ as $i\to\infty$.
We apply \rl{Ask2k} to nest the domains.
 Let us first defined a domain on which the sequence is well-defined. Set for $i=0,1,\dots$
\gan
\rho_i=\frac{1}{2}+\frac{1}{2(i+1)}, \quad
\rho_{i+1}=(1-\theta_i)^2\rho_i.
\end{gather*}
Note that $\rho_i$ decreases to $\rho_\infty=1/2$.
 We have
for $\rho_i/4<\rho<4\rho_i$, especially for $\rho_0/2<\rho<2\rho_0$
\eq{nestF}
F_i^\la\colon B^{2n}_{(1-\theta_i) \rho}\to B^{2n}_{\rho},\quad G^\la_i\colon B^{2n}_{(1-\theta_i)^2\rho}\to B^{2n}_{(1-\theta_i) \rho}
\eeq
provided \re{4p6} holds for $ A_i$, i.e.
\eq{4p6j-}
\| A_i\|_{{\rho_i},r_0,s_*}\leq  K_0^{-1}(r_0,\theta_i).
\eeq
By \rl{Ask2k} in which $1-\theta_k^*=(1-\theta_k)^2$, we have $\tilde F_i^\la\colon B_{\rho_\infty/2}\to B_{\rho_\infty}$.
Furthermore, in addition to \re{nestF}, we express  \re{hAtQ}, in which $\hat  A= A_{i+1}$ and $ A= A_i$, in a simple form
 \eq{Ai1i}
\|   A_{i+1}\|_{\rho_{i+1};r,s}\leq L_i(\|  A_i\|_{\rho_i;r,s}^2+\|  A_i\|_{\rho_i;r,s}^3), \quad L_i:=C_rK_1(r,\theta_i).
\eeq
Therefore, for the sequence $F_i$ to be defined, we need to achieve \re{4p6j-} for $i=0,1,2,\dots$.
Note that  $K_0(r_0,\theta_i)$ and $L_i$ have linear growth as $i\to\infty$. Thus
\eq{KLj}
K_0(r_0,\theta_i)+L_i\leq e^{P(i)}
\eeq
for some polynomial $P$ of positive coefficients. Let constant $C_\infty^*$ be defined in \rp{KAM} in which $\kappa=2$ and $\kappa_1=3$.  By \rd{def6.1} and \rl{defrapid}, there is $\hat e_0>0$ so that
$$
(C_\infty^*\hat e_0)^{\kappa^i}\leq e^{-P(i)}\leq K_0^{-1}(r_0,\theta_i).
$$
We may assume that $2C_\infty^*\hat e_0\leq 1$. By \rp{KAM} in which $\kappa=2$ and \re{Ai1i}-\re{KLj},     if
\eq{tau0}
\| A_0\|_{\rho_0;r,s}\leq  \hat e_0,
\eeq
then
\eq{tauj}
\| A_i\|_{\rho_i;r,s}\leq  (C_\infty^*\hat e_0)^{\kappa^i}, \quad i\geq1.
\eeq
In particular, \re{4p6j-} holds for every $i$. On the other hand, the initial condition \re{tau0} can be achieved by dilation and initial normalization in \rl{intialcod}.

We now consider the convergence of the sequence $F_i\circ\cdots\circ F_0$.
The mapping $F^\la_i$ transforms $S_i^\la $ into $S_{i+1}^\la $.
Thus the structure $S_i^\la $ is defined on $B^{2n}_{\rho_i}$.
Then $\tilde F_i^\la \colon B^{2n}_{\rho_\infty/2}\to B^{2n}_{\rho_\infty}$  transforms $S_0^\la $, restricted to $B^{2n}_{\rho_\infty/2}$, into $S_{i+1}^\la $.  We have
$$
\tilde F^\la_{i}-\tilde F^\la_{i-1}=f^\la_{i}\circ F^\la_{i-1}\circ\cdots\circ F^\la_0.
$$
By \rp{AuFF}, we have
\al\nonumber 
&\|\{\tilde F^\la_{i}-\tilde F^\la_{i-1} \}
\|_{ \rho_\infty/2;r+1,s}\leq C_r^{i+1}   \left\{\|f_i\|_{r+1,s}+ \|f_i\|_{1,0} \sum_{k\leq j<i} \|f_k\|_{s+1,s} \|f_j\|_{r+1,s_*} \right.\\
&\hspace{5em}+\left.
\sum_{j<i}|f_i|_{1,s_*} \|f_j\|_{r+1,s}
+\|f_i\|_{s+1,s} |f_j\|_{r+1,s_*}+\|f_i\|_{r+1,s_*} \|f_j\|_{s+1,s}\right\}\nonumber
\\
\nonumber
&\quad \leq C_r^{i+1}K_1(r)   \left\{\|A_i\|_{r,s}+ \|A_i\|_{r_0,s_*} \sum_{k\leq j<i} \|A_k\|_{r_1,s} \|A_j\|_{r,s_*} \right.\\
&\hspace{6em}+\left.
\sum_{j<i}\|A_i\|_{r_0,s_*} \|A_i\|_{r,s}
+\|A_i\|_{r_1,s} \|A_j\|_{r,s_*}+\|A_i\|_{r,s_*} \|A_j\|_{r_1,s}\right\}.
\nonumber
\end{align}
Here the last inequality is obtained by \re{4p4}.
We already know the rapid convergence of $\|\A_i\|_{\rho_i;r_0,s_*}$. They are bounded from above by a constant $C_*$. Then we simplify the above to obtain
\al\label{hderF}
\|\{\tilde F^\la_{i}-\tilde F^\la_{i-1}\}\|_{ \rho_\infty/2;r+1,s}
\leq \tilde C_r^{i+1}  K^3(2)K(r+2) \| A_i\|_{\rho_i;r,s}.
\end{align}
The rapid decay of $\| A_i\|_{r,s}$ via \re{tauj} implies the convergence of $\tilde F_i$ to $\tilde F_\infty$ in $\cL C^{r+1,s}( B^{2n}_{\rho_\infty/2})$. Finally,  for each $\la$, the limit of $\tilde F_i^\la $ is a $\cL C^{1}$ diffeomorphism defined in $B^{2n}_{\rho_\infty/2}$, because for the Jacobian matrix $\pd_x \tilde F^i$ of $ x\to\tilde F^i(x)$ with  $x\in B^{2n}_{\rho_\infty/2}$,  its operator norm   satisfies
\aln
 \|\pd_x \tilde F_{i}^\la -I\|&\leq\sum_{j=0}^i\|\{\tilde F^\la_{j}-\tilde F^\la_{j-1}\}\|_{ \f{\rho_{\infty}}{2};r_0+1,s_*}
\leq \sum_{j=0}^\infty\tilde C_{r_0}^{j+1}  K^3(2)K(r_0+2) \| A_j\|_{\rho_j;r_0,s_*}.
 %
\end{align*}
By \re{tau0}-\re{tauj}, we obtain $\|\pd_x \tilde F_{i}^\la -I\|<1/2$, for a possibly smaller $\hat e_0$. This shows that the limit mapping $\tilde F_\infty$ is a diffeomorphism.

\medskip
\noindent
{\bf Case 2. $r=\infty$.}
We first  consider the case where $r=\infty$ and $s$ is finite. We first use \re{hAtQ} for $s= s_*$ and $r=r_0$ to get a rapid decay of $\|A_j\|_{r_0,s_*}$.  Next, we use  \re{hAtQ} in which $\hat  A= A_{i+1}$ and $ A= A_i$ a few times in a bootstrap argument. We first
simplify \re{hAtQ} with $s=s_*$ and $r_1=r_0=1+s+\{r\}$ as
\eq{sai1} \nonumber %
| A_{i+1}|_{\rho_{i+1};r',s_*} \leq C_{r'} K^3(2) K(r'+2) \| A_i\|_{\rho_{i};r_0,s_*}\| A_i\|_{\rho_{i};r',s_*}
\eeq
for any $r'>s+1$ with $\{r'\}=\{r\}$.
By \rl{defrapid} (ii), we get rapid decay of $\| A_i\|_{\rho_i;r',s_*}$.
We can also simplify  \re{hAtQ} as
\eq{Airs}
\|  A_{i+1}\|_{\rho_{i+1};r',s} \leq  C_rK^3(2)K(r'+2)  \| A_i\|_{\rho_{i};r',s_*}\| A_i\|_{\rho_{i};r',s}.
\eeq
By the rapid decay of $\|  A_i\|_{\rho_{i};r',s_*}$ and \rl{defrapid} ({\it ii}), we obtain the rapid convergence
of $\| A_i\|_{\rho_{i};r',s}$.
By \re{hderF} again, we obtain the rapid convergence of $\|\tilde F_{i+1}-\tilde F_i\|_{ \rho_\infty/2;r',s}$.
This shows that $\tilde F_\infty$ is in $\cL C^{\infty,s}$.

Next, we consider the case $s=\infty$. Applying the argument in Case $1$  to $s=1$, $r=5/2$, we obtain rapid convergence of $\|\tilde F_{i+1}-\tilde F_i\|_{ \rho_\infty/2;5/2,1}$.
Next, we use \re{Airs} with $r'=s+3/2$ for any positive integer $s$.
By \rl{defrapid} ({\it iii}), we obtain the rapid convergence of $\| A_i\|_{ \rho_i;s+3/2,s}$. By \re{hderF}, we obtain the rapid convergence of $\|\tilde F_{i+1}-\tilde F_i\|_{ \rho_\infty/2;s+3/2,s}$
for any positive integer $s$. This shows that $\tilde F_\infty$ is in $\cL C^{\infty,\infty}$.

\medskip
\noindent
{\bf Case 3. $3+s\leq r\in\nn$.}
We first apply estimates in Case 1 to $r=r_1=r_0=s+\f{5}{2}$. This gives us a rapid convergence of
$\|A_i\|_{\rho_i;s+\f{5}{2},s},  \|\tilde F_{i+1}-\tilde F_i\|_{\rho_\infty/2;s+\f{7}{2},s}$.
Fix $r-1/2<r'<r$. We want to show that
$ \|A_i\|_{\rho_i;r',s},  \|\tilde F_{i+1}-\tilde F_i\|_{\rho_\infty;r'+1,s}$
converge rapidly.  By \re{hAtQ} we have
$$
| A_{i+1}|_{\rho_{i+1};r',s} \leq C_{r'} K^3(2) K(r'+2) \| A_i\|_{\rho_{i};r_1,s}\| A_i\|_{\rho_{i};r',s}.
$$
Here $r_1=s+1+\{r'\}<s+2$. Since $\| A_i\|_{\rho_{i};s+2,s}$ converges rapidly,  then $\| A_{i+1}\|_{\rho_{i+1};r',s}$ also converges rapidly.
We verify that  $ \|\tilde F_{i+1}-\tilde F_i\|_{\rho_\infty/2;r'+1,s}$ converges rapidly, by \re{hderF}.  Therefore, $\tilde F_\infty\in \cL C^{r'+1,s}$ for any $r'<r$.
The proof of \rp{pnl} is complete.
\end{proof}

\begin{rem}\label{1drem}
Strictly speaking, the above proof assume that $n\geq2$. When $n=1$, the integrability condition \re{2p3} is vacuous. We can replace the
homotopy formula by the Cauchy-Green operator
$$
a(z)=\f{-1}{\pi}\f{\pd}{\pd\ov z}\int_{|\zeta|<\rho}\f{a(\zeta)}{\zeta-z}\, d\xi d\eta,\quad \zeta=\xi+i\eta.
$$
  Thus the last term $P_2\langle A,\pd A\rangle$ in \re{2p12}, which is from  the integrability condition, is removed.
Therefore, the proof with possible simplifications is still valid. In one-dimensional case, one can also employ the Picard iteration as in Bers~\ci{Be57}, Chern~\ci{Ch55}, and Bertrand-Gong-Rosay~\ci{BGR14} (for a parametric  version), except possibly  for Case~3 where both $r,s$ are integers. Of course, for the non-parametric case, it is a simple fact that any $C^1$ diffeomorphism that transforms a complex structure of class $C^r$ into another one of the same class belongs to $\cL C^{r_-}$ when $r=2,3,\dots$.  When $\cL C^{r,s}$ substitutes $C^r$, the latter is however not valid for the parametric complex structures.    \end{rem}

We now prove the following by using an argument in   Nirenberg~\ci{Ni57} and \rp{pnl}.

 \pr{nithm} Let $s\in\nn$, $r\in (s+1,\infty]\setminus\nn$.
Let $\{S^\la \}\in \cL  C^{r,s}(D)$ be a family of complex Frobenius structures defined by \rta{thm1}.
There exists a family of $\{F^\la \}\in \cL C^{1,0}(U)$ transforming the structures into  the standard complex Frobenius structure in $\cc^{n}\times\rr^M\times\rr^L$. Furthermore, we have
\bppp
\item
$\{F^\la \}\in \cL C^{r,s}(U)$ when $L=0$.
\item  $\{F^\la \}\in \cL C^{[r]-1,s}(U)$ when $L>0$ and $s=[r]-1$.\eppp
 \epr
 \begin{proof}We will use the summation convention described before \rl{intialcod}.
 We use an adapted frame
 \aln
X^\la _{m}&=\pd'_{m}+ \RE (B^\la _{m\beta} \pd_\beta)+b^\la _{m\ell}\pd_{ \ell}'',  \qquad 1\leq m\leq M,\\ Z^\la _{\ov\all}&=\pd_{\ov\all}+ A^\la _{\ov\all \beta}\pd_\beta+a^\la _{\ov\all  \ell }\pd_{\ell }'', \ \ \  \qquad
  1\leq \all\leq n,
 \end{align*}
 with $X^\la _ m =\pd'_m$ and $Z^\la _{\ov\all}=\pd_{\ov\all}$ at the origin.

$(i)$
Suppose that $L=0$. Then the integrability condition on $S^\la $  implies that
 $$
 [X_m^\la ,X_{m'}^\la ]=0, \quad [Z_{\ov\all}^\la ,Z_{\ov\beta}^\la ]=0
 $$
(see~\re{6p11} below), while the Levi-flatness condition~\re{LLf} is vacuous.
 Restricted to $t=0$,   $\{Z^\la _{\ov\all}\}$ defines a family of complex structures. By \rp{pnl}, we can find   $\{F_0^\la \}\in\cL C^{r+1,s}(U')$ transforming
 the restricted structures into the standard complex structure in $\cc^n$. Then $\{X_m^\la \}$ is still in $\cL C^{r,s}$.  Applying the (real) Frobenius theorem (\rp{imf}), we find a unique family $\{F^\la \}\in\cL C_{t,z;\la}^{r+1,r;s}(U)$ such that $F^\la =F_0^\la $ for $t=0$ and $F^\la $ transform $\{X_1^\la,\dots,X_m^\la \}$ into the standard structure in $\cc^n\times\rr^M$.   As in \ci{Ni57}, let us  show that $\{Z^\la \}$ is already the standard complex structure in $\cc^n$. Indeed, we know that
 $$
 A^\la _{\ov\all\beta}=0, \quad \text{when $t=0$}.
 $$
 Now $[Z_{\ov\all},X_m]'=0$ in \re{6p11} implies that $\pd_{t}A^\la_{\ov\all\beta}(z,t)=0$. Hence $A_{\ov\all\beta}=0$.

\medskip
 $(ii)$ Suppose that $L>0$.  We reduce it to the previous case by using the Frobenius theorem.
 Recall that $S^\la$ are Levi-flat.  Applying \rp{imf}, we find $F_0^\la (z,t,\xi)$ with $F_0^\la(0,0,\xi)=\xi$ so that
 $X^\la_jF_0^\la=\ov Z^\la_{\ov\alpha}F^\la= Z^\la_{\ov\all}F_0^\la=0$, while $\{F_0^\la \}\in C_{(z,t),\xi;\la}^{r+1,r;s}$. Let $\tilde F_0^\la$ be defined by $(\hat z,\hat t,\hat\xi)=(z,t,F_0^\la(z,t,\xi))$. Then $(z,t,\xi)=(\hat z,\hat t,\hat F^\la_0(\hat z,\hat t,\hat\xi))$ with  $\{\hat F_0^\la\}\in \cL C^{r,s}$.
 Now $\{(\tilde F_0^\la) _*Z_{\ov\all}^\la ,(\tilde F_0^\la)_*X_m^\la \}$,  denoted by $\{\widehat Z_{\ov\all}^\la ,\widehat X_m^\la \}$, has the form
 \aln
\widehat X^\la _{m}&=\widehat\pd'_{m}+ \RE (B^\la _{m\beta}(\hat z,\hat t,\hat F_0^\la(\hat z,\hat t,\hat\xi))) \widehat\pd_\beta),  \qquad 1\leq m\leq M,\\
\widehat  Z^\la _{\ov\all}&=\hat\pd_{\ov\all}+ A^\la _{\ov\all \beta}(\hat z,\hat t,\hat F_0^\la(\hat z,\hat t,\hat\xi))\widehat \pd_\beta, \ \ \  \qquad
  1\leq \all\leq n.
 \end{align*}
 We drop all hats.   Hence, $\{Z_{\ov\all}^\la ,X_m^\la \}$ is still in $\cL C^{r,s}$, but its coefficients depend on $\xi$.
 Since $[r]-1=s$,
 we  treat $\xi,\la$ as parameters, restrict $Z_{\ov\all}^\la$ to $t=0$ and find a diffeomorphism $\tilde F_1^\la\colon (z,\xi)\to(F_1^\la (z,\xi),\xi)$ such that for  fixed $(\xi,\la)$
 the mapping transforms $\{Z_{\ov\all}^\la\}$ into the standard complex structure in $\cc^n$.  Here $\{\tilde F_1^\la\}\in \cL C^{r+1,[r]-1;s}_{(z,t),\xi;\la}$.
 It also transforms $X_m^\la$ to new vector fields,  denoted by $\widehat X_m^\la$. We have
 $$
 \widehat X_m=X_mt_{m'}\pd_{t_{m'}}+X_m\ov {(F^\la_1)_\all}\pd_{\ov\all}+X_m( F^\la_1)_\all\pd_\all.
 $$
Since $X_m$ does not involve $\pd_{\xi_\ell}$, then
 $
 \{\widehat X_m^\la\}\in \cL C^{r,[r]-1;s}_{(z,t),\xi;\la}.
 $
 Again, we apply the real Frobenius theorem to $\{X_m^\la\}$, treating $\xi,\la$
 as parameters,  to find
 $$\tilde F_2^\la\colon (z,t,\xi)\to(F_2^\la (z,t,\xi),t,\xi), \quad F_2^\la(z,0,\xi)=F_1^\la(z,\xi)$$
with $\{F_2^\la\}\in \cL C^{r+1,r,[r]-1;s}_{t,z,\xi;\la}$ transforming $\{\hat X_1,\dots, \hat X_M\}$ into $\{\pd_1',\dots, \pd_M'\}$. We apply the argument from (i) to conclude that $\{\tilde F_2^\la\}$ normalizes the   structures.
 \end{proof}
Roughly speaking \rp{nithm} ({\it ii}) gives us a desired regularity result with   loss of one derivative. One of main results of this paper is to deal with the loss of derivatives. We will also study the case when $s$ is non integer.

All results in this paper indicate that for a family of Levi-flat CR structures, it is important to seek solutions that are less regularity in parameter $\la$. This is clearly showed by the examples in section~\ref{sec4}. There is however  an exception, namely, for a holomorphic family of complex structures, as shown by Nirenberg~\ci{Ni57} for the $\cL C^\infty$ case.
\begin{defn}Let $r\in(0,\infty]$. Let $D$ $($resp. $P,Q)$ be a  domain in $\cc^n$ $($resp. $\cc^m,\rr^k)$. We say that $\{u^\la\colon\la\in P\}$ is a holomorphic family
of $\cL C^r$ $($resp. $\cL C^{r_-})$ functions $u^\la$ in $D$ if $(x,\la)\to u^\la(x)$ is continuous in $D\times P$,    $u^\la\in \cL C^r(D)$ $($resp. $\cL C^{r_-}(D))$
for each   $\la$, and $u^\la(x)$ is holomorphic
in $\la\in P$ for each   $x\in D$. We say that $\{v^{\la,\mu}\colon\la\in P,\mu\in Q\}$ is a family of $\cL C^r$ functions in $D$ that depend holomorphically in $\la$ and analytically in $\mu$, if there are a domain $\tilde Q$ in $\cc^k$ with $\tilde Q\cap\rr^k=Q$ and
a holomorphic family $\{\tilde v^{\la,\mu}\colon\la\in P,\mu\in\tilde Q\}$ of $\cL C^r$ $($resp. $\cL C^{r_-})$ functions $\tilde v^{\la,\mu}$ in $D$ such that $\tilde v^{\la,\mu}=v^{\la,\mu}$ for $\mu\in Q$.
\end{defn}
For the finite smoothness case, we have the following analogue of Nirenberg's theorem.
\begin{cor}\label{}   Let $D$ $($resp. $P,Q)$ be a  domain in $\cc^n$ $($resp. $\cc^m,\rr^k)$.   Let $S^{\la,\mu} $ be a    family of complex structures in $D$ defined by
$$
Z_{\ov\all}^{\la,\mu} =\sum_{\beta=1}^na_{\ov\all\ov \beta}^{\la,\mu}\DD{}{\ov z_\beta}+b_{\ov\all\beta}^{\la,\mu} \DD{}{z_\beta}, \quad 1\leq \all\leq n
$$
where $\{a_{\ov\all\ov \beta}^{\la,\mu}\} ,\{b_{\ov\all\beta}^{\la,\mu} \}$ are families of $\cL C^r$ functions in $D$ that depend holomorphic in $\la\in P$ and analytic in $\mu\in Q$. Then for each $(x_0,\la_0,\mu_0)\in D\times P\times Q$ there is a family of $\cL C^{r+1}$ $($resp. $\cL C^{(r+1)_-})$ diffeomorphisms $F^{\la,\mu} $ in $  D_1$
that depend holomorphically in $\la\in P_1$ and analytically in $\mu\in  Q_1$ such that $F^{\la,\mu} $ transform $\{Z_{\ov 1}^{\la,\mu},\dots,Z^{\la,\mu}_{\ov n} \}$ into the standard complex structure in $\cc^n$, provided $r$ is in $(1,\infty]\setminus\nn$ $($resp. $\{2,3,\dots\})$. Here $  D_1\times  P_1\times  Q_1$ is a neighborhood of $(x_0,\la_0,\mu_0)$.
\end{cor}
\begin{proof}By the definition, we extend $a^{\la,\mu},b^{\la,\mu}$ to a family
of $\cL C^r$ functions in $D_0$  that depend holomorphically in $(\la,\mu)\in P_0\times Q_0$, where $D_0\times P_0\times Q_0$ is a neighborhood of $(x_0,\la_0,\mu_0)$ in $\cc^n\times\cc^m\times\cc^k$. The extension of the coefficients allows us to extend $Z_{\ov 1}^{\la,\mu},
\dots, Z_{\ov n}^{\la,\mu}$ to a complex structure defined in $D_0\times P_0\times Q_0$, when we include the vector fields
$$
\DD{}{\ov\la_i}, \quad \DD{}{\ov\mu_j},\quad 1\leq i\leq m,\quad 1\leq j\leq k.
$$
By Cauchy's formula, we verify that $(z,\la,\mu)\to (a^{\la,\mu}(x),b^{\la,\mu}(x))$
is $\cL C^r$ in $(z,\la,\mu)$. Assume first that $r>1$ is not an integer. By Webster's result on Newlander-Nirenberg theorem, we can find a $\cL C^{r+1}$ diffeomorphism $(z,\la,\mu)\to (f_1,\dots, f_{n+m+k})$ such that $Z_{\ov\all}^{\la,\mu}f_\ell=0$ while $f_j$ are holomorphic in $\la,\mu$. Near $(z_0,\la_0,\mu_0)$, by a linear combination we may assume that $z\to (f_1(z,\la_0,\mu_0),\dots, f_{n}(z,\la_0,\mu_0))$ is a diffeomorphism near $z_0$. Let $F_\all^{\la,\mu}=f_\all^{\la,\mu}$ for $1\leq\all\leq n$. Then $F^{\la,\mu}$ has the desired property. When $r=2,3,\dots$, we use \rp{pnl} instead of Webster's result and repeat the above argument. The proof is complete.
\end{proof}

\setcounter{thm}{0}\setcounter{equation}{0}
\section{Approximate solutions via two homotopy formulae
}\label{sec6}

In this section, we prove the results for the general case.
We will first set up a procedure for our problem to apply a homotopy formula
for the normalization of the complex Frobenius stricture. We will need two
known homotopy formulae. One of them  due to  Tr{e}ves~\ci{Tr92} is
for the complex differential $\cL D$ associated the flat CR structure in $\cc^n\times \rr^m$.
Another is the Poincare lemma for $d^0+\db+\pd$.

We will finish the section by illustrating how the homotopy formula
is used more effectively. See also Gong-Webster~\cite{GW12}  for the case of CR embedding,
which improves Webster's original application of the homotopy formula for $\db_b$.

Recall that on $\cc^n$,  we set  $\pd_{\ov\all}=\pd_{\ov z_\all}$ and $\pd_\all=\pd_{ z_\all}$. On $\rr^m$,
set $\pd'_m=\pd_{t_m}$. On $\rr^L$,  set $\pd_\ell''=\pd_{\xi_\ell}$. Let $N=2n+M+L$.
We will use the summation convention described before \rl{intialcod}.
To normalize $S$ near a point $p$, we may assume that $p=0\in D\subset \rr^N$ and require that
all changes of   coordinates   fix the origin.
By \rl{intialcod} we first choose   linear coordinates of $\rr^N$  so that near $0$,  $S$ is spanned by the adapted frame
\al\nonumber %
X_{m}&=\pd'_{m}+2\RE (B_{m\beta} \pd_\beta)+b_{m  \ell }\pd_{\ell}'',  \quad 1\leq m\leq M,
\\ Z_{\ov\all}&=\pd_{\ov\all}+ A_{\ov\all \beta}\pd_\beta+a_{\ov\all \ell}\pd_{\ell }'', \ \  \qquad
  1\leq \all\leq n.\nonumber %
\end{align}
Here $A_{\ov\all \beta}, B_{m\ell}, a_{\ov\all\ell}$ are complex-valued functions and $b_{m\ell}$
are real-valued. They all vanish at the origin.  Define $ \A=(A,B,a,b)$ and     assume that
\eq{}\nonumber %
|(A,B,a,b)|<1/{C_0}.
\eeq
Then $Z_{\ov\all}, Z_\all:=\ov{Z_{\ov\all}}$, and $ X_m$ are pointwise $\cc$-linearly independent.    Set
$$
\ov A_{\ov \all \beta}:=\ov{A_{\ov\all\beta}}, \quad \ov B_{m\ov\beta}:=\ov{B_{m\ov \beta}}, \quad \ov a_{\all\ell}:= \ov{a_{ \all\ell}}.
$$
The Lie brackets of the adapted frame  are
\aln
[Z_{\ov\all},Z_{\ov\beta}]&=
(Z_{\ov\all}A_{\ov\beta \gaa}-Z_{\ov\beta}A_{\ov\all \gaa})\pd_\gaa+(Z_{\ov\all}
a_{\ov\beta \ell}-Z_{\ov\beta}a_{\ov\all \ell})\pd_\ell'',\\
[X_m,X_{m'}]&=
2\RE\{(X_{m}B_{{m'} \gaa}-X_{m'}B_{m \gaa})\pd_\gaa\}+(X_{m}
b_{{m'}\ell}-X_{m'}b_{m\ell})\pd_\ell'',\\
[Z_{\ov\all},X_{m}]&=[Z_{\ov\all},X_{m}]'\mod Z_{\ov\gaa},\\
\intertext{with}
[Z_{\ov\all},X_{m}]'
&:=(Z_{\ov\all}B_{m \gaa}-X_{m}A_{\ov\all \gaa}-A_{\ov\beta \gaa}Z_{\ov\all}\ov B_{m \beta})\pd_\gaa
 \\ &\quad +(Z_{\ov\all}
b_{m \ell}-X_{m}a_{\ov\all \ell}-a_{\ov\beta \ell}Z_{\ov\all} \ov B_{m\beta})\pd_\ell''.
\end{align*}
To find equations for $S+\ov S$ to be Levi-flat, we compute
Levi-forms
\gan
[Z_{\all},Z_{\ov\beta}]=
Z_{\all}A_{\ov\beta \gaa}\pd_\gaa-Z_{\ov\beta} \ov A_{ \ov\all\gaa}\pd_{\ov\gaa}+(Z_{\all}
a_{\ov\beta \ell}-Z_{\ov\beta} \ov  a_{ \ov\all \ell})\pd_\ell''.
\end{gather*}
Define the matrix
\ga\nonumber
R:= (I-A\ov A)^{-1}-I, \quad
  R=(R_{\ov\all\ov\beta}), \quad \ov R=(\ov R_{\ov\all\ov \beta}), \quad U:=I+R.
\end{gather}
Let us substitute $\pd_\all$ and $\pd_{\ov\all}$ by using
\aln
\pd_{\ov\all}&=Z_{\ov\all}- A_{\ov\all \beta}Z_\beta-a_{\ov\all \ell}\pd_{\ell}''  +R_{\ov\all \ov\beta} Z_{\ov\beta}-
R_{\ov\all \ov\beta} A_{\ov\beta \gamma} Z_{\gaa}
\\ &\quad -( R_{\ov\all \ov\beta}a_{\ov \beta \ell}-U_{\ov\all\ov\beta}A_{\ov\beta\mu}\ov a_{\ov\mu\ell})\pd_\ell''.
\end{align*}
Write $[X_{\all},X_{\ov\beta}]=
L_{\ell;\all\ov\beta}\pd_\ell''
\mod(Z_{\ov\gaa}, {Z_{ \mu}})$ with Levi forms $(L_{\ell;\all\ov\beta})_{1\leq\all,\beta\leq n}$ being defined by
\al
\label{LLf}
L_{\ell; \all\ov\beta} &:=Z_{\all}
a_{ \ov\beta \ell}-Z_{\ov\beta}{\ov a_{\ov\all \ell}}
+Z_{\all}{ A_{\ov\beta \gaa}}\ov U_{\ov\gaa\ov\del}(\ov A_{\ov\del\nu}a_{\ov\nu\ell}-\ov a_{\ov\del\ell})
\\
&\quad\  -Z_{\ov\beta}\ov A_{\ov\alpha  \gaa}   U_{\ov\gaa\ov\del}(  A_{\ov\del\nu}\ov a_{\ov\nu\ell}-  a_{\ov\del\ell}), \quad 1\leq\all,\beta\leq n.\nonumber
\end{align}
Then the integrability conditions are equivalent to
\eq{6p11}
[Z_{\ov\all},Z_{\ov\beta}]=0,\quad
[X_m,X_{m'}]=0, \quad [Z_{\ov\all},X_m]'=0, \quad  L_{\ell;\all\ov\beta}=0.
\eeq

We now introduce  some differential forms by using the adapted frame.
On $\cc^n$ and $\rr^m$ we introduce the following $(0,1)$ and $1$ forms
\ga
A_\beta=A_{\ov\all \beta} d\ov z_\all, \quad a_\ell =a_{\ov\all \ell} d\ov z_\all,\nonumber \\
B_\beta=B_{m\beta }dt_m,\quad b_\ell=b_{m\ell} dt_m, \nonumber\\
 {\mathcal B}_\beta=A_\beta+B_\beta,\quad  {\mathfrak b}_\ell =a_\ell+\ov a_\ell+b_\ell. \label{a73}
 \end{gather}
For a function $f$ in $\rr^N$,
we denote by $\ov\pd f=\pd_{\ov\all} fd\ov z_\all$ and $d_tf=\pd'_mfdt_m$.
Recall that $\pd_{z,\xi} u$ stands for first-order derivatives of $u$ in $\cc^n\times\rr^L$.

The conditions \re{6p11}
are equivalent to
\ga\label{6p14}
\db A=\langle(A,a), \pd_{z,\xi}  A\rangle, \quad \db a=\ov{\pd\ov a}=\langle(A,a), \pd_{z,\xi} a\rangle, \\
\label{6p15}d_tB=\langle(B,b), \pd_{z,\xi} B\rangle, \quad d_tb=\langle(B,b),  \pd_{z,\xi} b\rangle, \\
\label{6p16}\db B+d_tA=\langle(A,a), \pd_{z,\xi} B\rangle+\langle(B,b),  \pd_{z,\xi} A\rangle+\langle A,\pd_{z,\xi} B+\langle(A,a),\pd_{z,\xi} B\rangle,\\
\label{6p17}\db b+d_ta=\langle(A,a), \pd_{z,\xi} b\rangle+\langle(B,b), \pd_{z,\xi} a\rangle+\langle a, \pd_{z,\xi} B\rangle+\langle a(A,a), \pd_{z,\xi} B\rangle,\\
\label{6p18}\pd a+\ov{\pd a}=2\RE\left\{ \langle A, \pd_{z,\xi} a\rangle+\langle P(R(A),A,a)(\ov A,\ov a), \pd_{z,\xi} A\rangle\right\}.
\end{gather}
 Recall that $d=d_t+\pd_z+\db_z$ and $\cL D=d_t+\db_z$. Define
 \eq{AdA}
  \A:=( {\B},\mathfrak b),\quad \mathbf d\A:=(\cL D \B,d\mathfrak b).
 \eeq
Thus we can rewrite the above conditions  as
\ga \label{6819}
{\mathbf d}\A=\langle\psi(\A)\A,\pd_{z,\xi}\A\rangle.
\end{gather}
Here
$\psi(\A)$ are power series  in the coefficients  of $\A,\ov{\A}$, which are convergent for $|\A|<1/C$ for some constant $C>1$.

 We consider a transformation   \eq{}H\colon \hat
z=z+f(z,t,\xi),\quad \hat t=t,\quad \hat \xi=\xi+g(z,t,\xi)\nonumber
\eeq
with
$f(0)=0$, $g(0)=0$,   $|\pd f(0)|<1$ and $|\pd g(0)|<1$. Define $\widehat\pd_\all=\pd_{\hat z_\all},\widehat\pd'_m=\pd_{\hat t_m}$, and $\widehat\pd''_ \ell =\pd_{\hat \xi_ \ell }$, etc.  Then
 \aln
H_*Z_{\ov\all}&=
(\del_{\ov\all \ov\beta}+(A_{\ov\all \gaa}\pd_\gaa  +
a_{\ov\all \ell}\pd_{\ell}'') \ov f_\beta)
\widehat\pd_{\ov{\beta}}+
(\pd_{\ov\all}f_\beta+
A_{\ov\all \beta}+(A_{\ov\all \gaa}\pd_\gaa  +
a_{\ov\all \ell}\pd_{\ell}'')f_\beta)\widehat \pd_{\beta}\\
&\quad +
(\pd_{\ov\all}g_\ell+a_{\ov\all \ell}+(A_{\ov\all \gaa}\pd_\gaa  +
a_{\ov\all \ell'}\pd_{\ell'}'')g_\ell)\hat \pd_{\ell}'', \\
H_*X_{m}&=\hat\pd'_{m}+
2\RE((\pd'_mf_\beta+B_{m \beta}+(2\RE (B_{m\gamma} \pd_\gamma)+b_{m \ell}\pd_{\ell}'')
f_\beta)\hat\pd_{\beta})\\ &\quad +
(\pd'_mg_\ell+b_{m{m'}}+(2\RE (B_{m \gamma} \pd_\gamma)+b_{m {\ell'}}\pd_{\ell'}'')g_\ell)\hat \pd_{\ell}''.
\end{align*}
Assume that
$$
f(0)=0, \quad g(0)=0, \quad |\pd f|+|\pd g|<1/{C_0}.
$$
Then the new
adapted frame for $H_*S$  is
$
\widehat Z_{\ov\all}=\widehat\pd_{\ov\all}+ \widehat A_{\ov\all \beta}\hat\pd_\beta+
\widehat  a_{\ov\all \ell}\widehat\pd_{\ell}''$ and $
\widehat  X_{m}'=\widehat\pd_{m}'+2\RE (\widehat B_{m \beta}\widehat \pd_\beta)+\widehat b_{m\ell}\widehat\pd_{\ell}''$. Their coefficients are determined as follows. We have
\gan
\widehat X_m=H_*X_m, \quad C_{\ov\all \ov\beta}\widehat Z_{\ov\beta}=H_*Z_{\ov\all}
\end{gather*}
with $C_{\ov\all \ov\beta}\circ H
=\delta_{\ov\all \ov\beta}+ (A_{\ov\all \gaa}\pd_\gaa  +
a_{\ov\all \ell}\pd_{\ell}'') \ov f_\beta.$ Also
\aln
C_{\ov\all \ov\beta}\circ H
&=\delta_{\ov\all \ov\beta}+ (A_{\ov\all \gaa}\pd_\gaa  +
a_{\ov\all \ell}\pd_{\ell}'') \ov f_\beta,\\
(C_{\ov\all \ov\gaa}\hat A_{\ov\gaa \beta})\circ H&=\pd_{\ov\all}f_\beta+
A_{\ov\all \beta}+(A_{\ov\all \gaa}\pd_\gaa  +
a_{\ov\all \ell}\pd_{\ell}'')f_\beta,\\
(C_{\ov\all \ov\gaa}\hat a_{\ov\gaa \ell})\circ H&=\pd_{\ov\all}g_\ell+a_{\ov\all \ell}+(A_{\ov\all \gaa}\pd_\gaa  +
a_{\ov\all m}\pd_{m}')g_\ell,\\
\hat B_{m\beta}\circ H&=\pd_m'f_\beta+B_{m \beta}+(2\RE (B_{m \gamma} \pd_\gamma)+b_{m \ell}\pd_{\ell}'')
f_\beta,\\
\hat b_{m\ell}\circ H&=\pd_m'g_\ell+b_{m \ell}+(2\RE (B_{m \gamma} \pd_\gamma)+b_{m\ell'}\pd_{\ell'}'')g_\ell.
\end{align*}
By an abuse of notation, define 
\aln
\widehat{\A}_\beta\circ H&=\widehat A_{\ov\all\beta }\circ Hd\ov z_\all+\widehat B_{m\beta}\circ Hdt_m,
\\
\hat{\mathfrak b}_\ell\circ H&=2\RE\{\hat a_{\all \ell}\circ Hdz_\all\}+\hat b_{m\ell}\circ Hd t_m.
\end{align*}
Thus,  $H$  transforms $\{Z_{\ov \gamma},X_m\}$ into the span of $\{\widehat\pd_{\ov\del},\widehat\pd_{m'}'\}$ if and only if
$\widehat A_\beta$, $\widehat B_\beta$, $\hat a_\ell$, and $\hat b_\ell$ are zero, i.e.
\ga\label{6p32}
\cL Df+\B+\langle(\B,\mathfrak b), \pd_{z,t}  f\rangle=0, \\
 dg+\mathfrak b+\langle(\B,\mathfrak b), \pd_{z,t}    g\rangle=0.
\label{6p33}
\end{gather}
When $\widehat \A=(\B,\mathfrak b)$ is non-zero,   we have,  for $\mathbf d$ defined by \re{AdA},
\ga\label{ChAH}
(\cL C\widehat {\A})\circ H=\mathbf dh+\A+\langle\A , \pd_{z,t}   h\rangle, \quad h:=(f,g), \quad \cL C\widehat{ \A}=(C\widehat{\B},\hat{\mathfrak b}).
\end{gather}
As in  Webster~\ci{We89}, we apply   homotopy formulae.
 In our case,
we apply the (approximate) Poincar\'e lemma and the  Koppelman-Leray formula for forms
of degree $1$ and $(0,1)$ to find approximate solutions to \re{6p32}-\re{6p33}. We will then use the compatibility conditions \re{6p14}-\re{6p18}
and two homotopy formulae to verify that
the solutions $f$ and $g$ are indeed    good approximate solutions to \re{6p32}-\re{6p33}.

 Set $D_\rho=B_\rho^{2n}\times B^M_\rho\times B_\rho^L$. Recall that $\cL D =d^0+\db$ is defined in $\cc^n\times\rr^M$ and $d$ is the  standard real differential in $\cc^n\times\rr^M$.
We have
homotopy formulae
\aln \var_{q}&=\cL DT_{B^{2n}_\rho\times B_{\rho}^M}\var_{q}
 +T_{B^{2n}_\rho\times B_{\rho}^M}\cL D  \var_{(0,q)},  \quad q>0,\\
  \psi_{q}&=dR_{B^{2n}_\rho\times B_{\rho}^M}\psi_{q}
 +R_{B^{2n}_\rho\times B_{\rho}^M} d\psi_{q}, \quad q>0.
\end{align*}
 where $T_\rho$ is defined by \re{hvar},  and $R_{\rho}$ is defined by \re{pht} in which $\rr^M, B_\rho$ are replaced by $\cc^n\times\rr^M, B^{2n}_\rho\times B_{\rho}^M$ respectively.
Define $\cL T_{\rho}(\B,\mathfrak b):=(T_{B_\rho\times B_\rho'}\B,R_{B_\rho\times B_\rho'}\mathfrak b)$.
Thus we have the homotopy formula
$$
\A=\mathbf d\cL T_\rho\A+\cL T_\rho\mathbf d\A.
$$
To find approximate solutions to \re{6p32}-\re{6p33}, we take
\ga\label{6p39} 
h=-\cL T_\rho\A +(\cL T_\rho\A)(0).
\end{gather}
Note that $g$ is real-valued, as required.
By \re{6819}, formally, we obtain $|\widehat {\A} |\ple |\A|^2$.
 However,
the argument cannot be repeated infinitely many times as $\hat A$ in \re{ChAH} is less smooth than $\cL A$
because of no gain in derivative from the estimate of $ \cL T_\rho\A$.

Instead, we use a smoothing operator $S_\tau$ in $(z,t,\xi)$ variables on $D_\rho$, and take
\ga\label{6p46} 
H(z,t,\xi)=(z+f_\tau(z,t,\xi),t,\xi+g_\tau(z,t,\xi)),\\
h_\tau=(f_\tau,g_\tau):=-S_\tau   \cL T_\rho\A+S_\tau \cL T_\rho\A(0).
\label{6p47}
\end{gather}
On shrinking domains $D_{\rho_1 }$, we have
\begin{equation} \nonumber %
[S_\tau,\mathbf d]=0.
\eeq
As in \ci{We89},  we apply the homotopy formula,  a second time, to write $\mathbf d\cL T_\rho\A$ as $\A-\cL T_\rho\mathbf d\A$. We obtain
\aln 
\A+ \mathbf dh_\tau  &=\A -\mathbf dS_\tau \cL T_\rho\A =\A -S_\tau\mathbf d \cL T_\rho\A\\
&=(I-S_\tau )\A+S_\tau \cL T_\rho\mathbf d\A.
\end{align*}
Now equation \re{ChAH} for $\widehat\A$ is changed to
\al   \label{6p50+}
(\cL C\widehat A )\circ H
= \A + \mathbf dh_\tau+\langle\A, \pd h_\tau\rangle=I_1+I_2+I_3
\end{align}
where $\cL C\widehat{\A}=(C\widehat{\B},\hat{\mathfrak b})$ and
\ga
\label{6p51}
I_1=(I-S_\tau )\A,\quad   I_2=S_\tau
\cL T_\rho{\mathbf d}\A,\quad  I_3= \langle\A , \pd  S_\tau \cL T_\rho\A\rangle.
\end{gather}
We introduce
\ga
C\circ H-I=I_4, \quad I_4=\langle \A,\pd S_\tau \A\rangle.\label{717} 
\end{gather}
Here $
I_{4,\ov\all\beta}=-(A_{\ov\all \gaa}\pd_\gaa  +
B_{\ov\all \ell}\pd_{\ell}'')S_\tau (P_\rho\A)_\beta$. We also have
\ga
\mathbf d\A=\langle\psi(\A)\A,\pd\A\rangle.
\label{6p54}\end{gather}
Notice  that all   $I_i$ have the smoothing operator $S_\tau $. And
 $I_4$ will be estimated as $I_3$ because they have the same form. The
$I_1$   will
dictate the regularity result when we apply the iteration method.

\setcounter{thm}{0}\setcounter{equation}{0}
\section{The proof   for  a family of complex Frobenius structures
}\label{sec7}

In this section, we prove the general version of our theorem. Recall that the proofs in section~\ref{sec5}  rely on  gaining one full derivative in the  Koppelman-Leray homotopy formula. Such a gain does not exist in the general case. Thus in an iteration
procedure, we need to apply the Nash-Moser smoothing methods. By using smoothing operators, we will also be able
to get rid of the extra one derivative required in the non-parametric  variables, i.e., we will consider complex
Frobenius structures of class $\cL C^{r,s}$ requiring $r\geq s$ and $r>1$, instead of $r>s+1$.
%

Besides the more elaborated use of homotopy formula, we will also need to use the interpolation of H\"older norms
to estimate the norms for new complex Frobenius structures  and to obtain rapid convergence in higher order
derivatives via the rapid convergence in lower order derivatives. Note that the interpolation is
avoided in the rapid KAM arguments in section~\ref{sec5}.


We first recall estimates for smoothing operator
\al\label{7p1}
\|S_\tau f\|_{D,b}&\leq C_{b-a}\tau^{a-b}\|f\|_{D',a},\quad 0\leq b-a<\infty;\\
\|f-S_\tau f\|_{D,a}&\leq C_{a-b}\tau^{b-a}\|f\|_{D',b}, \quad 0\leq b-a\leq m_0.
\label{7p2}
\end{align}
Here $S_\tau $ depends on $m_0\in \nn$. See Moser~\ci{Mo62}, where the above inequalities were stated and proved
for integers $a,b$. The general case of real numbers $a,b$ was  derived in \ci{GW12} by interpolation in H\"older norms for domains of the cone property.

For the H\"older spaces with parameter, we will use the same smoothing operator. The approximation is subtle and we derive it in full details.
Fix a positive integer $L$.
There exists a
   smooth function $\varphi$ in $\rr^{N_0}$ with compact
support in the unit ball such that
\eq{intP}
 \int P(x)\psi(x)\, dx=P(0),
\eeq
where $P$ is a polynomial of degree at most $L$; see Moser~\ci{Mo62}.
The
smoothing operator for functions $f$   with compact support
is
$$
S_\tau f=\psi_\tau * f, \quad \psi_\tau =\tau ^{-N_0}\psi(\tau x), \quad \tau >0.
$$

%

It will be convenient to use Seeley extension operator~\ci{Se64} to extend $\{f^\la\colon0\leq t\leq1\}$ to
a larger family defined for $\la$ in a larger interval, say $\tilde I=[-1,2]$.
Let us recall the extension.  Seeley~\cite{Se64} showed that there
 are  numerical sequences $\{a_k\}_{k=0}^\infty, \{b_k\}_{k=0}^\infty$   so that ($i$) $b_k < 0$ and $b_k\to-\infty$, $(-1)^ka_k>0$,
$(ii)$ $\sum_{k=0}^\infty |a_k|\cdot |b_k|^n<\infty$  for $n = 0, 1, 2, \ldots$, $(iii)$ $\sum_{k=0}^\infty  a_k(b_k)^n=1$  for
$n = 0, 1, 2, \ldots$.  When $f^\la=0$ for $t\geq1/2$, we can define   the extension
\eq{seeley}
(E^sf)(y)=\sum_{k=0}^\infty a_k\phi(b_ks)f^{b_ks}(y), \quad s\leq0.
\eeq
Here $\phi$ is a $\cL C^\infty$ function satisfying $\phi(\la)=1$ for $\la<1$ and $\phi(\la)=0$
for $\la>2$. For a differential form $f$, we define $Ef$ by extending coefficients of $f$ via $E$.
Using a partition of unity on $[0,1]$, we obtain an extension $E\colon \cL C^0([0,1])\to \cL C^0_0((-1,2))$.  Applying the extension to the parameter $\la$, we define   function $\{E^\la f\}\in \cL C^{0,0}(\ov D)$ for $\la\in[-1,2]$.  Furthermore,
\eq{}
|\{E^\la  f\}|_{r,s}\leq C_{r,s}|f|_{r,s}, \nonumber %
\eeq
where $C_{r,s}$ depends on $r$ and $\cL C^r$ coordinate charts mapping $\pd D$ into the half-space, and $D$ is independent of $\la$. In our applications, $D$ is a ball in a Euclidean space of which the radius is bounded between two fixed positive numbers. Therefore, $C_{r,s}$ depends only on $r,s$.

We will use two smoothing operators.  The first smoothing operator is $S_\tau$, applied to all variables $(y,\la)$. This is suitable for functions in $\cL C^{r,s}$ for $r=s$. When $r>s$,
 we will use a {\it partial} smoothing operator $S_\tau$, by smoothing in variables $y\in D$.  Let us exam the effects of partial smoothing operator,
which is still denoted by $S_\tau$,  using \re{7p1}-\re{seeley}.
For   real numbers $a$, $b$ and $s$, define
\gan
d_s(a,b)=\tilde d_s(a,b)=b-a, \quad  \ s\in\nn;  \\
 d_s(a,b)=
\min(b-a, [b]-[a]), \quad  \tilde d_s(a,b)=
\max(b-a, [b]-[a]), \quad   s\not\in\nn.
\end{gather*}
When $s$ is not an integer,  $d_s(a,b)>0$ if and only if $[b]>[a]$.  Furthermore,
if $\{r\}\geq\{s\}$, then
\eq{}\nonumber
d_s(a,b)>1\Leftrightarrow
 \text{$b>a+1$ and $s\in\nn$, or $[b]\geq[a]+2$}.
\eeq

The following is a basic property of partial smoothing operator.
\pr{st} Let $D\subset D'$ be domains in $\rr^N$.  Suppose that $D$ has \ps.
 Let $0< \tau < \min\{1,\dist(D,\partial D')\}$. Then
\ga
|S_{\tau}f|_{D; b,s}
\leq
C_{b-a}\tau^{-\tilde d_s(a,b)}|f|_{D'; a,s}, \quad   0\leq a\leq b<\infty,\label{sftaj}
\\
|f-S_{\tau}f|_{D; a,s}
\leq
C_{b-a} \tau^{d_s(a,b)}|f|_{D'; b,s},\quad
  0\leq a\leq b<a+m_0,\label{fstfds}
\\
\|S_{\tau}f\|_{D; b,s}
\leq
C_{b-a}\tau^{-\tilde d_s(a,b)}\|f\|_{D'; a,s}, \quad0\leq  s\leq a\leq b<\infty,\label{stfj+}
\\
\|f-S_{\tau}f\|_{D; a,s}
\leq
C_{b-a} \tau^{d_s(a,b)}\|f\|_{D'; b,s},\quad
  0\leq s\leq a\leq b<a+m_0. \label{fsftds+}
\end{gather}
Here
    $ C_{a}$  depends on two constants $C_1^*,C_2^*$
in \ps\  of $D$ and bounded from above by an upper bound of $a$.
\end{prop}
\begin{proof} It suffices to verify the inequalities for $0\leq a<1$ and  $0\leq s<1$, since
$$
d_s(a,b)=d_{s-n}(a,b)=d_s(a-m,b-m), \quad
\tilde d_s(a,b)=\tilde d_{s-n}(a,b)=\tilde d_s(a-m,b-m),
$$
for integers $m,n$.
Let $\pd^k$ be a derivative in $y\in D'$ of order $k$. Let $\pd _\la$ be the partial derivative in parameter $\la$.
We have $\pd^k S_\tau f^\la =S_\tau  \pd^kf^\la $ and $\pd _\la S_\tau f^\la =S_\tau \pd_\la f^\la $ on shrinking domain $D'\times[0,1]$.

(i) Recall from \re{intP} that there is   a smooth function $\chi$  with   support in $|y|<1/2$  so that
  \eq{moserchi}
\int\chi(y)\, dy=1, \quad \int\chi(y)y^I\, dy=0, \quad 0<|I|\leq m_0+1.
\eeq
Let  $\chi_\tau (y)=\tau^{-n}\chi(\tau ^{-1}y)$.
We have $S_\tau  f^\la =f^\la  \ast\chi_\tau $. Thus $\|\{S_\tau   f^\la\} \|_{D;a,s}\leq
C\|f\|_{D';a,s}$. We have $\pd^k S_\tau  f^\la =f^\la \ast\pd^k\chi_\tau $. Since
$$\tau <\tau_0:=\min\{1,\dist(D,\partial D')\},$$
 we  have
$$
|\pd^k(\chi_\tau (y))|\leq C_k\tau ^{-k-n}\sum_{|I|\leq k}|(\pd^I\chi)(\tau ^{-1}y)|.
$$
This shows that $|S_\tau   f|_{D;a+k,s}\leq C_k\tau ^{-k}|f|_{D';a,s}$, which
also gives us the first inequality when $a, b$ are an integer. For
$k-1<b< k$, we write $b= a+\theta  k$ and obtain
$$
 |S_\tau   f|_{D;b,[s]}\leq C_b|S_\tau  f|_{D';a,[s]}^{1-\theta }
 |S_\tau   f|_{D';a+k,[s]}^\theta  \leq C_b'\tau ^{a-b}|f|_{D';a,[s]}.
 $$
When $ \{s\}>0$, we have for $[b]=[a]+\theta ([b]-[a])$
$$
\left|\f{ \pd_{\la'}^{[s]}S_\tau   f^{\la'}-\pd_{\la}^{[s]}S_\tau   f^{\la}}{|\la'-\la|^{\{s\}}}\right|_{D;[b]}
\leq C_b|S_\tau  f|_{D';[a],s}^{1-\theta }
 |S_\tau   f|_{D';[a]+k,s}^\theta  \leq C_b'\tau ^{[a]-[b]}|f|_{D';a,s}.
 $$

(ii) Let $k\leq b< k+1$. Let $y\in D$.
 Let $P^\la _k(y,y')$ be the Taylor polynomial of $f^\la $ of degree $k$
about $y$ so that
$$
|\pd_{\la}^{[s]}f^\la (y+y')-\pd_{\la}^{[s]}P^\la _k(y, y')|\leq C_b|f|_{D';b,[s]}| y'|^b,   \quad |y'|<\tau_0.
$$
By \re{moserchi},
$\int P^\la _k(y,y')\chi(y')\, dy'=f^\la (y).$
We have
\aln
S_\tau  f^\la (y)& =f^\la (y) +\int \{f^\la (y-\tau y')-P^\la _k(y,-\tau y')\}\chi(y')\, dy'\\
 & =f^\la (y)+ \mathcal I^\la (y).
\end{align*}
We have $|f^\la (y+\tau y')-P^\la _k(y,\tau y')|\leq C_b|f|_{D';b,[s]}\|\tau y'\|^b$ for
  $|\tau y'|<\tau_0$. Therefore
$$
|\mathcal I^\la(y)|\leq C_b\tau ^b|f|_{D';b,[s]}
\int_{|y|<1/2} |y|^b|\chi(y)|\, dy\leq C_b'\tau ^b|f|_{D';b,[s]}.
$$
This completes the proof of the second inequality
when $a=0$. For $0<a<1$, we have
\aln
\left|\pd_{\la}^{[s]}f^\la-\pd_{\la}^{[s]}S_\tau  f^\la \right|_{D;a}&\leq C_{b}\left|\pd_{\la}^{[s]}f^\la -\pd_{\la}^{[s]}S_\tau f^\la \right|_{D;0}^{\f{b-a}{b}}\left|\pd_{\la}^{[s]}f^\la -\pd_{\la}^{[s]}S_\tau f^\la \right|_{D;b}^{\f{a}{b}}
\\
& \leq C_b'( \tau ^b|f|_{D';b,[s]})^{\f{b-a}{b}}\|f\|_{D';b,[s]}^{\frac{a}{b}}.
\end{align*}
We can also estimate the H\"older ratio in $t$ as in ({\it i}). This
 completes the proof.
  \end{proof}
\begin{rem}
The cone property is used only for interpolation. Thus
 \rpa{st} hold for integral $a,b$ without \ps\ of $D$. We will use \re{sftaj} and \re{fsftds+} only for $b-a=1$, in which case $\tilde d_s(a,b)=1$.
\end{rem}

The following is the main ingredient in the Nash-Moser iteration procedure. Another ingredient is the interpolation inequality. As in Gong-Webster~\ci{GW12}, it is useful to adjust the parameters for iteration to avoid unnecessary loss of regularity.
\pr{NashMoser} Let $\kappa_0,\kappa, d$ be numbers bigger than $1$. Suppose that
\eq{goldenratio}   %
(\kappa -1)d>1.
\eeq
Let $P,P_0$ be polynomials of non-negative coefficients.
Let $K_j$ be a sequence of  numbers satisfying
$$
0\leq K_j\leq e^{P(j)}
$$
 for $j\geq0$.
 Let $\hat L_0\geq0$.
Suppose that $\kappa_*,d_*$ satisfy
\eq{ksds}   %
  d>d_*\kappa_*,     \quad d_*>1,\quad  \kappa_0>\kappa_*\geq1, \quad d_*(\kappa-\kappa_*)>1.
\eeq
Let $\tau_{j+1}=\tau_j^{d_*}$ for $j\geq0$.
There exist   positive numbers $\hat a_0$, $ \hat\tau_0$ in $(0,1)$, which depend only on $\hat L_0,\kappa,\kappa_0,\kappa_*,d,d_*, P$ and  $P_0$  so that if $0<\tau_0\leq\hat\tau_0$ then $\tau_j\leq e^{-P_0(j)}$, and furthermore if $0\leq L_0\leq\hat L_0$, $0\leq a_0\leq \hat a_0$  and $a_j, L_j$ are two sequences of nonnegative  numbers satisfying
\ga\label{aj+1}
  a_{j+1}\leq K_j(\tau_j^{d}L_j+a_j^{\kappa_0}L_j+\tau_j^{-1}a_j^\kappa),\quad j\geq0,\\
\label{Lj+1}
 L_{j+1}\leq K_j(1+\tau_j^{-1}a_j)L_j, \quad j\geq0,
\end{gather}
then
for $j\geq0$ and $P_1(j)=j+P(1)+\dots+P(j)$, we have
\eq{ajLj0}
a_{j}\leq \tau_{j}^{\kappa_*}, \quad L_{j+1}\leq L_0 e^{P_1(j)}.
\eeq
\epr
\begin{rem} Condition \rea{goldenratio} ensures the existence of $\kappa_*,d_*$ that satisfy \rea{ksds}.
\end{rem}
\begin{proof} Suppose that $0<\tau_0<1$.
By the definition of $\tau_j$,   $\tau_j=\tau_0^{d_*^j}$. By \rl{defrapid}, we have
$$
\tau_j^{\kappa_*-1}\leq 1,\quad \tau_j\leq e^{-P_0(j)},$$
 when $\tau_0\leq \hat\tau_0$. Here $\hat\tau_0$ depends on $P_0$. Suppose that we can verify the first inequality in \re{ajLj0}.
   From \re{Lj+1} we also have
$$
L_{j+1}\leq 2K_jL_j\leq e^{1+P(j)}L_{j}\leq   e^{{P_1(j)}}L_0.
$$
This shows the second inequality in \re{ajLj0}. For the first inequality,  we get $a_0\leq\tau_0^{\kappa_*}$ by requiring $\hat a_0\leq r_0^{\kappa_*}$.  By \re{aj+1}  we deduce
\ga \nonumber %
\tau_{j+1}^{-\kappa_*}a_{j+1}\leq \tau_j^{-\kappa_*d_*}\left\{L_0e^{P(j)
+P_1(j)}(\tau_j^d+\tau_j^{\kappa_0 d_*})+e^{P(j)}\tau_j^{\kappa d_*-1}\right\}.
\end{gather}
For the three components of $\tau_j$, we have $d-\kappa_*d_*>0$, $(\kappa_0-\kappa_*)d_*>0$, and $\kappa d_*-1-\kappa_*d_*=d_*(\kappa-\kappa_*)-1>0$. This shows that $\tau_j^{\kappa-\kappa_*d_* }$, $\tau_j^{(\kappa_0-\kappa_*)d_* }$, and $\tau_j^{(\kappa -\kappa_*)d_*-1}$ converge to zero rapidly.   By \re{defrapid}, we have
$$
\tau_j^{-\kappa_*d_*}\left\{L_0e^{P(j)
+P_1(j)}(\tau_j^d+\tau_j^{\kappa_0 d_*})+e^{P(j)}\tau_j^{\kappa d_*-1}\right\}\leq1,
$$
which is achieved by choosing $\hat\tau_0$ that depends on $\kappa_0,\kappa,\kappa_*,d,d_*, \hat L_0,P,\kappa_0$, as well as $P_0$ indicated early.
Note that $\hat a_0$ also depends on $\hat\tau_0$.
We have proved that $\tau_{j+1}^{-k_*}a_{j+1}\leq1$.
\end{proof}

%

Let us restate  \rt{thm1} here. We also organize the statements in the order of the proof.
Recall that $r,s$ must satisfy \re{rsrs}.
\th{thm1+}
Let $\{S^\la\}\in \cL C^{r,s}$ be as in \rta{thm1}.
Then we can find $\{F^\la\}\in \cL C^{1,0}$ satisfies the assertion in \rta{thm1}. Furthermore, we have
 \bppp
 \item
 $\{F^\la \}\in\cL C^{r_-,s}(  U )$, provided
  $  r>s+2\in\nn$.
\item $\{F^\la \}\in\cL C^{\infty,s}(  U )$,  provided $r=\infty$, and $ s\in [1,\infty]$.
  \item
 $\{F^\la \}\in\cL C^{r_-,s_-}(  U )$, provided   $ r\geq s+3\geq4$ and $ \{r\}\geq\{s\}>0$.
 \item $\{F^\la \}\in\cL C^{r_-,r_-}(  U )$, provided $r=s>1$.
 \item
 $\{F^\la \}\in\cL C^{r_-,0}(  U )$, provided
  $r>1$.  \eppp
\eth

\begin{proof}
For the first three parts we will not apply the Nash-Moser smoothing operator to
 the parameter $\la$. The integrability condition \re{6p54m} already involves a derivative in $x\in\rr^N$. Therefore, for the iteration method to work we requires that $\{S^\la\}\in \cL C^{r,s}$ with $r\geq s+1$.

 \medskip

 ({\it i }\!\!)
 Let us recall the approximation for a change of coordinates described at the end of section~\ref{sec6}.
Recall that $\A$ is defined by \re{AdA} and \re{a73} for the adapted frame of $S$. Applying the transformation $H$ defined by $H=I+h_\tau$
in \re{6p46}-\re{6p47}, we obtain a new complex Frobenius structure $\widehat S=H_*S$ of which the corresponding $\widehat{\A}$ has
the form
\al \label{6p50+m}
(\cL C\widehat{\A })\circ H
= I_1+I_2+I_3.
\end{align}
Recall from \re{6p51} that  $\cL C\widehat{\A}=(C\widehat{\B},\hat{\mathfrak b})$ and
\ga
 \nonumber %
I_1=(I-S_\tau )\A,\quad   I_2=S_\tau
\cL T_\rho{\mathbf d}\A,\quad  I_3= \langle\A,\pd S_\tau \cL T_\rho\A\rangle.
 \end{gather}
By \re{717},  the matrix $C$ satisfies
$$
I_4:=C\circ H-I
 =  \langle\A,\pd S_\tau \cL T_\rho\A\rangle.
 $$
 By \re{6p54}, the integrability condition has the form
 \ga
\mathbf d\A=\langle\psi(\A)\A,\pd\A\rangle.
\label{6p54m}\end{gather}
The estimates for
 $I_4$ and $I_3$ will be identical. By \re{6p50+m}, we obtain
 \eq{i4321} \nonumber %
 \widehat{\A}\circ H=(I+I_4)^{-1}(I_1+I_2+I_3).\eeq

 We first assume that
\eq{irs1}\nonumber %
\infty>r\geq s+1,\quad \{r\}\geq\{s\}.
\eeq
In what follows, we assume that
\eq{hyper0}
\|\A\|_{\rho;s+1,s}\leq K^{-1}_0(s,\theta)\tau.
\eeq
Here $
 K_0(r_0,\theta):=2{K(r_0+2,\theta)}\cdot\f{C_{N}}{\theta}$, which is analogous to \re{dfK0}. The factor $\tau$
 is needed in reflecting the need of applying Nash-Moser smoothing operator.
We need to derive estimates on shrinking domains. Set
$$\tilde\rho_i=(1-\theta)^i\rho, \quad i=1,2,3.
$$
To apply estimates for $S_\tau$, we assume that
\eq{taurho}
0<\tau<\rho\theta/C_0.
\eeq
By   estimate \re{Pqva} on
$D_{\tilde\rho_1}$ for the homotopy formulae and estimate  on $D_{\tilde\rho_2}$ for $S_\tau$, we obtain
\eq{stTr}
|S_\tau \cL T_\rho\A|_{\tilde\rho_2;\ell,a}\leq   K(\ell+1)|\A|_{\rho;\ell,a}.
\eeq
Therefore, we have
\eq{estI0}\nonumber %
\|I_1\|_{\tilde\rho_1;0,0} \ple\|\cL A\|_{\rho;0,0}, \quad  \|I_i\|_{\tilde\rho_2;0,0} \leq K(1)\|\cL A\|_{ \rho;0,0} \|\cL A\|_{ \rho;1,0},\quad 2\leq i\leq 4.
\eeq
 Note that $|I_4|_{\tilde\rho_2;0,0}\leq1/{(2N)}$ and $|I_i|_{\tilde\rho_2;0,0}\leq1.$  Applying the product rule in \rl{Aff-}, we obtain
 \al\label{hAHms}\nonumber %
 \|\widehat{\cL A}\circ H\|_{\tilde\rho_2;r,s}&=\|(I+I_4)^{-1}(I_1+I_2+I_3)\|_{\tilde\rho_2;r,s}\\
 &\ple\sum_{i=1}^3 \|I_i\|_{\tilde\rho_2;r,s}+\|I_i\|_{\tilde\rho_2;b,s}|I_4|_{\tilde\rho_2;0,0}+\|I_i\|_{\tilde\rho_2;0,0}
 \|I_4\|_{\tilde\rho_2;r,s}\nonumber
 \\
 &\quad +\sum_{i=1}^3\|I_i\|_{\tilde\rho_2;0,0}Q_{\tilde\rho_2,\tilde\rho_2;r,s}(I_4,I_4)+Q_{\tilde\rho_2,
 \tilde\rho_2;r,s}(I_i, I_4)\nonumber
 \\ &\ple\sum_{i=1}^4\|I_i\|_{\tilde\rho_2;r,s}+Q_{\tilde\rho_2,\tilde\rho_2;r,s}(I_i,I_4). \nonumber
 \end{align}
Recall that  $h=-S_\tau \cL T_\rho\A+S_\tau \cL T_\rho\A(0)$.
By \re{stTr},
 \eq{fms}
|h|_{\tilde\rho_2;\ell,a} \leq 2K(\ell+1)|\A|_{\rho;\ell,a}.
\eeq
Then \re{hyper0} and \re{fms} imply
\eq{h1ss}
\|h\|_{\tilde\rho_2;s+1,s}\leq\theta/{C_N}.
\eeq

To simplify notation, let $\widetilde\A=\widehat A\circ H$.
Write $\widehat \A=\widetilde\A\circ H^{-1}$.  By \re{h1ss} we can apply \re{AuF-1}, and by \re{fms}  we get from \re{AuF-1}
\al
&\|\widehat\A\|_{\tilde\rho_3;r,s}\leq C_r\left\{\|\widetilde\A\|_{\tilde\rho_2;r,s}
+|\widetilde\A|_{\tilde\rho_2;1,s_*} \|\A\|_{\tilde\rho_2;r,s}+\|\widetilde\A\|_{\tilde\rho_2;s+1,s}\|\A\|_{\tilde\rho_2;r,s_*}\label{hAbs}\right.\\
 & \quad \quad\left.+\|\widetilde\A\|_{\tilde\rho_2;1,0}\|\A\|_{\tilde\rho_2;s+1,s}\odot\|\A\|_{\tilde\rho_2;r,s_*} +\|\widetilde\A\|_{\tilde\rho_2;r,s_*}\odot\|\A\|_{\tilde\rho_2;s+1,s}\right\}.\nonumber
 \end{align}
%

We need to estimate $\widetilde\A=(I+I_4)^{-1}(I_1+I_2+I_3)$ by the product rule.
By the properties of the smoothing operator, we get
\ga\label{i1Ls}
\|I_1\|_{\tilde\rho_1;\ell,s}\ple \tau^{d_s(b,\ell)}\|\cL A\|_{b,s}, \quad s\leq \ell\leq b<\ell+m_0,\\
\|I_1\|_{\tilde\rho_1;b,s}\ple  \|\cL A\|_{b,s},\quad s\leq b<\infty.
\label{i1Ls+}\end{gather}
Here and in what follows, we write $\|\A\|_{a,b}=\|\A\|_{\rho;a,b}$ and $|\A|_{a,b}=|\A|_{\rho;a,b}$ for simplicity.
 By \re{sftaj} and \re{stfj+},  we   have
\ga\label{i4ea} \nonumber %
|I_4|_{\tilde\rho_2;\ell,a}=|\langle\A,\pd S_\tau \cL T_\rho\A\rangle|_{\tilde\rho_2;\ell,a}\leq K(\ell)\tau^{-1}(|\A|_{\ell,a}|\A|_{0,0}+|\A|_{\ell,0} |\A|_{0,a}).
\end{gather}
Recall that $
\widehat Q_{b,s}(\A,\A)=\|A\|_{b,s}\|\A\|_{0,0}+Q_{b,s}(\A,\A)$. Since $I_3$ has the same form as $I_4$, then we have verified for $i=3,4$
\aln 
\|I_i\|_{\tilde\rho_2;r,s}&\leq K(r+2)\tau^{-1}(\|\A\|_{r,s}|\A|_{0,0}+|\A|_{r,0}\odot |\A|_{0,s})\\
\|I_i\|_{\tilde\rho_2;s+1,s}&\leq K(s+2)\|\A\|_{s+1,s}.\nonumber
\end{align*}
Here the second inequality follows from the first one and  $|\A|_{s+1,s}\leq K_0^{-1}(1)\tau$ by \re{hyper0}.

The estimates for $I_2$ need the integrality condition
 $d\A=\langle \psi(\A)\A,\pd\A\rangle$.  By \re{recallq+}  the latter gives us
\al
\label{7p2-}|I_2|_{\tilde\rho_2;\ell,a}&=|S_\tau
\cL T_\rho\mathbf d\A|_{\tilde\rho_2;\ell,a}\leq K(\ell)\tau^{-1}|
\cL T_\rho\mathbf d\A|_{\tilde\rho_1;\ell-1,a}\\
\nonumber & \leq
K(\ell)\tau^{-1}
\Bigl\{|\A|_{\ell-1,a}|\A|_{1,0}+ |\A|_{0,0}|\A|_{\ell,a}\\
\nonumber &\quad+|\A|_{\ell-1,0}\odot|\A|_{1,a}+|\A|_{0,a}\odot|\A|_{\ell,0}\Bigr\}\\
\nonumber & \leq K(\ell) \tau^{-1}
(|\A|_{\ell,a}|\A|_{0,0}+|\cL A|_{\ell,0}\odot\A|_{0,a}),\quad \ell\in[1,\infty),\\
\quad|I_2|_{\tilde\rho_2;0,s}&\leq K(1)(|\A|_{1,s}|\A|_{0,0}+|\A|_{0,s}|\A|_{1,0}).\nonumber
\end{align}
In particular, with $\|\A\|_{s+1,s }\leq\tau$  we have
\eq{i21s} \nonumber %
\|I_2\|_{\tilde\rho_2;s+1,s} \leq K(s+1) \|\A\|_{s+1,s}.
\eeq
Since $
[r]\geq [s]+1$, from \re{7p2-} we  conclude
\al\nonumber
\|I_2\|_{\tilde\rho_2;r,s}&
\leq K(r)
\tau^{-1}(\|\A\|_{r,s}|\A|_{0,0}+|\A|_{r,0}\odot |\A|_{0,s}). 
\end{align}
Since $|I_i|_{\tilde\rho_2;0}\leq 1/2$ and $\widetilde\A=(I+I_4)^{-1}(I_1+I_2+I_3)$, we have
\al
\label{hAbs3} \|\widetilde\A\|_{\tilde\rho_2;r,s}&\ple \sum_{i=1}^4 \| I_i\|_{\tilde\rho_2;r,s}+ \sum_{i=1}^4  (\|I_i\|_{\tilde\rho_2;0,s}   \|I_4|_{\tilde\rho_2;r,0}+\|I_4\|_{\tilde\rho_2;0,s}   \|I_i|_{\tilde\rho_2;r,0})\\
&\ple
\nonumber
\|I_1\|_{\tilde\rho_2;r,s}+K(r+1) \tau^{-1}(\|\A\|_{r,s}|\A|_{0,0}+|\A|_{r,0}\odot |\A|_{0,s}).
\end{align}
Using $|I_i|_{\tilde\rho_2;s+1,s}\leq K(s+2)\|\A\|_{s+1,s}$ and \re{i1Ls+}, we   obtain its first consequence
\ga
\label{ha1s}\|\widetilde\A\|_{\tilde\rho_2;r,s}\leq K(r+1)\| \A\|_{r,s}.
\end{gather}
By  \re{hAbs}, we immediately obtain
\eq{Hder} \nonumber %
\|\widehat{\cL A}\|_{\tilde\rho_3;r,s}\leq K(r+1)\|{\cL A}\|_{r,s}.
\eeq
We can also take $b=s+1$ in \re{hAbs} and use \re{ha1s} to simplify all quadratic terms on the right-hand side of \re{hAbs}. We conclude
 $$
\|{\A}\|_{\tilde\rho_3;s+1,s}\leq C_r \|\widetilde\A\|_{\rho_1;s+1,s}   + K(s+2) \|\A\|_{s+1,s}^2.
$$
To improve it, we use \re{hAbs3} again by taking $b=s+1$. We now use \re{i1Ls} in which $\ell=s+1$ to achieve
   \al\label{hAbs2}
\|\widehat\A\|_{\tilde\rho_3;s+1,s}&\leq   C_r\tau^{d_s(r,s+1)}\|\cL A\|_{r,s} + K(s+2) \tau^{-1}\|\A\|_{s+1,s}^2.
 \end{align}
We have derived necessary estimates. We will use \re{hAbs2} to get rapid convergence in $\cL C^{s+1,s}$ norms, while \re{ha1s} is for controlling the high order derivative in  linear growth.  This allows us to achieve rapid convergence for intermediate derivatives via interpolation.

\medskip

Assume now that
$
d:=d_s(r,s+1)>1$, and $s$ is a positive integer.


We briefly recall part of the proof of \rp{pnl} for nesting domains. We need to find a sequence of transformations $F_j$
so that $\tilde F_i^\la :=F^\la _i\circ\dots\circ F^\la _0$ converges to $\tilde F_\infty^\la $, while $S_{i+1}^\la :=(\tilde F^\la_i)_*S_i^\la $, with $S_0^\la $ being the original structure, converges to the standard complex Frobenius structure in $\rr^N$. Set for $i=0,1,\dots$
\ga
\rho_i=\frac{1}{2}+\frac{1}{2(i+1)}, \quad
\rho_{i+1}=(1-\theta_i)^3\rho_i.\nonumber
\end{gather}
Thus, we take $1-\theta_k^*=(1-\theta_k)^3$ in \rl{Ask2k}.
Let $\rho_\infty=\lim_{i\to\infty}\rho_i$.  Recall that
$$
D_\rho=B^{2n}_\rho\times B^M_\rho\times B_\rho^L.
$$
 We have
    for $\rho_0/2<\rho<2\rho_0$
\eq{nestH}
H_i\colon D_{(1-\theta_i) \rho}\to D_{\rho},\quad H^{-1}_i\colon D_{(1-\theta_i)^2\rho}\to D_{(1-\theta_i) \rho}
\eeq
provided $\tau_i$ satisfies \re{taurho}, and \re{hyper0} holds for $\A_i$, i.e.
\ga\label{taurho+}
\tau_i\leq\rho_i\theta_i/C_0,\\
\label{4p6j}
  \|\A_i\|_{\rho_i;s+1,s}\leq K^{-1}_0(s+2,\theta_i)\tau_i,\quad K_0(s+2,\theta_i):=  \f{4K(s+2,\theta_i)}{(\rho_i\theta_i)^{s+2}}.
\end{gather}
Then \re{ha1s} and \re{hAbs2} take the form
\ga\label{Lderj} 
\| \A_{i+1}\|_{\rho_{i+1};s+1,s}\leq  K(s+2,\theta_i)(\tau_i^{d }\|\A_i\|_{\rho_i;r,s}+\tau_i^{-1}\|\A_i\|^2_{\rho_i;s+1,s}),\\
\label{Hderj}
\| {\cL A}_{i+1}\|_{\rho_{i+1};r,s}\leq K(r+1,\theta_i) \|{\cL A}_i\|_{\rho_i;r,s}.
\end{gather}
We want to apply \rp{NashMoser} in which $a_i=\|\A_i\|_{\rho_i;s+1,s}$ and $L_i=\|\A_i\|_{\rho_i;r,s}$.
Obviously, there are polynomials $P,P_0$  satisfying
$$
K(r,\theta_i)\leq e^{P(i)}, \quad  K_0(s+2,\theta_{i})\leq e^{P_0(i)}.
$$
Applying \rp{NashMoser} to \re{Lderj}-\re{Hderj}, we find   $d_*>1$, $0<\tau_0=\hat\tau_0<1$ such that for $\tau_{j+1}=\tau_j^{d_*}$,  we have
\ga\label{aiss}
\|\A_{i}\|_{\rho_{i};s+1,s}\leq \tau_{i}^{\kappa_*}
 \leq K_0^{-1}(s+2,\theta_{i})\tau_{i}
\end{gather}
for some $\kappa_*>1$, provided we have
\eq{A0ha}
\|\A_0\|_{\rho_0;s+1,s}\leq\hat a_0.
\eeq
Here $\hat a_0$ and $\hat\tau_0$ are fixed now and they depend only on the polynomials $P$,  $P_0$, and $d$. We have achieved \re{A0ha} in \re{initiale} via an initial normalization and a dilation. Therefore, we have achieved \re{aiss} for all $i$.
 By \re{fms}, \re{Hderj} and \rl{defrapid}, we know that $\|h_i\|_{\rho_i;r,s}$ has  linear growth, i.e.
 \eq{hiHder}\nonumber %
 \|h_{i}\|_{ \rho_\infty;r,s}\leq  e^{P_r(i)}
 \eeq
 for some polynomial $P_r$.
  By \re{fms} and \re{aiss}, we have
 \eq{fmsi}\nonumber %
 \|h_{i}\|_{\rho_\infty;s+1,s}\leq K(s+2,\theta_i)\tau_i^{\kappa_i}.
  \eeq
 We know consider the convergence of $\tilde h_{i}^\la :=\tilde H_{i}^\la -\tilde H_{i-1}^\la =h_{i}\circ\tilde H_{i-1}^\la $. We have
 \aln
&\|\{h_{i}^i \circ H_{i-1}^\la \circ\cdots\circ H_0^\la \}\|_{\f{1}{2}\rho_\infty;r,s}\leq C_r^{i+1}  \Bigl\{\|h_{i}\|_{\rho_{i};r,s} \\
&\quad \left.+ |h_{i}|_{\rho_{i};1,0} \sum_{j\leq k<i} \|h_j\|_{\rho_{j};s+1,s}\odot\|h_k\|_{\rho_{k};r,s_*}
+|h_{i}|_{\rho_{i};1,s_*}\sum_{j<i} \|h_j\|_{\rho_j;r,s }
 \right.\\
&\quad +\sum_{j<i}\left\{\|h_{i}\|_{\rho_{i};s+1,s}\odot\|h_j\|_{\rho_{j};r,s_*}+\|h_j\|_{\rho_{j};s+1,s}
\odot\|h_{i}\|_{\rho_{i+1};r,s_*}\right\}\Bigr\}.
\end{align*}
This shows that $\|\tilde h_{i}\|_{ \rho_{\infty}/2;r,s}$ has  linear growth.
The above formula also holds when $r=s+1$. Then we see that $\|\tilde h_{i}\|_{\rho_{\infty}/2;s+1,s}$ converges rapidly.
We now consider intermediate H\"older estimates. Assume that $s+1<\ell<r$. By  interpolation in $x$ variables, we have
\al\label{aimp}\nonumber %
|\tilde h_{i}|_{\rho_{\infty}/2;\ell-i,i}&\leq C_r|\tilde h_{i}|_{\rho_{\infty}/2;s+1-i,i}^{\theta}|\tilde h_{i}|_{\rho_{\infty}/2;r-i,i}^{1-\theta}
\end{align}
for positive numbers $\theta=\f{r-\ell}{r-s-1}$. This shows that $\tilde H_i^\la $ converges in $\cL C^{\ell,s}$ norm for any $\ell<r$.
That $\tilde H_i^\la $ converges to a $\cL C^{1,0}$ diffeomorphism can be proved in a way similar to the proof of \rp{pnl}.

\medskip

 ({\it ii }\!\!) Assume that $r=\infty>s$.
We first apply the above arguments to  $r=s+3$. This gives us the rapid convergence of $\|\tilde h_i\|_{\rho_i;s+1,s}$. We also apply the above estimates to any finite $r = s+\ell$ for any integer $\ell\geq3$.    Then $\|\tilde h_i\|_{\rho_{\infty}/2;r,s}$ has  linear growth.  By interpolation in $x$ variables as before, we get rapid convergence of $\|\tilde h_i\|_{\rho_{\infty}/2;[s]+\ell-1,s}$ for $\ell=3,4,\dots$.
%

We remark that the case $r=s=\infty$ needs to be treated differently as \re{hAbs3} is too weak to conclude that $\|\A_i\|_{\rho_\infty; r,s}$ has  linear growth as we cannot have \re{hyper0} for all $s$. Therefore, we will treat
the case in  ({\it iv}).

\medskip

({\it iii}) Assume that $s$ is finite. Since
$$d_{[s]}(r,[s]+1)>1,$$
we can apply the above arguments in the first case to get rapid convergence of $\tilde h_i$
in $\cL C^{r',[s]}$ norm for any $r'<r$.
Since $\tilde h_i$ has  linear growth
in $\cL C^{r,s}$ norm.  By interpolation in $t$ variable
\al
\nonumber %
|\tilde h_{i}|_{\rho_{\infty}/2;[r]-i,i+\beta}&\leq C_r|\tilde h_{i}|_{\rho_{\infty}/2;[r]-i,i }^{1-\beta/{\{s\}}}|\tilde h_{i}|_{\rho_{\infty}/2;[r]-i,i+\{s\}}^{\beta/{\{s\}}}.
 \end{align}
We obtain the rapid convergence of $|\tilde h_i|_ {\rho_{\infty}/2;[r]-j,j+\beta}$ for any $\beta<\{s\}$.

\medskip

 ({\it iv }\!\!)  We consider   the    case with $\infty\geq r=s>1$. We will apply smoothing operator
in all variables, including the parameter $\la$. This requires us to re-investigate the use of the homotopy formula. Recall that $D_\rho=B_{\rho}^{2n}\times B_\rho^M\times B_\rho^L$ and $\tilde\rho_i=(1-\theta)^i\rho$. We will assume that $1/4<\rho<4$.
Set $|u|_{\rho;a}=|u|_{D_\rho;a}$ and $|\A|_{\rho;a}=|\A|_a$.

We first assume that $r$ is finite.
We use a smoothing operator $S_\tau$ in $(z,t,\xi,\la)$ variables on $D_\rho\times(-1,2)$. We first apply the Seeley extension operator $E$ in parameter $\la$ and then $S_\tau$ to define
\ga\nonumber %
h_\tau=(f_\tau,g_\tau):=-S_\tau  E \cL T_\rho\A+S_\tau E \cL T_\rho\A(0).
\end{gather}
On shrinking domains $D_{\tilde\rho_1 }\times[0,1]$, we have
\begin{equation}
[S_\tau,{\mathbf d}]=0, \quad [E,{\mathbf d}]=0.\nonumber
\end{equation}
The vanishing of last commutator follows from the linearity of the Seeley extension    \re{seeley} which is applied in the $\la$ variable.
 We obtain for $0\leq \la\leq1$
\aln 
\A+ \mathbf dh_\tau  &=\A -\mathbf dS_\tau E\cL T_\rho\A =\A -S_\tau E\mathbf d\cL T_\rho\A=\A -S_\tau E\A+S_\tau E\cL T_\rho\mathbf d\A\\
&=(I-S_\tau )E\A+S_\tau E\cL T_\rho\mathbf d\A.
\nonumber\end{align*}
Here the  last identity is obtained by the homotopy formula and   $E^\la u^\la =u^\la $ for $0\leq \la\leq1$.
We  now express
\al \label{r6p50+m}
(\cL C\widehat A )\circ H
= \A + \mathbf dh_\tau+\langle\A, \pd h_\tau\rangle=I_1+I_2+I_3
\end{align}
where $\cL C\widehat{\A}=(C\widehat{\B},\hat{\mathfrak b})$ and $I_i$ are now given by
\ga
\nonumber %
I_1=(I-S_\tau )E\A,\quad   I_2=S_\tau
E\cL T_\rho{\mathbf d}\A,\quad  I_3= \langle\A , \pd  S_\tau E\cL T_\rho\A\rangle,\\
C\circ H-I=I_4, \quad I_4=\langle \A,\pd S_\tau E\cL T_\rho\A\rangle.\nonumber
\end{gather}
Here $
I_{4,\ov\all\beta}=-(A_{\ov\all \gaa}\pd_\gaa  +
B_{\ov\all \ell}\pd_{\ell}'')S_\tau E(\cL T_\rho\A)_\beta$.   Each $I_i$ contains the extension and smoothing operators.   By \re{6p54}, the integrability condition has the form
\ga\nonumber %
\mathbf d\A^\la =\langle\psi(\A^\la )\A^\la ,\pd\A^\la \rangle, \quad 0\leq \la\leq1.
\label{r6p54m}\end{gather}
We will use
 \eq{ri4321} \nonumber
 \widehat{\A}\circ H=(I+I_4)^{-1}(I_1+I_2+I_3).\eeq
 From \re{seeley}, we know that the extension $E$ does not depend on the domain $D_\rho$. Thus we have
 $
 |Eu|_{r}\leq C_{r}|u|_{\rho;r},
 $
 where $|Eu|_{r}$ is computed on $\rr^N\times(-1,2)$ and $C_{r}$ does not depend on shrinking domains. Although the extension operator does not preserve  the integrability condition when $\la$ is outside $[0,1]$, we still have
 \eq{}
  |EP_\rho\mathbf d\A|_{\tilde\rho_1; a}\leq C_{ a}|P_\rho\mathbf d\A|_{\tilde \rho_1;a}\leq C_a(\rho\theta)^{- a}|\langle\psi(\A^\la )\A^\la ,\pd\A^\la \rangle|_{ \rho;a}.\nonumber
 \eeq
 This suffices our proof. The reader is also referred to Nijenhuis and Woolf~\ci{NW63} where integrability condition is relaxed to
 differential inequalities.

\medskip

  The rest of the proof is much simpler. We can  simplify the iteration procedure. In a simpler way
  we will   establish  the rapid convergence in the   $\cL C^0$ norms. This will avoid loss of regularity.

Set $|\cL A|_{r}=|\cL A|_{\rho;r}$.  Let us start with
 \eq{}
|I_1|_{\tilde\rho_2;0}\ple \tau^{r}|\cL A|_{r}, \quad |I_1|_{\tilde\rho_2;r}\ple|\A|_{r}.\nonumber
\eeq
Applying estimates \re{3p13} and \re{3p70} for the homotopy formulae,
 we get
\al\nonumber %
|I_2|_{\tilde\rho_3;0}&= |S_\tau E
\cL T_\rho{\mathbf d}\A|_{\tilde\rho_3;0}\leq K(1)
|\A|_{1}|\A|_{0}\leq K(1)|\A|_0^{2-\f{1}{r}}|\A|_r^{\f{1}{r}},
\end{align}
where the last inequality is obtained by interpolation \re{1-theta}. We also have
\al\nonumber %
|I_2|_{\tilde\rho_3;r}&\ple\tau^{-1}|
E\cL T_\rho{\mathbf d}\A|_{\tilde\rho_2;r-1}\leq K(r)\tau^{-1}(
|\A|_{1}|\A|_{r-1}+|\A|_r|\A|_0)\\
&\leq K(r)\tau^{-1}|\A|_r|\A|_0.\nonumber
\end{align}
By \re{sftaj} and \re{stfj+},  we   have
\aln
|DS_\tau E\cL T_\rho\cL A|_{\tilde\rho_3;0}&\leq|S_\tau E\cL T_\rho\cL A|_{\tilde\rho_3;1}\leq C\tau^{-1}|E\cL T_\rho\cL A|_{\tilde\rho_2;0}\leq
 K(0)\tau^{-1}|\cL A|_{0},\\
|DS_\tau E\cL T_\rho\cL A|_{\tilde\rho_3;r}&\leq |S_\tau E\cL T_\rho\cL A|_{\tilde\rho_3;r+1}\leq K(r)\tau^{-1}|  \A|_{r}.\end{align*}
Thus by \re{estprod}, we get for $i=3,4$
\al
\nonumber %
|I_i|_{\tilde\rho_3;0}&\leq C|\langle\A, D S_\tau E \cL T_\rho\A\rangle
|_{\tilde\rho_3;0}\leq K(1)\tau^{-1}|\A|_{0}^2,\\
|I_i|_{\tilde\rho_3;r}&\leq C_r|\langle\A, D S_\tau E \cL T_\rho\A\rangle
|_{\tilde\rho_3;r}\leq K(r)\tau^{-1}|\A|_r|\A|_0.\nonumber
 \end{align}
Therefore,
\ga \label{chata}
|(C\widehat{\A})\circ H|_{\tilde\rho_3;0}\leq
K(r)\left(\tau^{r}|\A|_{r}+|\A|_0^{2-\f{1}{r}}|\A|_r^{\f{1}{r}}+
\tau^{-1}|\A|_{0}^2\right),\\
|(C\widehat{\A})\circ H|_{\tilde\rho_3;r}\leq K(r)\left(
 |\A|_r+\tau^{-1}|\A|_{0}|\A|_r\right).\nonumber
\end{gather}

Assume that  $r>1$. We first require that $|I_i|_{\tilde\rho_3;0}\leq 1/{C_N}$. Thus we assume that
\eq{A0Ar}
\tau^{-1}|\A|_{0}+ |\A|_0^{2-\f{1}{r}}|\A|_r^{\f{1}{r}}\leq 1/{(C_NK(1))}.
\eeq
By the product rule in \rl{Aff-},  \re{chata}, and above estimates for $I_i$, we get
\begin{gather}\label{hAh1}
|\widehat{\A} \circ H|_{\tilde\rho_3;0}\leq K(r)\left(\tau^{r}|\A|_{r}+|\A|_0^{2-\f{1}{r}}|\A|_r^{\f{1}{r}}+
\tau^{-1}|\A|_{0}^2\right),\\
|\widehat{\A} \circ H |_{\tilde\rho_3;r}\leq K(r)(|\A|_r+ \tau^{-1}|\A|_{0}|\A|_r)\leq 2K(r)|\A|_r.\nonumber
\end{gather}
 Thus $h=H-I$ satisfies
\ga\label{hrh1m}
|h|_{\tilde\rho_2;m}=|S_\tau \cL T_\rho\A|_{\tilde\rho_2;m}\leq K(m)|\A|_{\rho;m}.
\end{gather}
Interpolating again, we have $|\A|_{1}\leq C_r|\A|_0^{1-\f{1}{r}}|\A|_r^{\f{1}{r}}$. Assume now that
\eq{876n}
|\A|_0^{1-\f{1}{r}}|\A|_r^{\f{1}{r}} \leq K_0^{-1}(1)\theta/{C_{N}}
\eeq
with $K_0(1)\geq K(1)$.
Then $|h|_{\tilde\rho_2;1}\leq \theta/{C_N}$, and by \re{3p6} we obtain
\al\nonumber
|H^{-1}-I|_{\tilde\rho_3;1}&\leq 2|H-I|_{\tilde\rho_2;1}\leq K(1)|\A|_0^{1-\f{1}{r}}|\A|^{\f{1}{r}}_r
 \leq 1/2.
\end{align}

By \re{hAAf} and $H^{-1}\colon D_{\tilde\rho_4}\to D_{\tilde\rho_3}$,  we have
\eq{hAAh}
|\widehat\A|_{\tilde\rho_4;r}\ple|\widehat\A\circ H|_{\tilde\rho_3;r}+|\widehat\A\circ H|_{\tilde\rho_3;1}|h|_{\tilde\rho_3;r}\leq K(r)|\A|_{r}.
\eeq
By \re{hAh1}, we have
\eq{hAAh+}
|\widehat{\A}  |_{\tilde\rho_4;0}\leq K(r)\left(\tau^{r}|\A|_{r}+|\A|_0^{2-\f{1}{r}}|\A|_r^{\f{1}{r}}+
\tau^{-1}|\A|_{0}^2\right).
\eeq

The rest of proof is analogous to the previous cases. We will be brief. Let
 \ga\nonumber
\rho_i=\frac{1}{2}+\frac{1}{2(i+1)}, \quad
\rho_{i+1}=(1-\theta_i)^4\rho_i.
\end{gather}
Thus, we take $1-\theta_k^*=(1-\theta_k)^4$ in \rl{Ask2k}.
 We return to the sequence $H_i$ which needs to satisfy \re{nestH}. Let $S_i^\la $ be the sequence of the structures and let $\A_i$ be the adapted $1$-forms of $S_i^\la $. Suppose that \re{A0Ar} and \re{876n} hold for $\A_i$, i.e.
\ga\label{A0Ari}
K_0(1)\tau_i^{-1}|\A_i|_{\rho_i;0}\leq 1/{C_N},\quad |\A_i|_{\rho_i;0}^{1-\f{1}{r}}|\A_i|_{\rho_i;r}^{\f{1}{r}} \leq K_0^{-1}(1)\theta_i/{C_N}.
\end{gather}
Then we have by  \re{hAAh}-\re{hAAh+}
\ga\label{Ai+1}
|\A_{i+1}|_{\rho_{i+1};r}\leq K(r)  |\A_i|_{\rho_{i};r},\\
| \A_{i+1}|_{\rho_{i+1};0}\leq K(r)\left(\tau_i^{r}|\A|_{\rho_{i};r}+|\A_i|_{\rho_{i};0}^{2-\f{1}{r}}|\A_i|_{\rho_{i};r}^{\f{1}{r}}+
\tau_i^{-1}|\A_i|_{\rho_{i};0}^2\right).
\label{Ai+1+}\end{gather}
  We apply \rp{NashMoser} to $a_i:=|\A_i|_{\rho_i;0}$ and $L_i:=|\A_i|_{\rho_i;r}$.  By \re{ajLj0} we get
 \eq{aiLi}
 a_i\leq a_i^*:=\tau_i^{\kappa_*}, \quad L_i\leq L_i^*:= 2^iL_0e^{P_1(i)},
  \eeq
 provided $\tau_0\leq\hat\tau_0$. Here $\kappa_*>1$ and it depends only on $r$. Recall from \rp{NashMoser} that $\tau_{i+1}=\tau_i^{d_*}$ for some $d_*>1$. We may choose $\tau_0$ so small that
 \re{aiLi} implies  \re{taurho+} and \re{A0Ari}.   Therefore, using the initial normalization and dilation we first obtain \re{A0Ari}  for $i=0$. Then we have \re{Ai+1}-\re{Ai+1+} for $i=0$. By \rp{NashMoser}, we have \re{aiLi} for $i=0,1$ and hence \re{A0Ari} for $i=0,1$. 
 We repeat this procedure to get \re{A0Ari}-\re{Ai+1+} for all $i$. Thus each $H_i$ satisfies \re{hrh1m}, which now has the form
 \eq{him}
 |h_i|_{(1-\theta_i)^2\rho_i;m} \leq K(m)|\A_i|_{\rho_i;m},\quad m\leq r.
 \eeq
By \re{aiLi} and interpolation, we get rapid convergence of $|\A_i|_{\rho_i;m}$ for any $m<r$. Hence we have rapid convergence of $|\tilde h_i|_{\rho_\infty/2;m}$ for any $m<r$.  This proves the result when $r$ is finite.

When $r=\infty$, we apply the above argument to $r=2$ to obtain the rapid convergence of $|\tilde h_i|_{\rho_\infty/2;a}$ for any $a<2$. We still have \re{Ai+1} and \re{him} for any finite $r$, which are
 obtain via \re{hAAh}-\re{hAAh+}. Thus, $\tilde h_i$ has  linear growth in $\cL C^r$ norm for any finite $r$. By interpolation, we obtain the rapid convergence of $\tilde h_i$ in $\cL C^m$ norm for any $m<r$.

  ({\it v}) In this case, we will applying smoothing operator to the variables $x\in\rr^N$ only. However, the argument
  is identical to the proof of  ({\it  iv}).
\end{proof}

\appendix

\setcounter{thm}{0}\setcounter{equation}{0}
\section{H\"older norms for functions  with parameter
}\label{sec3+}

The main purpose of the appendix is to derive some interpolation properties for H\"older norms defined in domains depending on a parameter.
The interpolation properties were derived in H\"ormander~\ci{Ho76} and Gong-Webster~\ci{GW11}.

  We say that a domain $D$ in $\rr^m$ has the
{\it cone  property} if the following
hold:
(i)
 Given two points $p_0,p_1$ in $D$ there exists
a piecewise  $\cL C^1$   curve $\gamma(t)$ in $D$
such that $\gamma(0)=p_0$ and $\gamma(1)=p_1$,
$|\gamma'(t)|\leq   C_*|p_1-p_0|$ for all $t$
except  finitely many   values. The diameter of $D$ is
less than   $C_*$.
(ii)
For each point $x\in \ov D$, $D$ contains a cone $V$  with vertex $x$,
opening $\theta>C_*^{-1}$ and height $h>C_*^{-1}$.
We will denote by $C_*(D)$ a  constant $C_*>1$  satisfying
 (i) and (ii). A constant $C_a(D)$ may also depend on $C_*(D)$.
 In our applications,  we will apply the inequalities to domains that are products of balls of which the radii are between two fixed numbers. Therefore,  $C_*(D)$ and $C_a(D)$ do not depend on $D$, which will be assumed in the appendix.

If $D$ is a domain of \ps, then the norms on $\ov D$ satisfy
\ga
|u|_{D;(1-\theta)a+\theta b}\leq
C_{a,b}  |u|_{D;a}^{1-\theta}|u|_{D;b}^{\theta },\label{1-theta}
\\
|f_1|_{a_1+b_1}|f_2|_{a_2+b_2}\leq C_{a,b} (|f_1|_{a_1+b_1+b_2}|f_2|_{a_2}+|f_1|_{a_1}|f_2|_{a_2+b_1+b_2}).\label{f1f2}\nonumber %
\end{gather}
Here $|f_i|_{a_i}=|f_i|_{D_i;a}$. 
Throughout the appendix, we always assume that domains $D,D',D_i$ have the cone property.

Let $|u_i|_{a_i}=|u_i|_{D_i;a_i}$. Then we have
\ga
\label{estprod}
\prod_{j=1}^m|u_j|_{d_j+a_j}\leq C^m_{a}\sum_{j=1}^m|u_j|_{d_j+a_1+\dots+a_m}\prod_{i\neq j}|u_i|_{d_i}.
\end{gather}

Let  $f=\{f^\la \}$, $g=\{g^\la \}$ be two families of functions on $\ov D$ and $\ov{D'}$, respectively.
To deal with two H\"older exponents in $x,t$ variables, we introduce notation
\ga \nonumber %
|f|_{D;a,0}\odot |g|_{D';0,b}:=|f|_{D;a,0}|g|_{D';0,[b]}+
 |f|_{D;[a],0}|g|_{D';0,b},\\  \label{qrs} \nonumber %
Q^*_{D,D';a,b}(f,g) := |f|_{D;a,0}\odot |g|_{D';0,b}+ |f|_{D;0,b}\odot|g|_{D';a,0}, \\
\label{qrsfg}
Q_{D,D';r,s}(f,g):=\sum_{j=0}^{\mfl{s}}Q^*_{D,D';r-j,j+\{s\}}(f,g),\quad [r]\geq [s],\\
\hat Q_{D,D';r,s}(f,g):=\|f\|_{D;r,s}|g|_{D';0,0}+
|f|_{D;0,0}\|g\|_{D';r,s}+Q_{D,D';r,s}(f,g).  \nonumber %
\end{gather}
For simplicity, the dependence of $Q,Q^*,\hat Q$ on domains $D,D'$ is not indicated when it is clear from the context.

Throughout the paper, by $A\ple B$ we mean that $A\leq CB$ for some constant $C$.
\le{ff-} Let $r_i,s_i, a_i,b_i\  ( 1\leq i\leq m)$ be non-negative real numbers.
Assume that   $(r_1,\dots,r_m, a_1,\dots, a_m)\in\nn^{2m}$, or $(s_1,\dots,s_m, b_1, \dots, b_m)\in\nn^{2m}$.
Let $D_i$  be a  domain in $\rr^{n_i}$ with the cone property and let $f_i=\{f_i^\la \}$ be a family of functions on $D_i$. Let $|f_i|_{c,d}=|f_i|_{D_i;c,d}$. Assume that $\infty>m\geq2$. Then
\ga
\label{ff-a}
\prod_{i=1}^m|f_i|_{r_i,s_i}  
\leq C^m_{r+s}\left\{\sum_i|f_i|_{r,s}\prod_{\ell\neq i}|f_\ell|_{0,0}
+ \sum_{ i\neq j}Q^*_{ r,s}(f_i,f_j)\prod_{l\neq i,j}|f_\ell|_{0,0}
\right\},\\
\label{ff-b}
\prod_{i=1}^m|f_i|_{r_i+a_i,s_i+b_i}
\leq C^m_{r+s+a+b}\left\{\sum_i|f_i|_{r_i+a,s_i+b}\prod_{\ell\neq i}|f_\ell|_{r_\ell,s_\ell}\right.
\\
\hspace{8em}
+ \left.\sum_{ i\neq j}|f_i|_{r_i+a,s_i}|f_j|_{r_j,s_j+b} \prod_{\ell\neq i,j}|f_\ell|_{r_\ell,s_\ell}
\right\}.\nonumber
\end{gather}
Here   $a=\sum a_i$,  $b=\sum b_i$, etc..
Assume further that $r_i\geq s_i$ and $a_i=b_i=0$ for all $i$. Then
\ga\label{ff-c}
\|f_1\|_{r_1,s_1}\cdots \|f_m\|_{r_m,s_m}\leq C^m_{r,s}\sum_{i\neq j} \hat Q_{ r,s}(f_i,f_j)\prod_{\ell\neq i,j}|f_\ell|_{0,0}.
\end{gather}
\ele
\begin{proof} Let  $k_i,j_i\in\nn$.
Let $\pd^{k_i}_{x_i}\pd^{j_i}_{\la_i}f_i$  denote a partial derivative of $f^\la (x)$ of  order $k_i$ in $x$ and order $j_i$ in $\la$, evaluated at $x=x_i$ and
$\la=\la_i$.
  Let $k=\sum k_i$ and $j=\sum j_i$.

We will prove the inequality by estimating derivatives pointwise and H\"older ratios of derivatives.
  We apply \re{estprod} for domains $D_i$ to obtain
\al\label{pleb}
\prod |\pd^{k_i}_{x_i}\pd^{j_i}_{\la_i} f_i|\ple\prod|\pd_ {\la_i}^{j_i}f_{i}|_{k_i}\ple\sum|\pd_ {\la_i}^{j_i}f_{i}|_{k}\prod_{\ell\neq i}|\pd_ {\la_\ell}^{j_\ell}f_{\ell}|_{0}.
\end{align}
Estimate each term via pointwise derivatives, we can write
\eq{pleb1}
|\pd_ {\la_i}^{j_i}f_{i}|_{k}\prod_{\ell\neq i}|\pd_ {\la_\ell}^{j_\ell}f_{\ell}|_{0}=|\pd_{x_i}^{k'}\pd_ {\la_i}^{j_i}f_{i}|\prod_{\ell\neq i}|\pd_ {\la_\ell}^{j_\ell}f _{\ell}(x_\ell)|
\eeq
for some partial derivative $\pd^{k'}$ with order $k'\leq k$.
While applying \re{estprod} for the domain $I$, we  see the last term is bounded by the right-hand side of \re{ff-a}. Thus, we have verified \re{ff-a} when $r_i,s_j$ are integers.

Let $\beta_i=\{r_i\}$ and $k_i=[r_i]$, or $\beta_i=0$ and $k_i\leq[r_i]$. We estimate the product of H\"older ratios
\al 
\nonumber %
&\prod_{\beta_i>0} \f{|\pd_{y_i}^{k_i}\pd_{\la_i}^{j_i}f_{i}- \pd_{x_i}^{k_i}\pd_{\la_i}^{j_i}f_{i}|}{|{y_i}-{x_i}|^{\beta_i}}
\times\prod_{\beta_i=0}|\pd_{x_i}^{k_i}\pd_ {\la_i}^{j_i}f_{i}|
\ple\sum |\pd_{\la_i}^{j_i} f_{i}|_{r}\prod_{\ell\neq i} |\pd_{\la_\ell}^{j_\ell} f_{\ell}|_{0}.
\nonumber\end{align}
We estimate each term in the last sum by considering pointwise derivatives and H\"older ratios. The former can be estimated by computation similar to \re{pleb}-\re{pleb1}. For the H\"older ratio,
applying \re{estprod}   for  $I$  we can bound
$$
\f{|\pd_{y_i}^{k}\pd_{\la_i}^{j_i}f_{i}- \pd_{x_i}^{k}\pd_{\la_i}^{j_i}f_{i}|}{|{y_i}-{x_i}|^{\beta}}
\times\prod_{\ell\neq i}| \pd_ {\la_\ell}^{j_\ell}f_{\ell}|_0
$$
by the right-hand side of \re{ff-a}.

To verify \re{ff-b} it suffices to  consider the   case with $(s_1,\dots, s_m, b_1,\dots, b_m)\in\nn^{2m}$.
Furthermore, it is easy to reduce it to the case with $s_i=0$, which we now assume.
Let $k_i\leq r_i+a_i$ and $j_i\leq b_i$. We apply \re{estprod} to domains $D_i$ to obtain
\al\label{pleb2}
\prod |\pd^{k_i}_{x_i}\pd^{j_i}_{\la_i}  f_i|
\ple\prod_{\ell|k_\ell\leq r_\ell}
|\pd_ {\la_\ell}^{j_\ell}f_{\ell}|_{k_\ell}\sum_{i |k_i>r_i}
 |\pd_ {\la_i}^{j_i}f_{i}|_{k_i+a'}\prod_{\ell|\ell\neq i,k_\ell>r_\ell}|\pd_ {\la_\ell}^{j_\ell}f_{\ell}|_{a_\ell}
\end{align}
with $a'=\sum_{\ell|k_\ell>a_\ell} (k_\ell-r_\ell)\leq a$.
Again we estimate each term via pointwise derivative. We can write
$$
\Bigl\{\prod_{\ell|k_\ell\leq r_\ell}
|\pd_ {\la_\ell}^{j_\ell}f_{\ell}|_{k_\ell}\Bigr\}|\pd_ {\la_i}^{j_i}f_{i}|_{[k_i+a']}\prod_{\ell|\ell\neq i,k_\ell>a_\ell}|\pd_ {\la_\ell}^{j_\ell}f_{\ell}|_{[a_\ell]}=
|\pd_{x_i}^{k'}\pd_ {\la_i}^{j_i}f_{i}|\prod_{\ell\neq i}|\pd_{x_\ell}^{k_\ell}
\pd_ {\la_\ell}^{j_\ell}f_{\ell}|.
$$
While applying \re{estprod} for the domain $I$, we  see that the last term is bounded by the right-hand side of \re{ff-b}.
For H\"older ratios on the right-hand side of \re{pleb2}, let $\beta_i=\{a_i+r'\}$ and $k_i=[a_i+r']$,  or $\beta_i=0$
and $k_i\leq [a_i+r']$; and  for $\ell\neq i$, let $\beta_\ell=\{a_\ell\}$ or $\beta_\ell=0$ and let $k_\ell\leq[a_\ell]$.
   We estimate the product of H\"older ratios
\al\label{bi0}\nonumber %
&\prod_{i|\beta_i>0} \f{|\pd_{y_i}^{k_i}\pd_{\la_i}^{j_i}f_{i}- \pd_{x_i}^{k_i}\pd_{\la_i}^{j_i}f_{i}|}{|{y_i}-{x_i}|^{\beta_i}}
\times\prod_{i|\beta_i=0}|\pd_{x_i}^{k_i}\pd_ {\la_i}^{j_i}f_{i}|
\ple\sum |\pd_{\la_i}^{j_i} f_{i}|_{r}\prod_{\ell\neq i} |\pd_{\la_\ell}^{j_\ell} f_{\ell}|_{0}.
\end{align}
We need to estimate each term in the last sum by considering pointwise derivatives and H\"older ratios.
Applying \re{estprod} to $D_i$, we can bound  $|\pd_{\la_i}^{j_i} f_{i}|_{[r]}\prod_{\ell\neq i} |\pd_{\la_\ell}^{j_\ell} f_{\ell}|_{0}$ by the right-hand side of \re{ff-b}. Let $\beta=\{r\}$ and $k=[r]$.  Applying \re{estprod} to the interval  $I$, we can bound
$$
\f{|\pd_{y_i}^{k}\pd_{\la_i}^{j_i}f_{i}- \pd_{x_i}^{k}\pd_{\la_i}^{j_i}f_{i}|}{|{y_i}-{x_i}|^{\beta}}
\times\prod_{\ell\neq i}| \pd_ {\la_\ell}^{j_\ell}f_{\ell}|_0
$$
  by the right-hand side of \re{ff-b}. We have verified \re{ff-b} and hence \re{ff-c}.
\end{proof}
By the product rule and \rl{ff-}, we obtain the following.
\le{ff} Let $D$  be a domain in $\rr^n$ with the cone property.
Assume that $1\leq m'<m<\infty$. With all norms on $D$, we have
\begin{align} \nonumber %
&|\{f^\la _1\cdots f^\la _m\}|_{r,s}+
Q^*_{r,s}\bigl(\{\prod_{i\leq m'}f^\la _i\},\{\prod_{i>m'}f^\la _i\}\bigr)\\ 
 &\qquad\leq C^m_{r,s}\left\{\sum_i|f_i|_{r,s}\prod_{k\neq i}|f_k|_{0,0}
+  \sum_{i\neq j} Q^*_{r,s}(f_i,f_j)\prod_{k\neq i,j}|f_k|_{D;0,0}\right\}.
\nonumber 
\end{align}
\ele

Let us use Lemmas~\ref{ff-} and \ref{ff} to verify the following.
\pr{phifi} Let $r,s\in[0,\infty)$. Let $\phi_i$ be $\cL C^{r+s+1}$ functions in $[a_i,b_i]$. Let $D$ be as in \rla{ff}. Let $f_i\in C_*^{r,s}(\ov D)$ and $f_i^\la (D)\subset[a_i,b_i]$. Suppose that $0\in[a_i,b_i]$ and $\phi_i(0)=0$. Let $1\leq m'<m$.
With all norms on $D$, we have
\begin{align} 
& |\phi_1(f_1)|_{r,s}\leq C_{r,s}(|f_1|_{r,s}+|f_1|_{r,0}\odot|f_1|_{0,s}),
 \\
\label{phi1f1}
&|\phi_1(f_1)\cdots \phi_m(f_m)|_{r,s}+
Q^*_{r,s}(\prod_{i\leq m'}\phi_i(f_i),\prod_{i>m'}\phi_i(f_i))\\
& \qquad\leq C^m_{r,s}\left\{ \sum_i|f_i|_{r,s}\prod_{k\neq i}|f_k|_{0,0}
+\sum_{i\leq j} Q^*_{r,s}(f_i,f_j)\prod_{k\neq i,j}|f_k|_{0,0}\right\}.\nonumber
\end{align}
where   $C_{r,s}$ also depends on $|\phi_i|_{[a_i,b_i];[r]+[s]+1}$.
\epr
\begin{proof}Since $\phi_i\in\cL C^1$ and $\phi_i(0)=0$, we get $|\phi_i(f_i)|_{D;\all,\beta}\ple|f_i|_{\all,\beta}$ for $\all,\beta\in[0,1]$. Applying the chain rule and \rl{ff-}, we get
\aln
|\phi_i(f_i)|_{r,s}&\ple\sum_{r_1+\dots+ r_m=r,s_1+\dots+s_m=b}\prod|f_i|_{r_\ell,s_\ell}
\ple |f_i|_{r,s}+|f_i|_{r,0}\odot|f_i|_{0,s}
\end{align*}
with all $r_i,s_j$ are integers, except for possible one. We can verify \re{phi1f1} similarly by using \rl{ff-}.
\end{proof}

When applying the chain rule, we need to count derivatives efficiently.
\begin{defn}\label{defbiord}  Let $k\geq1$ and $k\geq j\geq0$.
 Let $F^\la $ be a family of mappings from $D$ to $D'$.  Let $\{u^\la \}$ be a family of functions on $D$
 and   let $\{f_i^\la \}, \dots, \{f_m^\la \}$  be families of functions on $D'$.  Define $\cL P_{k,j}(\{u^\la \circ F^\la \}; \{f^\la \})$ to
  be the linear space spanned by functions of the form
\eq{mfactor}
\{\pd^{K_0}\pd_\la^{j_0}u^\la \}\circ F^\la ,\quad \{\pd^{K_0}\pd_\la^{j_0}u^\la \}\circ F^\la \prod_{1\leq\ell\leq l }\pd^{K_\ell}\pd_\la^{j_\ell}f_{n_\ell}^\la , \quad 1\leq l<\infty
\eeq
with
\ga\label{effect}
|K_\ell|+j_\ell\geq1, \quad\ell\geq0;\qquad \sum_{\ell\geq0} j_\ell\leq j, \\
\label{vtx+}|K_0|+j_0+\sum_{\ell\geq1}(|K_\ell|+j_\ell-1)\leq k, \quad \ i.e. \quad \ \sum_{\ell\geq0}(|K_\ell|+j_\ell-1)<k.
\end{gather}
Analogously, define $\cL P_{k,j}(\{f^\la \})$ to be the linear space spanned by
\eq{mfactor+}
 \prod_{1\leq\ell\leq l }\pd^{K_\ell}\pd_\la^{j_\ell}f_{n_\ell}^\la , \quad1\leq  l<\infty
\eeq
with $|K_\ell|+j_\ell\geq1$,   $\sum j_\ell\leq j$, $\sum(|K_\ell|+j_\ell-1)<k$.
\end{defn}
By counting efficiently, we count one less for the order of derivative $\pd^{K_\ell}\pd_\la^{j_\ell}f_{n_\ell}^\la $. Also the $l$  in \re{mfactor} and \re{mfactor+} will have an upper bound depending on $k,j$.

 It is easy to see that if $\{a^\la \}\in\cL P_{k_0,j_0}(\{f^\la \})$ and $\{v^\la \}\in\cL P_{k_1,j_1}(\{u^\la \circ F^\la \};\{f^\la \})$,
 with $k_0\geq1$, $k_1\geq1$,  then
$
\{a^\la v^\la \}\in\cL P_{k_0+k_1-1,j_0+j_1}(\{u^\la \circ F^\la \};\{f^\la \})$ by \re{vtx+}.
We express it as
\eq{addp}
\cL P_{k_0,j_0}(\{f^\la \})\times \cL P_{k_1,j_1}(\{u^\la \circ F^\la \};\{f^\la \})\subset\cL P_{k_0+k_1-1,j_0+j_1}(\{u^\la \circ F^\la \};\{f^\la \})
\eeq
for $k_0,k_1\geq1$.
Also, if $F^\la =I+f^\la $, then
$$
\cL P_{k,j}(\{u^\la \circ F^\la \};\{F^\la \})=\cL P_{k,j}(\{u^\la \circ F^\la \};\{f^\la \}).
$$

\le{chainord}If $1\leq |K|+j\leq k$ then
 $
\pd^K\pd_\la^j\{u^\la \circ F^\la \}\in \cL P_{k,j}(\{u^\la \circ F^\la \};\{F^\la \})$.
\ele
\begin{proof}Let $y=F^\la (x)$. Let $\pd^i$ be a derivative of order $i$ in $x$. We use the chain rule.
 Applying $\pd$ or $\pd_\la$ one by one, we can verify that
$\pd^{k-j}\pd_\la^j(u^\la \circ G^\la )$ is  a linear combination of functions
\ga\label{vtx}
 (\pd_ \la^{j_0} \pd^{K_0} u^ \la)\circ G^\la  \prod_{1\leq i\leq m',j_i>0}
\pd_ \la^{j_i}\pd^{K_i} G_{n_i}^\la \prod_{m'<i\leq |K_0|,|K_i|>0}
 \pd^{K_i} G_{n_i}^\la .
\end{gather}
Here  $\sum_{\ell\geq0} j_\ell=j$, and $j+\sum_{\ell\geq1}|K_\ell|=k$.  When $|K_0|=0$, it is clear that \re{effect}-\re{vtx+} hold.  When $|K_0|>0$, we have $\sum_{\ell\geq0}(j_\ell+|K_\ell|-1)= k-|K_0|<k$.
Thus \re{vtx} is in $\cL P_{k,j}(\{u^\la \circ F^\la \};\{F^\la \})$.\end{proof}
Define
\aln
\tilde Q_{r,s}(\pd f,  \pd g)&:=Q_{r-1,s}(\pd f,\pd g),\\
 \tilde Q_{r,s}(\pd f,  \tilde\pd g)&:=\tilde Q_{r,s}(\pd f,\pd g)
+Q_{r-1,s-1}(\pd f,\pd_\la g):=\tilde Q_{r,s}(\tilde\pd g,\pd f), \\
\tilde Q_{r,s}(\tilde\pd f,\tilde\pd g)&=\tilde Q_{r,s}(\pd f,\tilde\pd g) +Q_{r-1,s-1}(\pd_\la f,\pd g) +Q_{r-1,s-2}(\pd_\la f,\pd_\la g),
\end{align*}
where $Q_{r-1,s}=0$ for $r<s+1$, $Q_{r-1,s-1}=0$ for $s<1$, and $Q_{r-1,s-2}=0$ for $s<2$.

\le{ff+}
Let $r,s$ satisfy \rea{rsrs}.   Let $D, D'$  be   bounded domains in $\rr^m,\rr^n$ respectively, which have \ps.
 Let $ G^\la  \colon D\to D'$ be of class $\cL C^{r,s}(\ov D)$.
 \bppp
 \item Assume that $|G|_{D;1,0}<2$.
 Then
\al \label{3p3-i}
|\{u^ \la\circ G^ \la\}|_{D;\all,0}&\leq C|u|_{D';\all,0},
\quad \ 0\leq\all<1,\\
\label{3p3-+i}|\{u^ \la\circ G^ \la\}|_{D;r,0}&\leq  C_{r}
\Bigl(
 |u|_{D';r,0} +|\pd u|_{D';0,0} |\pd G|_{D;r-1,0}\Bigr), \quad r\geq1.
\end{align}
\item
 Assume that $|G|_{D;1,0}\leq 2$ and $s\geq1$.
 Then
\al\label{3p3--}\nonumber %
\|\{u^ \la\circ G^ \la \}\|_{D;r,s}& \leq  C_{r}
 \left\{ \|u\|_{D';r,s}+\tilde Q_{D',D;r,s}(\pd u,  \tilde\pd G) \right.\\
 &\quad+\left.|u|_{D';1,0}\bigl(|G|_{D;r,s}
+ \tilde Q_{D',D';r,s}(\tilde\pd G,  \tilde\pd G)\bigr)\right\}.\nonumber
\end{align}
\eppp
\ele
\begin{proof} (i). Inequality \re{3p3-i} is immediate and \re{3p3-+i} is in \ci{Ho76} and \ci{GW11}.

 (ii).    Let $\all=\{r\}$ and $\beta=\{s\}$.
Let $r=k+\all$  with $k\geq1$.
Let $y=G^ \la(x)$. Let $\pd^i u^\la (y)$ be a partial derivatives in $y$ of order $i$.
We know that $\pd^{k-j}\pd_\la^j(u^\la \circ G^\la )\in \cL P_{k,j}(\{u^\la \circ G^\la \};\{G^\la \})$, i.e. it is  a linear combination of functions $v^\la (x)$ of the form \re{vtx}. We can express $v^\la $ as
\ga\nonumber %
 v^\la  =(\pd_ \la^{j_0} \pd^{m} u^ \la)\circ G^\la  \prod_{1\leq i\leq m'}
\pd_ \la^{j_i}\pd^{k_i}\pd_\la G_{n_i}^\la \prod_{m'<i\leq m}
 \pd^{k_i}\pd  G_{n_i}^\la ,
\end{gather}
where  $ j=j_0+\dots+j_{m'}+m'$, and    by \re{vtx+}
we have $
j+m+\sum k_i\leq k.
$

When $m=0$, it is immediate that  by
  $|G|_{1}<2$, we obtain $|v|_{\all,\beta}\ple |u|_{\all,j+\beta}$. Suppose that $m\geq1$. Computing the H\"older norms, we get
\aln
 |v|_{\all,\beta}&\ple|\pd u|_{m-1+\all,j_0+\beta}\prod|\pd_\la G_{n_i}|_{k_i,j_i}\prod|\pd G_{n_i}|_{k_i,0}\\
 &\quad+\sum_\ell |\pd u|_{j_0+m-1,j_0}|\pd_\la G_{n_\ell}|_{m_\ell+\all,j_\ell+\beta}\prod_{i\neq\ell, i,\ell\leq m'}|\pd_\la G_{n_i}|_{k_i,j_i}\prod_{i>m'}|\pd G_{n_i}|_{k_i,0}\\
 &\quad+\sum_\ell |\pd u|_{j_0+m-1,j_0}|\pd G_{n_\ell}|_{m_\ell+\all,j_\ell+\beta}\prod_{i\leq m'}|\pd_\la G_{n_i}|_{k_i,j_i}\prod_{i\neq\ell, i,\ell> m'}|\pd G_{n_i}|_{k_i,0}.
 \end{align*}
We use \rl{ff-} and obtain
\aln
|v|_{\all,\beta}&\ple
  |\pd u|_{r-1,s} + |u|_{1,0}|G|_{r,s}+    Q_{r-1,s}(\pd u,\pd G)+   Q_{r-1,s-1}(\pd u,\pd_\la G) \\
 &  +
 |u|_{1,0}\bigl\{  Q_{r-1,s}(  \pd G, \pd G)+    Q_{r-1,s-1}(  \pd G, \pd_\la G)+   Q_{r-1,s-2}(  \pd_\la G, \pd_\la G)\bigr\}.\qedhere\nonumber
\end{align*}
\end{proof}

Let $D_\rho=B_{\rho}^{2n}\times B_\rho^M\times B_\rho^L$ with $N=2n+K+L$, and set $|\cdot|_{\rho;r,s}=|\cdot|_{D_\rho;r,s}$ and $\tilde Q_{\rho,\tilde\rho;r,s}=\tilde Q_{D_\rho,D_{\tilde\rho};r,s}$.
\le{fg} Let $1\leq r<\infty$. Let    $0<\rho<\infty$, $ 0<\theta<1/2$ and $\rho_i=(1-\theta)^i\rho$.
Let $F^ \la=I+f^ \la$ be mappings from $D_\rho$ into $\rr^N$.
Assume that $F\in\cL C^{1,0}(
 \ov{D_{\rho}})$ and
\eq{ft0}\nonumber %
f^\la (0)=0, \quad |f|_{\rho;1,0}\leq \theta/{C_N}.
\eeq
\bppp \item Then $F^\la $ are injective in $D_\rho$.  There exist unique $G^ \la=I+g^ \la$ satisfying
\ga\nonumber
G^ \la\colon D_{\rho_1}
\to D_{\rho}, \quad F^ \la\circ G^ \la=I  \quad \text{on $D_{\rho_1}$}.
\end{gather} Furthermore,
$F^ \la\colon D_{\rho}\to D_{(1-\theta)^{-1}\rho}$ and $
G^ \la\circ F^ \la=I$  in $D_{ \rho_2}$.
\item Suppose that $1/4<\rho<2$.  Assume further that $F\in\cL C^{r,0}(
 {D_{\rho}})$. Then $\{g^\la \}\in \cL C^{r,0}(
 {D_{\rho_1}})$ and
\ga
\label{3p6}
|g|_{\rho_1;r,0}\leq C_r |f|_{\rho; r,0},
\\
|u\circ G|_{\rho_1;r,0}\leq C_r  (|u|_{\rho; r,0}+ | u|_{\rho;1,0}|f|_{\rho; r,0}).\label{3p6+} \nonumber %
\end{gather}
\item Suppose that $1/4<\rho<2$.
Assume further that  $r,s$ satisfy \rea{rsrs} with $  s\geq1$,  $|f|_{\rho;1 }<C$, and $f\in\cL C^{r,s}(
 {D_\rho})$.   Then for $u\in\cL C^{r,s}(D_\rho)$ 
\ga\label{3p6a}
\|g\|_{\rho_1;r,s}\leq C_r  (\|f\|_{\rho; r,s}+ \tilde Q_{\rho,\rho;r,s}(\pd f, \tilde \pd f)+|f|_{ \rho;1,0}\tilde Q_{\rho,\rho;r,s}(\tilde \pd f,\tilde\pd f)),\\
\|u^\la \circ G^\la \|_{\rho_1;r,s}\leq C_r  (\|u\|_{\rho; r,s}+ \tilde Q_{\rho,\rho_1;r,s}(\pd u, \tilde \pd f)+|u|_{\rho;1,0}\tilde Q_{\rho,\rho;r,s}(\tilde \pd f,\tilde\pd f)).\label{3p6a+}
\end{gather}
\eppp
\ele
\begin{proof}We may  reduce the proof  to $\rho=1$ by using dilations $x\to\rho^{-1}F^\la (\rho x)$ and $x\to\rho^{-1} G^\la (\rho x)$.
The proof of (i) can be obtained easily by applying the
  contraction mapping theorem to
$$
g^\la (x)=-f^\la (x+g^\la (x)).
$$

(ii) is a special case of (iii). Thus, let us verify (iii).
Assume that $r>1$.
 Differentiating the above identity,  we separate terms of the highest order derivatives of $g^\la $ from the rest to get identities
\ga\label{kjt}\nonumber %
( \pd^K\pd_ \la^j) g_l^ \la+\sum_{m}(\pd_{y^{m}}f_l^ \la)\circ G^ \la(\pd^{K}\pd_ \la^j)g_{m}^\la =E_{lKj}^\la.
\end{gather}
Here $E_{lKj}^\la $ are the linear combinations of the functions in $\cL P_{|K|+j,j}(\{f^\la \circ G^\la \},\{g^\la \})$ of the form
\ga
( \pd^{K_0}\pd_ \la^{j_0}) f_l^ \la, \quad
v^\la:=
 ( \pd^{M_0}\pd_ \la^{j_0 }f_l^ \la)\circ G^ \la\prod_{1\leq\ell\leq m} \pd^{M_\ell}\pd_ \la^{j_\ell}
g_{n_\ell}^ \la\label{skell-}
\end{gather}
with   $m\leq |M_0|$, $|M_\ell|+j_\ell\geq1$, and $
\sum j_\ell\leq j$. Furthermore,
\ga\label{skell}
 |M_\ell|+j_\ell<|K|+j, \quad \ell>0; \quad  \sum_{\ell\geq0} (|M_\ell|+j_\ell-1)<|K|+j.
\end{gather}
We now verify that  $\pd^{K}\pd_ \la^j g^ \la$  is a finite sum of
\eq{hKj}
h_{Kj}^\la:=\{A(\pd f)\pd^{K_1}\pd_ \la^{j_1}f_{n_1}\cdots\pd^{K_m}\pd_ \la^{j_m}f_{n_m}\}\circ G^ \la,
\eeq
where $A(\pd^1f)$ is a polynomial in $(\det(\pd f^\la ))^{-1}$ and $ \pd f^\la $. Moreover, $j_i\geq1$ or $|K_i|+j_i\geq2$, and
\eq{hKj+}
\sum_{\ell\geq0}( |K_\ell|+j_\ell-1) <|K|+j, \quad \sum j_\ell\leq j.
\eeq
The assertion is trivial when $|K|+j=1$. Assume that it holds for $|K|+j<N$. By \re{addp}, we can see that the $\{v^\la\}$ in \re{skell-} has the form \re{hKj}.
 This shows that $v^\la $ is of the form \re{hKj} and \re{hKj+}.
The claim has been verified.

The estimation of $|\{h_{Kj}^\la \}|_{\{r\},\{s\}}$  is the same as in the proof of \rl{ff+}. Indeed, when $|M_0|=0$, $|\{h_{Kj}^\la \}|_{\{r\},\{s\}}\ple\|f\|_{r,s}$
for $|K|+j\leq [r]$ and $j\leq[s]$. When $|M_0|>0$,
$$
|\{h_{Kj}^\la \}|_{\{r\},\{s\}}\ple\tilde Q_{r,s}(\pd f,\tilde\pd f)+|f|_{1,0}\tilde Q_{r,s}(\tilde \pd f,\tilde\pd f).
$$
 This verifies \re{3p6a}.

To verify \re{3p6a+}, as in \re{skell-}-\re{skell} we note that $\pd^k\pd_\la^j(u^\la \circ G^\la (x))$ is a linear combination of functions in $\cL P_{k+j,j}(\{u^\la \circ G^\la \};\{G^\la \})$ of the form
\ga \label{utvG}
\tilde v^\la = ( \pd^{M_0}\pd_ \la^{j_0 }u^ \la)\circ G^ \la\prod_{1\leq\ell\leq m} \pd^{M_\ell}\pd_ \la^{j_\ell}
g_{n_\ell}^ \la
\end{gather}
with   $m\leq |M_0|$, $|M_\ell|+j_\ell\geq1$,  $
\sum j_\ell\leq j$, and
\ga \label{vtvG+}
 |M_\ell|+j_\ell<k+j, \quad \ell>0; \quad  \sum_{\ell\geq0} (|M_\ell|+j_\ell-1)\leq k+j-1.
\end{gather}
Expressing  $\pd^{K}\pd_ \la^j g^ \la$ via linear combinations in \re{hKj} satisfying \re{hKj+} and applying them to  $\pd^{M_\ell}\pd_\la^{j_\ell}g^\la $ in \re{utvG}, we conclude that $\tilde v^\la $ and hence $\pd^K\pd_\la^jg^\la $ is a linear combination of
$$
\tilde h_{Kj}^\la :=
\left\{A(\pd f)\pd^{M_0}\pd_ \la^{j_0 }u^ \la \prod_{1\leq\ell\leq |M_0|}\prod_{i\leq n_\ell}\pd_0^{M_\ell}\pd_ \la^{j_\ell}f_{n_\ell^i} \right\}\circ G^\la
$$
where $M_\ell,j_\ell$ still satisfy \re{vtvG+}. When $|M_0|=0$, we actually have
$\tilde h_{Kj}^\la =
( \pd_ \la^{j }u^ \la) \circ G^\la $. It is straightforward that $|\tilde h_{Kj}|_{\all,\beta}\ple|u|_{r,s}$. When $|M_0|>0$, the estimate for $|\tilde h_{Kj}|_{\all,\beta}$ is obtained analogous to that of $h_{Kj}$.
\end{proof}

We will also need to deal with sequences of compositions.
\pr{seqint} Let $D_m$ be a sequence of domains in $\rr^d$ satisfying the cone property of which the constants $C_*(D_m)>1$ are bounded.
Assume that $D_m\subset D_\infty$ and $D_\infty$ also has the cone property. Let $F_m^\la =I+f_m^\la \colon D_m\to D_{m+1}$ for $m=1,\dots$.
Suppose that
$$
f_i^\la (0)=0,\quad
|f_i|_{D_i;1,0}\leq 1.
$$
Assume that $r\geq s\geq1$ and $\{r\}\geq\{s\}$.  Then
\al 
\nonumber %
\|\{u^\la \circ F_\ell^\la \circ&\cdots\circ F_1^\la \}\|_{D_1;r,s}\leq C_r^{\ell}  \left\{\|u\|_{D_\infty;r,s}+\sum \tilde Q_{r,s}(\pd u,\tilde\pd f_m)\right.\\
&\left.\quad+|u|_{D_\infty;1,0}  \left(\sum \|f_m\|_{D_m;r,s}+\sum_{i\leq j} \tilde Q_{r,s}(\tilde\pd f_i,\tilde\pd f_j)\right)
\right\}.\nonumber
\end{align}
\epr
\begin{proof} We first derive the factor $C_r^\ell$ in the inequality. Let $G_i^\la = F_i^\la \circ\cdots\circ F_1^\la $.
Let $h_{\ell,k}^\la $ be a $k$-th order derivative of $u^\la \circ G_\ell^\la $.
 We express $h_{\ell,k}^\la $ as a sum of terms of the form
 \ga
 h_2^\la   :=\tilde h_2^\la \{ (\pd_ \la^{j_0} \pd^{k_0} u^ \la)\circ G_\ell^\la \}\prod_{1\leq i\leq T'}
\pd_ \la^{j_i}\pd^{k_i}\pd_\la F_{n_i}^\la \prod_{T'<i\leq T}
 \pd^{k_i}\pd  F_{n_i}^\la .
 \label{3p5+s}
\end{gather}
Here $\sum k_i+\sum j_i\leq k-1$, $T'+\sum j_i\leq j$, and $\tilde h_2^\la$ is $1$ or  a product of first-order derivatives of $F_i^\la$ in $x$, which may be repeated. Let
 $T_{\ell,k}$ be the maximum number of derivative functions of $u^\la,F^\la_{\ell},\ldots, F^\la_1$ that appear
 in  \re{3p5+s}. We have $T_{\ell,0}=0$ and $T_{\ell,1}=\ell+1$. By the chain rule, $T_{\ell,k}\leq T_{\ell,k-1}+\ell\leq k\ell +1$.
Let $N_{\ell,k}$ be the maximum number of terms \re{3p5+s} that are needed to express $h_{\ell,k}^\la$ as a sum of the monomials  \re{3p5+s}.
 Note that $u^\la (x),F_i^\la (x)$ are functions in $d+1$ variables.
 By the chain rule $N_{\ell,1}\leq(d+1)^\ell$. By the product and chain rules,
$$
N_{\ell,k}\leq N_{\ell,k-1}T_{\ell,k-1}(d+1)^\ell \leq (k\ell+1)(d+1)^{\ell} N_{\ell,k-1}<C_k^\ell.
$$

For a fixed $i$, the first-order derivatives of $F_i^\la$ in $x$ cannot repeated more than $k$ times in $h_2^\la$.
Thus, we have
$$
|\tilde h_2^\la(x)|\leq(1+|f_1|_{1,0})^k\cdots (1+|f_\ell|_{1,0})^k<2^{\ell k}.
$$
Also the $T$ in \re{3p5+s} is less than $k$.
 By \rl{ff}, we have
$$
|h_2|_{\all,\beta} \leq  N_{k+1}2^{\ell k} \left\{\|u\|_{r,s}+\sum Q_{r,s}(u,f_i)
+|u|_{1,0} \bigl(\sum |f_i|_{r,s}+\sum Q_{r,s}(f_i,f_j)\bigr)\right\}.
$$
Here $N_{k+1}$ arises from the number of terms when computing the H\"older norms after $k$ derivatives.
\end{proof}

Next, we consolidate the expressions such as  $\tilde Q_{r,s}(\pd f,\tilde\pd g)$ in \rl{fg} and \rp{seqint} by a
 simpler expression    $Q_{r,s}(f,g)$, defined by \re{qrsfg},  and the new expression
\eq{qprs}\nonumber %
Q'_{r,s}(f,g):= Q_{r,s}(f,g)+\sum_{j=1}^{[s]}Q^*_{r-j+1,j+\{s\}}(f,g), \quad s\geq1.
\eeq
\le{qptpt} 
Let $f\in \cL C^{r,s}(D),g\in \cL C^{r,s}(D')$ with $D,D'$ having the cone property.
 Then
\al\label{qr-1s}
Q_{r-1,s}(\pd f,g)&\leq C_r\hat Q_{r,s}(f,g),\quad
Q_{r-1,s-1}(\pd_\la f,g)\leq C_r\hat Q_{r,s}(f,g),\\
\label{r-1dd}
Q_{r-1,s}(\pd f,\pd g)&\leq C_r\left\{ |f|_{1,\{s\}}\odot\|g\|_{r,0}+|f|_{1,0}\|g\|_{r,s}\right. \\
&\nonumber \quad 
\left.+|g|_{1,\{s\}}\odot|f|_{r,0}+|g|_{1,0}|f|_{r,s}+Q'_{r,s}(f,g)\right\},\\
\label{qr-1ddt}
Q_{r-1,s-1}(\pd f,\pd_\la g)&\leq C_r\left\{|f|_{1,\{s\}}\odot|g|_{r-1,1}+Q'_{r,s}(f,g)\right\},
\ 1\leq s<2\leq r,\\
\label{qr-1ddt+}
Q_{r-1,s-1}(\pd f,\pd_\la g)&\leq C_r\left\{ |f|_{1,\{s\}}\odot|g|_{r-1,1}+|f|_{0,1}\|g\|_{r,s-1}\right.\\
&\quad +\left.|g|_{0,1}\|f\|_{r,s-1}+Q'_{r,s}(f,g)\right\},\ s\geq1,\nonumber
\\
\label{qr-1dtdt}
Q_{r-1,s-2}(\pd_\la f,\pd_\la g)&\leq  C_r\left\{ |f|_{0,1}\|g\|_{r,s-1}+|g|_{0,1}\|f\|_{r,s-1}\right.\\ &\quad\left. +
|f|_{\{r\},s}|g|_{0,0}+|g|_{\{r\},s}|f|_{0,0}+Q_{r,s}'(f,g)\right\}.\nonumber
\end{align}
Here the norms of $f$ {\rm(}resp. $g${\rm)} and its derivatives are in $D$ {\rm(}resp. $D'${\rm)}.
\ele
\begin{proof} We verify the first inequality in \re{qr-1s} by $|\pd f|_{r-1-j,0}\odot|g|_{0,j+\{s\}}\leq Q_{r,s}(f,g)$ and
$$
|\pd f|_{0,j+\{s\}}\odot|g|_{r-1-j,0}\ple |f|_{0,j+\{s\}}\odot|g|_{r-j,0}+|f|_{r-j,j+\{s\}}|g|_{0,0}\ple \hat Q_{r,s}(f,g).
$$
Here the first inequality follows from \re{ff-b}.
The second inequality in \re{qr-1s}  follows from
$
|g|_{r-1-j,0}|\pd_\la f|_{0,j+\{s\}}\leq Q_{r,s}(f,g)$ and
\gan|g|_{0,j+\{s\}}|\pd_\la f|_{r-1-j,0}\ple |g|_{0,j+1+\{s\}}|f|_{r-1-j,0}+|g|_{0,0}|f|_{r-1-j,j+1+\{s\}}\ple \hat Q_{r,s}(f,g).
\end{gather*}

We verify \re{r-1dd},  by $|\pd f|_{r-1,0}\odot|\pd g|_{0,\{s\}}\leq |f|_{r,0}\odot|g|_{1,\{s\}}$ and for $j>0$
$$
|\pd f|_{r-1-j,0}\odot|\pd g|_{0,j+\{s\}}\leq |\pd f|_{r-j,0}\odot|g|_{0,j+\{s\}}+|\pd f|_{0,0}|g|_{r-j,j+\{s\}}.
$$

Inequality\re{qr-1ddt} follows from $|\pd f|_{r-j-1,0}\odot|\pd_tg|_{0,j+\{s\}}\leq Q'_{r,s}(f,g)$ and $|\pd f|_{0,\{s\}}\odot|\pd_tg|_{r-1,0}\leq |f|_{1,\{s\}}\odot|g|_{r-1,1}$. For \re{qr-1ddt+}, we need to verify extra terms. Suppose  $0<j<[s]$. Then
\aln
|\pd f|_{0,j+\{s\}}\odot|\pd_tg|_{r-j-1,0}&\ple |f|_{0,j+\{s\}}\odot|\pd_tg|_{r-j,0}+|f|_{r-j,j+\{s\}}|g|_{0,1}
\end{align*}
and $
|f|_{0,j+\{s\}}\odot|\pd_tg|_{r-j,0}\ple |f|_{0,j+1+\{s\}}\odot|g|_{r-j,0}+|f|_{0,1}|g|_{r-j,j+\{s\}}.
$

We verify \re{qr-1dtdt} by $|\pd_tf|_{\{r\},0}\odot|\pd_tg|_{0,j+\{s\}}\ple|f|_{\{r\},0}\odot|g|_{0,s}+|f|_{\{r\},s}|g|_{0,0}$ and
\begin{equation*}
|\pd\pd_tf|_{r-j-2,0}\odot|\pd_tg|_{0,j+\{s\}}\ple|\pd f|_{r-j-2,0}\odot|g|_{0,j+2+\{s\}}+\| f\|_{r,s-1}|g|_{0,1}.\qedhere
\end{equation*}
\end{proof}

We now state simple version of the above inequalities, which suffices   applications in this paper. The next two propositions are immediate consequences of \rl{fg}, \rp{seqint}, \rl{qptpt}, and the following  crude estimate
\ga\label{simpleTQ}\nonumber %
\tilde Q_{r,s}(\tilde\pd f,\tilde\pd g)\leq C(|f|_{1,s_*}\|g\|_{r,s}+|g|_{1,s_*}\|f\|_{r,s}\\
\hspace{26ex}+\|f\|_{s+1,s}\odot\|g\|_{r,s_*}+\|g\|_{s+1,s}\odot\|f\|_{r,s_*}).\nonumber
\end{gather}
\pr{pgrsf}
Let $s_*=0$ for $s=0$ and $s_*=1$ for $s\geq1$. Let $F^\la =I+f^\la $ and $(F^\la )^{-1}=I+g^\la $ be as in \rla{fg}.
Assume that
$$
f^\la (0)=0, \quad |f|_{1,0}\leq\theta/C_N, \quad |f|_{1,s_*}<1.
$$
Let $\rho_1=(1-\theta)\rho$ with $1/4<\rho<2$ and $0<\theta<1/2$. Then
\gan 
\|g\|_{\rho_1;r,s}\leq C_r\left\{\|f\|_{\rho;r,s}+ \|f\|_{s+1,s}\odot\|f\|_{r,s_*}\right\},\\
\label{uF-1}\|\{u^\la \circ (F^\la )^{-1}\}\|_{\rho_1;r,s}\leq C_r\left\{\|u\|_{\rho;r,s}+|u|_{\rho;1,s_*} \|f\|_{\rho;r,s}
\right.\\
  \hspace{8ex}\left.+\|u\|_{s+1,s}\odot\|f\|_{r,s_*}+\|u\|_{r,s_*}\odot\|f\|_{s+1,s}+|u|_{\rho;1,0}\|f\|_{s+1,s}\odot\|f\|_{r,s_*} \right\}.\nonumber
\end{gather*}
\epr
\pr{uFF} Let $r,s,s_*$ be as in \rpa{pgrsf}. Let $F_i^\la =I+f_i^\la$ map $ D_i$ into $ D_{i+1}$, where $D_\ell$ are as in \rpa{seqint}. Assume that
$f_i^\la (0)=0$ and $
  |f_i|_{1,s_*} \leq1.
$
Then
 \aln
&\|\{u^\la \circ F_m^\la \circ\cdots\circ F_1^\la \}\|_{D_1;r,s}\leq C_r^{m}  \left\{\|u\|_{r,s}+ \|u\|_{1,0} \sum_{i\leq j} \|f_i\|_{s+1,s}\odot\|f_j\|_{r,s_*} \right.\\
&\qquad +\left.
\sum_i|u|_{1,s_*} \|f_i\|_{r,s}
+\|u\|_{s+1,s}\odot\|f_i\|_{r,s_*}+\|u\|_{r,s_*}\odot\|f_i\|_{s+1,s}\right\},
\end{align*}
where   $\|f_i\|_{a,b}=\|f_i\|_{D_i;a,b}$ and $\|u\|_{a,b}=\|u\|_{D_{m+1};a,b}$.
\epr


\begin{bibdiv}
\begin{biblist}



\bib{Be57}{book}{
author={Bers, L.},
title={Riemann Surfaces},
series={(mimeographed lecture notes) (1957-1958)},
publisher={New York University},
 }

\bib{BG14}{article}{
   author={Bertrand, F.},
   author={Gong, X.},
   title={Dirichlet and Neumann problems for planar domains with parameter},
   journal={Trans. Amer. Math. Soc.},
   volume={366},
   date={2014},
   number={1},
   pages={159--217},
}

\bib{BGR14
}{article}{
   author={Bertrand, F.},
   author={Gong, X.},
   author={Rosay, J.-P.},
   title={Common boundary values of holomorphic functions for two-sided
   complex structures},
   journal={Michigan Math. J.},
   volume={63},
   date={2014},
   number={2},
   pages={293--332},
}

\bib{CS01}{book}{
   author={Chen, S.-C.},
   author={Shaw, M.-C.},
   title={Partial differential equations in several complex variables},
   series={AMS/IP Studies in Advanced Mathematics},
   volume={19},
   publisher={American Mathematical Society, Providence, RI; International
   Press, Boston, MA},
   date={2001},
}

\bib{Ch55}
{article}{
   author={Chern, S.-s.},
   title={An elementary proof of the existence of isothermal parameters on a
   surface},
   journal={Proc. Amer. Math. Soc.},
   volume={6},
   date={1955},
   pages={771--782},
}
	



\bib{Fr77}{article}{
   author={Frobenius, G.},
   title={Ueber das Pfaffsche Problem},
   journal={J. Reine Angew. Math.},
   volume={82},
   date={1877},
   pages={230--315},
}

\bib{GT01}{book}{
   author={Gilbarg, D.},
   author={Trudinger, N. S.},
   title={Elliptic partial differential equations of second order},
   series={Classics in Mathematics},
   note={Reprint of the 1998 edition},
   publisher={Springer-Verlag, Berlin},
   date={2001},
}

\bib{GK14}{article}{
author={Gong, X.},
author={Kim, K.},
title={The $\overline{\partial}$-equation on variable strictly pseudoconvex domains},
journal={Math. Z.},
note={to appear},
}

\bib{GW10}
{article}{
   author={Gong, X.},
   author={Webster, S. M.},
   title={Regularity for the CR vector bundle problem I},
   journal={Pure Appl. Math. Q.},
   volume={6},
   date={2010},
   number={4, Special Issue: In honor of Joseph J. Kohn.},
   pages={983--998},
}


\bib{GW11}
{article}{
   author={Gong, X.},
   author={Webster, S. M.},
   title={Regularity for the CR vector bundle problem II},
   journal={Ann. Sc. Norm. Super. Pisa Cl. Sci. (5)},
   volume={10},
   date={2011},
   number={1},
   pages={129--191},
}

\bib{GW12}
{article}{
   author={Gong, X.},
   author={Webster, S. M.},
   title={Regularity in the local CR embedding problem},
   journal={J. Geom. Anal.},
   volume={22},
   date={2012},
   number={1},
   pages={261--293},
}




\bib{Gu62}
{article}{
   author={Guggenheimer, H.},
   title={A simple proof of Frobenius's integration theorem},
   journal={Proc. Amer. Math. Soc.},
   volume={13},
   date={1962},
   pages={24--28},
}

	




\bib{HJ97}{article}{
   author={Hanges, N.},
   author={Jacobowitz, H.},
   title={The Euclidean elliptic complex},
   journal={Indiana Univ. Math. J.},
   volume={46},
   date={1997},
   number={3},
   pages={753--770},
}
\bib{Ha13}{book}{
   author={Hawkins, T.},
   title={The mathematics of Frobenius in context},
   series={Sources and Studies in the History of Mathematics and Physical
   Sciences},
   note={A journey through 18th to 20th century mathematics},
   publisher={Springer, New York},
   date={2013},
}


%


\bib{HT03
}{article}{
   author={Hill, C. D.},
   author={Taylor, M.},
   title={Integrability of rough almost complex structures},
   journal={J. Geom. Anal.},
   volume={13},
   date={2003},
   number={1},
   pages={163--172},
}

\bib{HT07}
{article}{
   author={Hill, C. D.},
   author={Taylor, M.},
   title={The complex Frobenius theorem for rough involutive structures},
   journal={Trans. Amer. Math. Soc.},
   volume={359},
   date={2007},
   number={1},
   pages={293--322
   },
   issn={0002-9947},
}


\bib{Ho65}
{article}{
   author={H{\"o}rmander, L.},
   title={The Frobenius-Nirenberg theorem},
   journal={Ark. Mat.},
   volume={5},
   date={1965},
   pages={425--432 (1965)},
}




\bib{Ho76}
{article}{
   author={H{\"o}rmander, L.},
   title={The boundary problems of physical geodesy},
   journal={Arch. Rational Mech. Anal.},
   volume={62},
   date={1976},
   number={1},
   pages={1--52},
}






\bib{KS61}
{article}{
   author={Kodaira, K.},
   author={Spencer, D. C.},
   title={Multifoliate structures},
   journal={Ann. of Math. (2)},
   volume={74},
   date={1961},
   pages={52--100},
}


\bib{Ko63}{article}{
   author={Kohn, J.J.},
   title={Harmonic integrals on strongly pseudo-convex manifolds. I},
   journal={Ann. of Math. (2)},
   volume={78},
   date={1963},
   pages={112--148},
}






\bib{MM94}{article}{
   author={Ma, L.},
   author={Michel, J.},
   title={Regularity of local embeddings of strictly pseudoconvex CR
   structures},
   journal={J. Reine Angew. Math.},
   volume={447},
   date={1994},
   pages={147--164},
}

\bib{Ma69}
{article}{
   author={Malgrange, B.},
   title={Sur l'int\'egrabilit\'e des structures presque-complexes},
   conference={
      title={Symposia Mathematica, Vol. II},
      address={INDAM, Rome},
      date={1968},
   },
   book={
      publisher={Academic Press, London},
   },
   date={1969},
   pages={289--296},
}


\bib{Mo62}
{article}{
   author={Moser, J.},
   title={On invariant curves of area-preserving mappings of an annulus},
   journal={Nachr. Akad. Wiss. G\"ottingen Math.-Phys. Kl. II},
   volume={1962},
   date={1962},
   pages={1--20},
}


\bib{Na73}
{book}{
   author={Narasimhan, R.},
   title={Analysis on real and complex manifolds},
   series={North-Holland Mathematical Library},
   volume={35},
   note={Reprint of the 1973 edition},
   publisher={North-Holland Publishing Co., Amsterdam},
   date={1985},
   pages={xiv+246},
}


\bib{NN73}
{book}{
   author={Nevanlinna, F.},
   author={Nevanlinna, R.},
   title={Absolute analysis},
   note={Translated from the German by Phillip Emig;
   Die Grundlehren der mathematischen Wissenschaften, Band 102},
   publisher={Springer-Verlag, New York-Heidelberg},
   date={1973},
}


\bib{NN57}
{article}{
   author={Newlander, A.},
   author={Nirenberg, L.},
   title={Complex analytic coordinates in almost complex manifolds},
   journal={Ann. of Math. (2)},
   volume={65},
   date={1957},
   pages={391--404},
}



\bib{NW63}
{article}{
   author={Nijenhuis, A.},
   author={Woolf, W. B.},
   title={Some integration problems in almost-complex and complex manifolds},
   journal={Ann. of Math. (2)},
   volume={77},
   date={1963},
   pages={424--489},
}


\bib{Ni57}
{article}{
   author={Nirenberg, L.},
   title={A complex Frobenius theorem},
   journal={ Seminars on Analytic Functions I},
   date={1957},
note={Institute for Advanced Study, Princeton},
page={172-189},
}









\bib{Se64}{article}{
   author={Seeley, R. T.},
   title={Extension of $C^{\infty }$ functions defined in a half space},
   journal={Proc. Amer. Math. Soc.},
   volume={15},
   date={1964},
   pages={625--626},
}



\bib{Tr92}
{book}{
   author={Tr{e}ves, F.},
   title={Hypo-analytic structures, Local theory},
   series={Princeton Mathematical Series},
   volume={40},
   publisher={Princeton University Press, Princeton, NJ},
   date={1992},
}

\bib{We89}{article}{
   author={Webster, S. M.},
   title={A new proof of the Newlander-Nirenberg theorem},
   journal={Math. Z.},
   volume={201},
   date={1989},
   number={3},
   pages={303--316},
}

\bib{We89a}{article}{
   author={Webster, S. M.},
   title={On the proof of Kuranishi's embedding theorem},
   journal={Ann. Inst. H. Poincar\'e Anal. Non Lin\'eaire},
   volume={6},
   date={1989},
   number={3},
   pages={183--207},
}


\bib{We91}
{article}{
   author={Webster, S. M.},
   title={The integrability problem for CR vector bundles},
   conference={
      title={Several complex variables and complex geometry, Part 3 (Santa
      Cruz, CA, 1989)},
   },
   book={
      series={Proc. Sympos. Pure Math.},
      volume={52},
      publisher={Amer. Math. Soc., Providence, RI},
   },
   date={1991},
   pages={355--368},
}

\end{biblist}
\end{bibdiv}
\end{document}